\documentclass[11pt]{article}
\usepackage[top=2.54cm, bottom=2.54cm, left=2.5cm, right=2.5cm]{geometry}
\usepackage[tbtags]{amsmath}
\usepackage{amssymb}
\usepackage{amsthm}
\usepackage{appendix}
\usepackage{fancyhdr}
\usepackage{latexsym}
\usepackage{mathrsfs}
\usepackage{xcolor}
\usepackage{wasysym}
\usepackage{fancyhdr}
\usepackage[colorlinks=true, linkcolor=blue, citecolor=red, urlcolor=magenta]{hyperref}
\usepackage{float}
\usepackage{graphicx}
\usepackage[numbers,sort&compress]{natbib}
\usepackage{pict2e,color,xcolor}
\usepackage{tikz}
\usepackage{tikz-3dplot}
\usepackage{tkz-graph}
\usetikzlibrary{arrows,shapes,chains}
\usepackage{enumerate}
\usepackage {verbatim}

\allowdisplaybreaks[4]

\newtheorem{theorem}{\bf Theorem}[section]
\newtheorem{lemma}[theorem]{Lemma}

\newtheorem{corollary}[theorem]{Corollary}

\theoremstyle{definition}
\newtheorem{definition}[theorem]{Definition}
\newtheorem{proposition}[theorem]{Proposition}

\usepackage{nomencl}

\newcommand\Bx{\boldsymbol{x}}
\newcommand\By{\boldsymbol{y}}
\newcommand\Bz{\boldsymbol{z}}
\newcommand\Bv{\boldsymbol{v}}
\newcommand\Be{\boldsymbol{e}}

\makenomenclature

\title{An exponential improvement for Ramsey lower bounds}
\author{Jie Ma\footnote{School of Mathematical Sciences, University of Science and Technology of China, Hefei, Anhui 230026, and Yau Mathematical Sciences Center, Tsinghua University, Beijing 100084, China.}~~~~~~
Wujie Shen\footnote{Department of Mathematics and Yau Mathematical Sciences Center, Tsinghua University,  Beijing 100084, China.}~~~~~~
Shengjie Xie\footnote{School of Mathematical Sciences, University of Science and Technology of China, Hefei, Anhui 230026, China.
}
}

\date{\today}

\begin{document}

\maketitle

\begin{abstract}
We prove a new lower bound on the Ramsey number $r(\ell, C\ell)$ for any constant $C > 1$ and sufficiently large $\ell$, 
showing that there exists $\varepsilon=\varepsilon(C)> 0$ such that
\[
r(\ell, C\ell) \geq \left(p_C^{-1/2} + \varepsilon\right)^\ell,
\]
where $p_C \in (0, 1/2)$ is the unique solution to $C = \frac{\log p_C}{\log(1 - p_C)}$. 
This provides the first exponential improvement over the classical lower bound obtained by Erd\H{o}s in 1947.
\end{abstract}


\section{Introduction}

The \emph{Ramsey number} \( r(\ell, k) \) denotes the smallest positive integer \( n \) such that every red-blue edge coloring of the complete graph \( K_n \) on \( n \) vertices contains either a red clique \( K_\ell \) (a complete subgraph on \( \ell \) vertices with all edges red) or a blue clique \( K_k \) (a complete subgraph on \( k \) vertices with all edges blue).  
In 1930, Ramsey~\cite{Ramsey} proved that \( r(\ell, k) \) is finite for all \( \ell, k \in \mathbb{N} \).  
Since then, understanding the growth rate of \( r(\ell, k) \) has been a central problem in combinatorics for nearly a century (see the comprehensive survey \cite{CFS}).
The asymptotic study of Ramsey numbers naturally splits into two regimes:  
the case where \( \ell \) is fixed and \( k \to \infty \)  
(see~\cite{AKSz,BK10,BK21,CJMS,PGM,Kim,MV24,S83,Spencer1975,Spencer1977} and references therein),  
and the case where both \( \ell, k \to \infty \).  
In this paper, we investigate the Ramsey number \( r(\ell, k) \) in the latter regime, where both parameters $\ell$ and $k$ tend to infinity.

In 1935, Erd\H{o}s and Szekeres~\cite{EG} established the first non-trivial upper bound for Ramsey numbers, proving that \( r(\ell,k) \leq \binom{k+\ell-2}{\ell-1} \), and in particular that the diagonal Ramsey number satisfies \( r(\ell, \ell) \leq 4^\ell \).
This upper bound remained unchanged until the work of R\"odl (unpublished) and of Graham and R\"odl~\cite{GR} in the 1980s. 
Thomason~\cite{T1988} was the first to improve it by a polynomial factor in $\ell$ when $k$ and $\ell$ are of the same order.
Conlon~\cite{Conlon2009} achieved a landmark result by extending Thomason's quasi-randomness method, yielding a superpolynomial improvement when $k$ and $\ell$ are comparable in size.
This approach was later refined and optimized by Sah~\cite{Sah2023}, who showed that for any $\delta\in (0,1/2)$, there exists $c_\delta >0$ such that 
\(
r(\ell,k)\leq e^{-c_\delta (\log \ell)^2} \binom{k+\ell}{\ell}
\)
holds whenever $\ell/k\in [\delta,1]$ and $\ell\geq 1/c_{\delta}$.
A major breakthrough was achieved in 2023 by Campos, Griffiths, Morris and Sahasrabudhe \cite{CGMS}, who proved that
\begin{equation*}\label{equ:CGMS}
    r(\ell,k)\leq e^{-\ell/400+o(k)} \binom{k+\ell}{\ell}
\end{equation*}
holds for all integers $\ell\leq k$,
making the first exponential improvement over the Erd\H{o}s-Szekeres bound.
In particular, this implies the existence of a constant $\varepsilon>0$ such that $r(\ell,\ell)\leq (4-\varepsilon)^\ell$. 
Building on a reinterpretation of this method, 
Gupta, Ndiaye, Norin and Wei \cite{GNNW} improved both bounds, showing that
\(
r(\ell,k)\leq e^{-\ell/20+o(k)} \binom{k+\ell}{\ell}
\)
holds for all integers $\ell\leq k$, and in particular $r(\ell,\ell)\leq 3.8^{\ell+o(\ell)}$. 
More recently, Balister, Bollobás, Campos, Griffiths, Hurley, Morris, Sahasrabudhe and Tiba \cite{BBCGHMST} provided a different proof of these bounds, which also yields an exponential improvement in the upper bound for multicolor Ramsey numbers.

For the lower bound on the Ramsey number \( r(\ell, k) \) in the regime where both \(\ell\) and \(k\) tend to infinity, our knowledge remains essentially limited to the classical result of Erd\H{o}s~\cite{E1947} from 1947. In this seminal work, Erd\H{o}s introduced a probabilistic argument that laid the foundation for the probabilistic method in combinatorics (see~\cite{AS}). He established the first exponential lower bound on \( r(\ell, k) \) in the regime where \(k\) and \(\ell\) are of comparable size; specifically, for any fixed \( C \ge 1 \), 
\begin{equation}\label{equ:original_bound}
r(\ell, C\ell) = \Omega\bigl(\ell \cdot M_C^{\ell}\bigr) \text{~~~as~~~} \ell \to \infty,\footnote{A proof is given in Appendix~\ref{app:Erd-pro}, establishing the optimality of this bound via Erd\H{o}s’s first moment method.}
\end{equation}
where \( M_C := p_C^{-1/2} \) and \(p_C \in (0, 1/2]\) denotes the unique solution to \(C = \frac{\log p_C}{\log(1 - p_C)}.\)
Over the 78 years since Erd\H{o}s's proof, the only improvement has been a constant-factor refinement of the original bound, 
due to Spencer~\cite{Spencer1975} in 1975, via an application of the Lovász Local Lemma.

In this paper, we introduce a model called {\it the random sphere graph} and use it to obtain an exponential improvement over the classical lower bound~\eqref{equ:original_bound} for the Ramsey number \( r(\ell, C\ell) \), 
where \( C > 1 \) is any fixed constant and \( \ell \to \infty \).  
Our main result is stated below.

\begin{theorem}\label{thm1}
For any constant \( C > 1 \), there exist \( \varepsilon = \varepsilon(C) > 0 \) and $\ell_0=\ell_0(C)>0$ such that for all sufficiently large integers \( \ell\geq \ell_0(C)\),
\begin{equation}\label{equ:improved_bound}
r(\ell, C\ell) \geq (M_C + \varepsilon)^\ell,
\end{equation}
where \( M_C = p_C^{-1/2} \) and \( p_C \in (0, 1/2) \) satisfies \( C = \frac{\log p_C}{\log(1 - p_C)} \).
\end{theorem}

Using this result, we can immediately obtain the following corollary, 
which gives an exponential improvement on the Ramsey number \( r(\ell, k) \) in the general regime where $\ell$ and $k$ are of comparable size. 
For any $\delta\in (0,1/2)$, let $c_\delta>0$ be such that
\[
c_\delta=\min_{C\in [1/(1-\delta),1/\delta]} \left\{\frac{\varepsilon(C)}{2M_C}, \frac{1}{\ell_0(C)}\right\}.\footnote{Note that \(M_C\) is continuous in \(C\), \(\varepsilon(C)\) is also continuous in \(C\), and \(\ell_0(C)\) is increasing in \(C\) (as can be seen from the proof); therefore, the minimum is attained.}
\]
For all $\ell,k\in \mathbb{N}$, let \( \mathsf{Er}(\ell,k) \) denote the lower bound obtained by  Erd\H{o}s’s probabilistic argument~\cite{E1947}.

\begin{corollary}\label{coro}
For any $\delta\in (0,1/2)$, we have
\[
r(\ell,k)\geq \left(1+2c_{\delta}\right)^\ell\cdot (M_{k/\ell})^{\ell}\geq (1+c_\delta)^\ell \cdot \mathsf{Er}(\ell,k)
\]
whenever $\delta\leq \ell/k \leq 1-\delta$ and $\ell\gg 1/c_\delta$.
\end{corollary}

We remark that our proof also yields an improvement in the {\it almost} diagonal range, namely when \(k=\ell+o(\ell)\).
To be precise, for any function \(f(\ell)\) with \(\sqrt{\ell}\ll f(\ell)\ll \ell\), we can show (see the discussion at the end of Section~\ref{sec:final-proof}) that
\begin{equation}\label{equ:almost-diag}
r(\ell, \ell+f(\ell)) \;\geq\; e^{\Omega\left(f^2(\ell)/\ell\right)} \cdot \mathsf{Er}(\ell,\ell+f(\ell)).
\end{equation}

We also note that the problem of finding explicit constructions of Ramsey numbers constitutes an important and highly challenging line of research; see the classical constructions \cite{F77,FW81}, the recent advances \cite{BRSW,CZ19,L23}, and the references therein.

The rest of the paper is organized as follows. 
In Section~\ref{sec:random-sphere}, we introduce our random graph model, related notations, and discuss several basic geometric properties of this model. 
In Section~\ref{sec:redu-skete}, we prove Theorem~\ref{thm1} by reducing it to Theorem~\ref{thm2} and provide a sketch of Theorem~\ref{thm2}’s proof.
Section~\ref{sec:auxi-lemmas} collects several auxiliary lemmas.
In Section~\ref{Perfect sequences}, we introduce the crucial concept of \emph{perfect sequences} for unit vectors. 
Section~\ref{sec:estimation_q} provides preliminary estimates on perfect sequences, while Section~\ref{sec:from-perfect-to-general} shows that perfect sequences capture the essential behavior of our problem.
Section~\ref{sec:key-quan} contains the core technical arguments, where we derive estimates on key quantities. 
Finally, in Section~\ref{sec:final-proof}, we assemble all these estimates to complete the proof of Theorem~\ref{thm2}.
Throughout, for any constant $C>1$, we define $p_C\in (0,1/2)$ to be unique solution to \( C = \frac{\log p_C}{\log(1 - p_C)} \).
We denote \([r] := \{1, 2, \ldots, r\}\), and unless otherwise specified, all logarithms are base~$e$.
For a vector \(\Bx \in \mathbb{R}^n\), we define its standard Euclidean norm by \(|\Bx| = \sqrt{\langle \Bx, \Bx \rangle}\).

\section{The Random Sphere Graph  \( G_{k,p}(n) \)}\label{sec:random-sphere}
In this section, we introduce a random graph model based on geometric measure, 
which serves as the foundation for our Ramsey construction. 

Throughout this paper, let \( S^k \) denote the  \( k \)-dimensional unit sphere embedded in the \( (k+1)\)-dimensional Euclidean space \( \mathbb{R}^{k+1} \).
Let \( \operatorname{Vol}(\cdot) \) denote the standard surface measure on \( S^k \),
and let \( \Bx \) be a point sampled uniformly at random from \( S^k \).
For any Borel set \( A\subseteq S^k \), we define the {\it probability \( \mathbb{P}(A) \)} of $A$ to be the probability that $\Bx$ belongs to $A$, that is, 
\begin{equation}\label{equ:P(A)}
\mathbb{P}(A) = \mathbb{P}(\Bx\in A) = \frac{\operatorname{Vol}(A)}{\operatorname{Vol}(S^k)}.
\end{equation}
Let $\Be$ denote an arbitrary but fixed point on \( S^k \). 
For any \( p \in (0,\frac12] \), 
define $c_{k,p}\geq 0$ as the unique constant satisfying
\begin{equation}\label{equ:c_pk}
\mathbb{P}\left( \left\langle \Bx, \Be \right\rangle \leq -\frac{c_{k,p}}{\sqrt{k}} \right) = p,
\end{equation}
where \( \left\langle \Bx, \Be \right\rangle \) denotes the standard inner product between the vectors \( \Bx \) and \( \Be \).\footnote{We often treat a point on \( S^k \) and a unit vector interchangeably without distinction.}

\begin{definition}[\bf{Random Sphere Graphs}]\label{def:sphere_random_graph}
Let \( n, k \in \mathbb{N} \) and \( p \in (0, \tfrac{1}{2}] \). 
The \emph{random sphere graph} \( G_{k,p}(n) \) is defined as a complete graph on 
\( n \) vertices \( \Bx_1, \Bx_2, \ldots, \Bx_n \), 
equipped with a red-blue coloring on its edges, constructed as follows: (see Figure~\ref{fig: view of S1} for an illustration)
\begin{itemize}
    \item The vertices \( \Bx_1, \Bx_2, \ldots, \Bx_n \) 
    are sampled independently and uniformly from \( S^k \);
    \item Each edge \( \Bx_i \Bx_j \) is colored \emph{red} if 
    \( \langle \Bx_i, \Bx_j \rangle \leq -\frac{c_{k,p}}{\sqrt{k}} \), 
    and \emph{blue} otherwise.
\end{itemize}
\end{definition}

\medskip

It is evident that each edge of \( G_{k,p}(n) \) is colored red with probability \( p \) and blue with probability \( 1 - p \). 
However, due to intrinsic geometric constraints, 
the events for coloring edges may not be independent. 
This marks an essential difference from the Erd\H{o}s--R\'enyi random graph model.
In the following subsections, we formalize the notation underlying this random model and highlight several elementary geometric properties that motivate our analysis later.

\begin{figure}[H]
  \centering
  \includegraphics[width=1\textwidth]{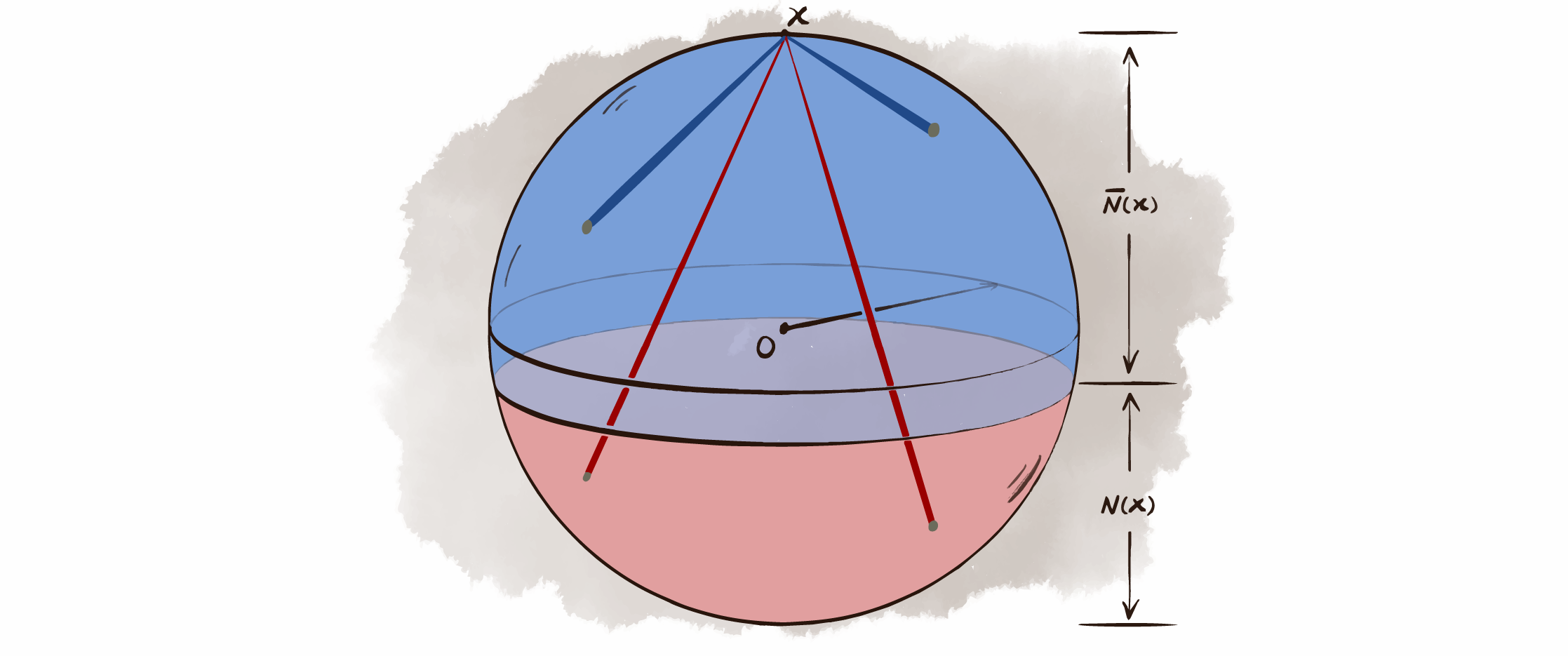}
  \caption{The random sphere graph.}
  \label{fig: view of S1}
  \end{figure}

Before proceeding, we offer a few remarks for historical context.
A relevant random graph model defined in Euclidean spaces, known as \emph{random geometric graphs}, has been extensively studied in probability theory and theoretical computer science (see, for example, \cite{DGLU} and the monograph~\cite{MP03}). 
Combinatorial constructions on the sphere date back almost sixty years.
In 1966, Erd\H{o}s \cite{E1966} employed a geometric graph on the Boolean cube (which is a subset of the sphere) to obtain an explicit construction for the off-diagonal Ramsey number \(r(3,k)\). 
A decade later, Bollob\'as and Erd\H{o}s introduced a random graph on the sphere in their celebrated paper  \cite{BE76} (see also \cite{FLZ}).
Their construction can be interpreted as the following random red/blue edge-coloring:
for two random points $\Bx, \By\in S^k$,
the edge is colored red if and only if \( \langle \Bx, \By \rangle \geq \frac{c_{k,p}}{\sqrt{k}} \). 
Thus, each edge is red with probability $p$, just as in our random sphere graph \( G_{k,p}(n) \).
However, we emphasize that the Bollob\'as-Erd\H{o}s model differs from the random sphere graph considered here.
Indeed, applying the same explanations as in Subsection~\ref{subsec:geometric_dependency} shows that, for their model, the two inequalities in \eqref{equ:geo-dependency} are reversed, 
making it inapplicable for the Ramsey construction.\footnote{Essentially, this is because the corresponding \eqref{equ:sum-two-base} for the Bollob\'as-Erd\H{o}s model would be strictly greater than 1.}
To the best of our knowledge, our work is the first to use this type of random model to investigate Ramsey numbers \(r(\ell, k)\) in the regime where \(\ell\) and \(k\) are of comparable size.

\subsection{Notation}

Let \( k, r \in \mathbb{N} \) and fix \( p \in (0, \tfrac{1}{2}] \). 
Let \( \Bx[r] := (\Bx_1, \Bx_2, \ldots, \Bx_r) \) be a sequence of points on \( S^k \).  
The graph \( G_p(\Bx[r]) \) defines the complete graph with vertex set \( \{\Bx_1, \Bx_2, \ldots, \Bx_r\} \), whose edges are colored red or blue according to the same rule used in the definition of \( G_{k,p}(n) \).  
We refer to \( G_p(\Bx[r]) \) as the \emph{induced subgraph} on \( \Bx[r] \).
Define the space \( (S^k)^r = \{ (\boldsymbol{x}_1, \ldots, \boldsymbol{x}_r) : \boldsymbol{x}_i \in S^k \text{ for all } 1 \leq i \leq r \} \).
A {\it random \( r \)-tuple} \( (\boldsymbol{x}_1, \ldots, \boldsymbol{x}_r )\) in \( (S^k)^r \) defines a tuple where each \( \boldsymbol{x}_i \) is sampled independently and uniformly from \( S^k \).
For any Borel set \( A \subseteq (S^k)^r\), we define \( \mathbb{P}(A) \) as the probability that a random \( r \)-tuple belongs to \( A \).
When \( r = 1 \), this coincides with the normalized surface measure as in \eqref{equ:P(A)}.

The following probabilities that a random \( r \)-tuple forms a monochromatic clique will play a central role in our analysis.

\begin{definition}\label{def:prod_redblue}
Let \( k, r \in \mathbb{N} \) and fix \( p \in (0, \tfrac{1}{2}] \). 
Let \( \Bx{[r]} = (\Bx_1, \Bx_2, \ldots, \Bx_r) \) be a random \( r \)-tuple in \( (S^k)^r \).  
We define \( P_{\mathrm{red}, r} \) as the probability that the induced subgraph \( G_p(\Bx{[r]}) \) forms a red clique, and \( \overline{P}_{\mathrm{blue}, r} \) as the probability that \( G_p(\Bx{[r]}) \) forms a blue clique.
\end{definition}

We also define neighborhoods and their associated probability measures for points and \( r\)-tuples.

\begin{definition}\label{def:sphere_neighborhood}
Let \(\Bx_1, \ldots, \Bx_r, \By \in S^k \) be given. 
Define \(\Bx[r]= (\Bx_1, \Bx_2, \ldots, \Bx_r)\).
\begin{itemize}
    \item[(1).] The \textit{red-neighborhood} \( N(\By) \) and the \textit{blue-neighborhood} \( \overline{N}(\By) \) of \( \By \) are defined as (see Figure~\ref{fig: view of S1})
    \[
    N(\By) := \left\{ \Bz \in S^k : \langle \By, \Bz \rangle \leq -\frac{c_{k,p}}{\sqrt{k}} \right\} \quad \text{and} \quad \overline{N}(\By) := \left\{ \Bz\in S^k : \langle \By, \Bz \rangle > -\frac{c_{k,p}}{\sqrt{k}} \right\},
    \]
    and their probability measures are given by
    \[
    P(\By) := \mathbb{P}(N(\By))\quad  \mbox{and}  \quad \overline{P}(\By) := \mathbb{P}(\overline{N}(\By)).
    \]
    
    \item[(2).] The \textit{red-neighborhood} \( N(\Bx[r]) \) and the \textit{blue-neighborhood} \( \overline{N}(\Bx[r]) \) of \( \Bx[r] \) are defined as (see Figure~\ref{fig2})
    \[
    N(\Bx[r]) := N(\Bx_1) \cap \cdots \cap N(\Bx_r) \quad \text{and} \quad \overline{N}(\Bx[r]) := \overline{N}(\Bx_1) \cap \cdots \cap \overline{N}(\Bx_r),
    \]
    and their probability measures are given by
    \begin{align*}
        &P(\Bx_1, \ldots, \Bx_r) = P(\Bx[r]) := \mathbb{P}(N(\Bx[r])) = \mathbb{P}(N(\Bx_1) \cap \cdots \cap N(\Bx_r)), \\
        &\overline{P}(\Bx_1, \ldots, \Bx_r) = \overline{P}(\Bx[r]) := \mathbb{P}(\overline{N}(\Bx[r])) = \mathbb{P}(\overline{N}(\Bx_1) \cap \cdots \cap \overline{N}(\Bx_r)).
    \end{align*}
\end{itemize}
\end{definition}

Loosely speaking, the average probability \( P_{\mathrm{red}, r} \) (from Definition~\ref{def:prod_redblue}) can be approximated by a product of  probabilities \( P(\boldsymbol{x}[s]) \) for \( 1 \leq s \leq r \) (from Definition~\ref{def:sphere_neighborhood}); 
see equations~\eqref{equ:decomp-of-kappas} and \eqref{equ:kappa-r} for the precise expression.
A similar relation holds for \( \overline{P}_{\mathrm{blue}, r} \), with the corresponding terms \( \overline{P}(\boldsymbol{x}[s]) \).

\begin{figure}[H]
  \centering
  \includegraphics[width=1.0\textwidth]{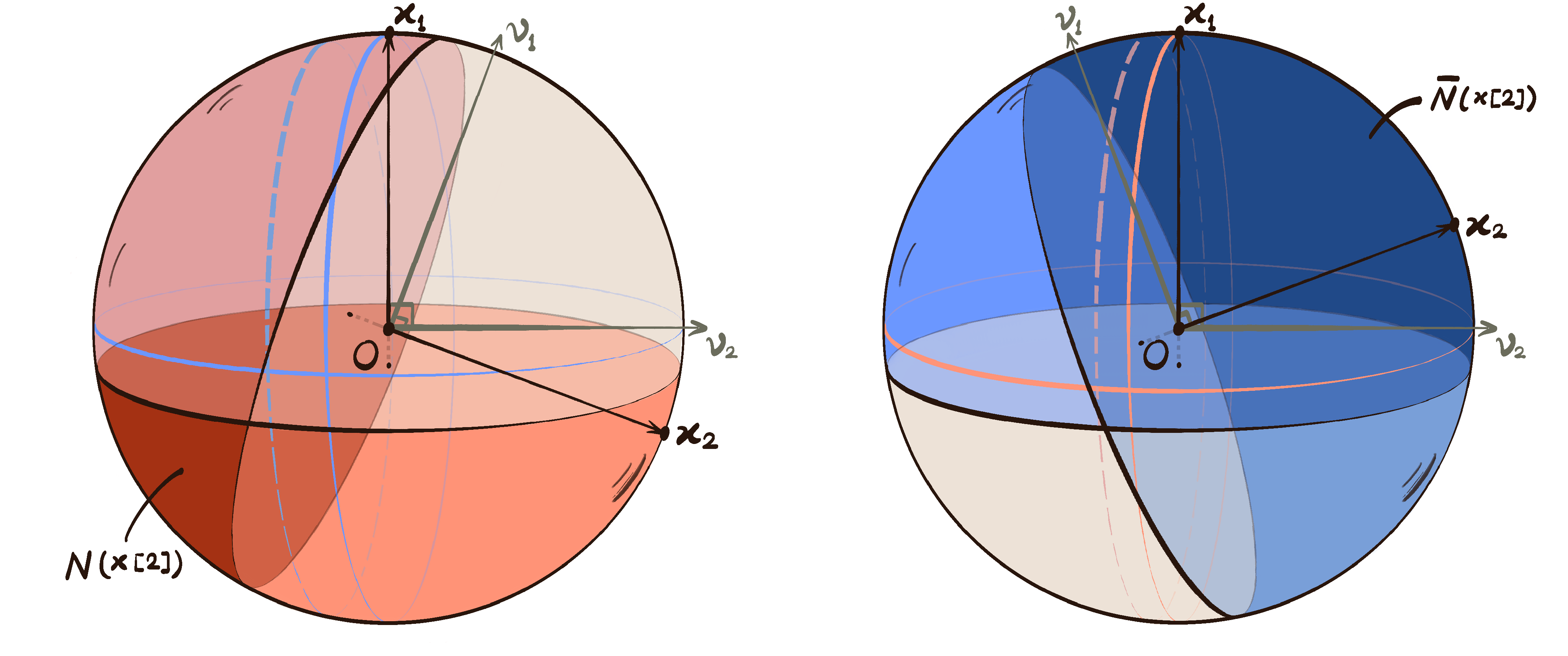}
  \caption{The red-neighborhood $N(\boldsymbol{x}[r])$ and the blue-neighborhood $\overline{N}(\boldsymbol{x}[r])$ for $r=2$.}
  \label{fig2}
  \end{figure}

\begin{figure}[H]
\centering
\includegraphics[width=1.0\textwidth]{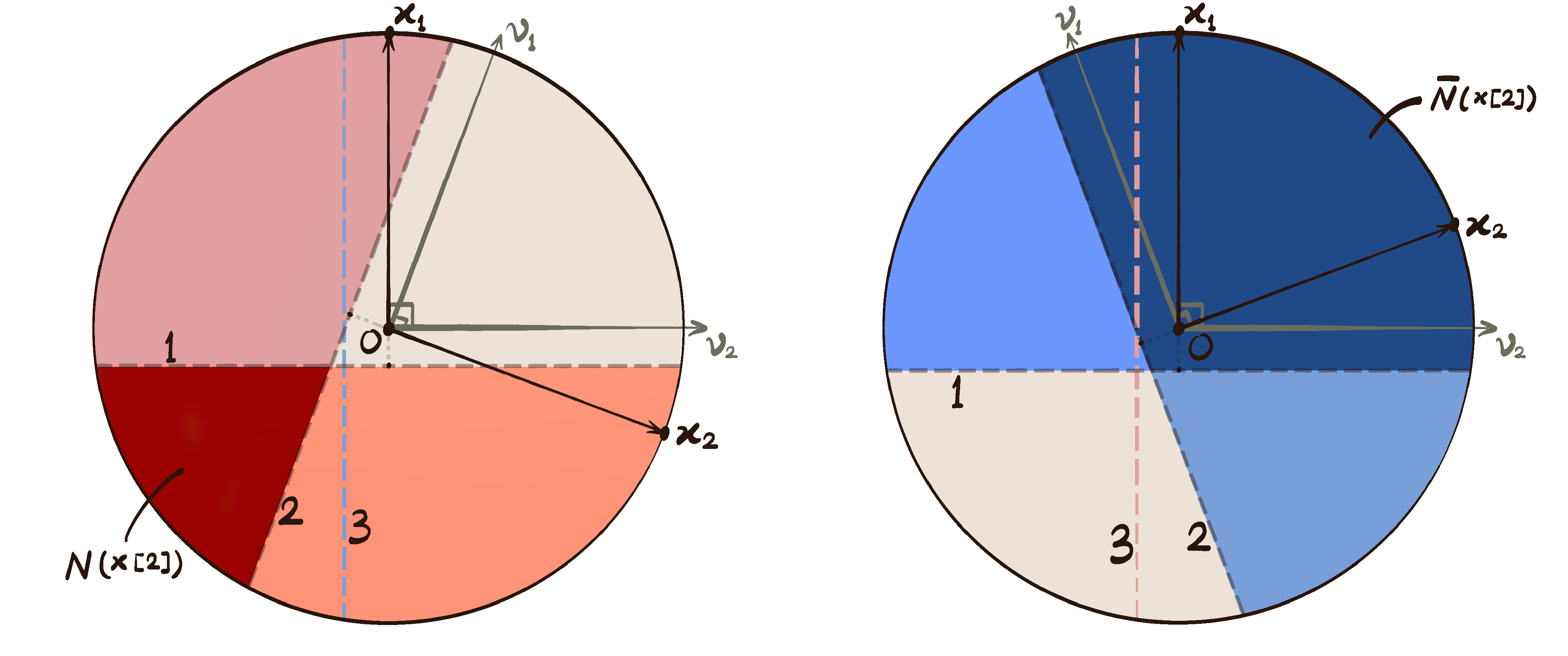}
\caption{Orthogonal projection of neighborhoods $N(\boldsymbol{x}[2])$ and $\overline{N}(\boldsymbol{x}[2])$ onto \( \operatorname{span}(\Bx_1, \Bx_2) \)}
\label{fig3}
\end{figure}

Many of our proofs rely on estimates involving orthogonal projections of vectors. 
Below, we summarize the definition of orthogonal projection used throughout this paper.

\begin{definition}\label{def:projection}
Let \( Y \subseteq \mathbb{R}^{k+1} \) be a linear subspace.  
The mapping \( \pi_Y: \mathbb{R}^{k+1} \to Y \) assigns to each \( \By \in \mathbb{R}^{k+1} \) its orthogonal projection onto \( Y \).  
For given vectors \( \Bx_1, \ldots, \Bx_r \in \mathbb{R}^{k+1} \),  
we write \( \pi_{[r]} \) for the special mapping \( \pi_Y \) where \( Y = \operatorname{span}(\Bx_1, \ldots, \Bx_r) \), that is,  
\[
\pi_{[r]}(\By) := \pi_{\operatorname{span}(\Bx_1, \ldots, \Bx_r)}(\By).
\]  
We often denote the projected image by \(\tilde{\boldsymbol{y}} := \pi_{[r]}(\boldsymbol{y})\) if there is no ambiguity.
\end{definition}

We refer the reader to Figure~\ref{fig3}, which illustrates the orthogonal projection of the neighborhoods of \( \Bx[2] = (\Bx_1, \Bx_2) \) from Figure~\ref{fig2} onto the subspace spanned by \( \Bx_1 \) and \( \Bx_2 \).

\subsection{Geometric Dependencies in \( G_{k,p}(n) \)}\label{subsec:geometric_dependency}

The geometric structure of random sphere graphs \( G_{k,p}(n) \) gives rise to essential mathematical features that differ fundamentally from those of classical random graph models, such as the Erdős–Rényi model \( \mathsf{ER}(n, p) \).  
Recall that \( \mathsf{ER}(n, p) \) denotes a random red-blue coloring of the edges of the complete graph \( K_n \), where each edge is independently colored red with probability \( p \) and blue with probability \( 1 - p \).
Analogous to Definition~\ref{def:prod_redblue}, we let \( P^{\mathsf{ER}}_{\mathrm{red}, r} \) denote the probability that any given set of \( r \) vertices forms a red clique \( K_r \) in \( \mathsf{ER}(n, p) \), and \( \overline{P}^{\mathsf{ER}}_{\mathrm{blue}, r} \) the probability that they form a blue clique \( K_r \).
Since edge colorings in \( \mathsf{ER}(n, p) \) are independent, it follows that for any \( r \geq 2 \),
\[
P^{\mathsf{ER}}_{\mathrm{red}, r} = p^{\binom{r}{2}} \quad \text{and} \quad \overline{P}^{\mathsf{ER}}_{\mathrm{blue}, r} = (1 - p)^{\binom{r}{2}}.
\]

However, the situation changes entirely in \( G_{k,p}(n) \) for all \( r \geq 3 \), 
assuming $k$ is sufficiently large.  
To illustrate this, consider the case \( r = 3 \). 
Recall from Definition~\ref{def:prod_redblue} that \( P_{\mathrm{red}, 3} \) denotes the probability that a random triple \( (\Bx_1, \Bx_2, \Bx_3) \) in \( (S^k)^3\) forms a red triangle in \( G_{k,p}(n) \).
Using the notation from Definition~\ref{def:sphere_neighborhood}, we can express this as
\[
P_{\mathrm{red}, 3}=\mathbb{P}\left(\Bx_1\Bx_2 \mbox{ is red}) \cdot \mathbb{P}(\Bx_3\in N(\Bx[2]) \ |\  \Bx_1\Bx_2 \mbox{ is red}\right)=p\cdot \mathbb{P}\left(N(\Bx[2]) \ |\  \Bx_1\Bx_2 \mbox{ is red}\right).
\]
Suppose $\Bx_1\Bx_2$ is red. 
As shown in Figure~\ref{fig2}, the red-neighborhood \( N(\Bx[2]) \) is precisely the intersection of the two light-red spherical caps associated with \( N(\Bx_1) \) and \( N(\Bx_2) \). 
To build geometric intuition, we consider the projection of \( N(\Bx[2]) \) onto the plane \( \operatorname{span}(\Bx_1, \Bx_2) \).
In Figure~\ref{fig3}, the projection of \( N(\Bx[2]) \) appears as the dark-red region \( A_{12} \), bounded by lines 1 and 2 representing \( N(\Bx_1) \) and \( N(\Bx_2) \), respectively. 
If we rotate \( \Bx_2 \) so that it becomes orthogonal to \( \Bx_1 \), then line 2 becomes line 3, 
and the resulting region $A_{13}$ (bounded by lines 1 and 3) has normalized surface measure approximately $p^2$.  
It is clear from Figure~\ref{fig3} that \( A_{12}\) is properly contained in \( A_{13}\), and hence 
\( \mathbb{P}\left(N(\Bx[2])\right)<p^2 \) holds whenever $\Bx_1\Bx_2$ is red.
This implies that 
\[
P_{\mathrm{red}, 3}=p\cdot \mathbb{P}\left(N(\Bx[2]) \ |\  \Bx_1\Bx_2 \mbox{ is red}\right)<p^3.
\]
Analogously, Figure~\ref{fig3} suggests that
\(
\overline{P}_{\mathrm{blue}, 3}=(1-p)\cdot \mathbb{P}(\overline{N}(\Bx[2]) \ |\  \Bx_1\Bx_2 \mbox{ is blue}) > (1-p)^3.
\)\footnote{We point out that it is not necessarily true that \( \mathbb{P}\left(\overline{N}(\Bx[2])\right) > (1-p)^2 \) for every blue edge \( \Bx_1\Bx_2 \), but this inequality holds on average.}
One can naturally extend the above argument to conclude that for any \( r \geq 3 \),
\begin{equation}\label{equ:geo-dependency}
P_{\mathrm{red}, r} < p^{\binom{r}{2}} \quad \text{and} \quad \overline{P}_{\mathrm{blue}, r} > (1 - p)^{\binom{r}{2}}
\end{equation}
hold in the random sphere graph \( G_{k,p}(n) \).
These properties highlight geometric dependencies in \( G_{k,p}(n) \), which set it apart from the Erdős–Rényi model, leaving room for potential improvements.

To prove the main result of this paper, we undertake a detailed analysis to accurately estimate $P_{\mathrm{red}, r}$ and $\overline{P}_{\mathrm{blue}, r}$. 
A sketch of the argument is given in the next section.

\section{Proof of Main Theorem: Reduction and Sketch}\label{sec:redu-skete}
In this section, we reduce the validation of our main result, Theorem~\ref{thm1}, to the following theorem and then provide a sketch of its proof.
We consider the Ramsey number $r(\ell, C\ell)$ for any fixed constant $C > 1$. 
Recall that $p_C$ denotes the unique real number in $(0,\frac12)$ satisfying $C = \frac{\log p_C}{\log(1-p_C)}$. 

\begin{theorem}\label{thm2}
For any constant \( C > 1 \), there exists \( \varepsilon_0 = \varepsilon_0(C) > 0 \) such that the following holds.

Let \( D = D(C) \) and \( \ell_0 = \ell_0(C) \) be constants with \( \ell_0 \gg D \gg C \).\footnote{Throughout this paper, we use \( D \gg C \) to mean that \( D \) is sufficiently large relative to \( C \).}
Then for every \( \ell \geq \ell_0 \) and \( k = D^2 \ell^2 \), there exists \( p = p(C, \ell) \in (p_C, \tfrac{1}{2}) \) such that in the random sphere graph \( G_{k,p}(n) \),
\begin{equation}\label{equ:improved_P}
P_{\mathrm{red}, \ell} \leq \left(p_C - \varepsilon_0 \tfrac{\ell}{\sqrt{k}}\right)^{\binom{\ell}{2}} 
\quad \mbox{and} \quad
\overline{P}_{\mathrm{blue}, C\ell} \leq \left(1 - p_C - \varepsilon_0 \tfrac{\ell}{\sqrt{k}}\right)^{\binom{C\ell}{2}}.
\end{equation}
\end{theorem}

\subsection{Reduction to Theorem~\ref{thm2}}\label{subsec:reduction}
We now show that Theorem~\ref{thm1} follows directly from Theorem~\ref{thm2}.

\begin{proof}[\bf Proof of Theorem~\ref{thm1}, using Theorem~\ref{thm2}.]
Fix any constant $C>1$. 
Let \( D = D(C) \) and \( \ell_0 = \ell_0(C) \) be constants, and let $\ell$ be any integer satisfying \( \ell\geq \ell_0 \gg D \gg C \).
Let $k=D^2\ell^2$.
Then there exist \( \varepsilon_0 = \varepsilon_0(C) \) and \( p = p(C, \ell) \in (p_C, \tfrac{1}{2}) \) such that the conclusion of Theorem~\ref{thm2} holds.
Let
\begin{equation}\label{equ:ramsey-n}
\varepsilon = \frac{M_C^3\varepsilon_0}{6 D} \quad \mbox{and} \quad n=(M_C+\varepsilon)^\ell,
\end{equation}
where $M_C=p_C^{-1/2}=(1-p_C)^{-C/2}$. 
Then the probability \( P^*\) that there exists a red clique \( K_\ell \) or a blue clique \( K_{C\ell}\) in the random sphere graph \( G_{k,p}(n) \) satisfies
\begin{align*}
P^*\leq&\,  \binom{n}{\ell} P_{\mathrm{red}, \ell} + \binom{n}{C\ell} \overline{P}_{\mathrm{blue}, C\ell}\\
\leq&\, \frac{n^\ell}{2}\left(p_C-\varepsilon_0\frac{\ell}{\sqrt{k}}\right)^{\frac{\ell(\ell-1)}{2}}
+\frac{n^{C\ell}}{2}\left(1-p_C-\varepsilon_0\frac{\ell}{\sqrt{k}}\right)^{\frac{C\ell(C\ell-1)}{2}}\\
\leq&\, \frac{n^\ell}{2}\left(p_C-\frac{\varepsilon_0}{2}\frac{\ell}{\sqrt{k}}\right)^{\frac{\ell^2}{2}}
+\frac{n^{C\ell}}{2}\left(1-p_C-\frac{\varepsilon_0}{2}\frac{\ell}{\sqrt{k}}\right)^{\frac{C^2\ell^2}{2}}\\
=&\, \frac{M_C^{\ell^2}}{2}\cdot\left(1+\frac{\varepsilon}{M_C}\right)^{\ell^2}\cdot\left(1-\frac{\varepsilon_0}{2p_C D}\right)^{\frac{\ell^2}{2}}\cdot p_C^{\frac{\ell^2}{2}}\\
&+\frac{M_C^{C\ell^2}}{2}\cdot\left(1+\frac{\varepsilon}{M_C}\right)^{C\ell^2}\cdot\left(1-\frac{\varepsilon_0}{2(1-p_C)D}\right)^{\frac{C^2\ell^2}{2}}\cdot \left(1-p_C\right)^{\frac{C^2\ell^2}{2}}\\
=&\,\frac12 \left(1+\frac{\varepsilon}{M_C}\right)^{\ell^2}\cdot\left(1-\frac{\varepsilon_0}{2p_C D}\right)^{\frac{\ell^2}{2}}
+\frac12 \left(1+\frac{\varepsilon}{M_C}\right)^{C\ell^2}\cdot\left(1-\frac{\varepsilon_0}{2(1-p_C)D}\right)^{\frac{C^2\ell^2}{2}},
\end{align*}
where the first inequality holds by the union bound, 
the second inequality follows from Theorem~\ref{thm2},
and the third inequality holds for sufficiently large $\ell\geq \ell_0$. 
In particular, it suffices to choose $\ell_0 \gtrsim \frac{D}{\varepsilon_0}$, which may be larger than the constant $\ell_0(C)$ given in Theorem~\ref{thm2}. 

Since \(p_C\geq \frac{1}{C+1}\) for any \(C\geq 1\), we have
\[
\frac{3\varepsilon}{M_C}=\frac{\varepsilon_0}{2p_CD}\leq \frac{\varepsilon_0C}{2(1-p_C)D},
\] 
which implies
\[
1-\frac{\varepsilon_0}{2p_CD}= 1 - 3\frac{\varepsilon}{M_C}\leq \left(1+\frac{\varepsilon}{M_C} \right)^{-3}
\]
and
\[
1-\frac{\varepsilon_0}{2(1-p_C)D}\leq 1-\frac{3\varepsilon}{CM_C}\leq \left(1-\frac{\varepsilon}{CM_C}\right)^3\leq \left(1+\frac{\varepsilon}{CM_C}\right)^{-3}\leq \left(1+\frac{\varepsilon}{M_C} \right)^{-\frac{3}{C}}.
\]
This yields
\[
P^*\leq \frac12 \left(1+\frac{\varepsilon}{M_C}\right)^{-\frac{\ell^2}{2}}+\frac12 \left(1+\frac{\varepsilon}{M_C}\right)^{-\frac{C\ell^2}{2}}
< 1.
\]
Hence, there exists an instance of $G_{k,p}(n)$ that contains neither a red $K_\ell$ nor a blue $K_{C\ell}$.
This shows that $r(\ell, C\ell)> n=(M_C+\varepsilon)^\ell$, finishing the proof of Theorem~\ref{thm1}.
\end{proof}

\subsection{Proof Sketch of Theorem~\ref{thm2}}\label{sec:proof_sketch}
The proof of Theorem~\ref{thm2} essentially requires determining both \( P_{\mathrm{red}, \ell} \) and \( \overline{P}_{\mathrm{blue}, C\ell} \), the probabilities of forming monochromatic cliques in the random sphere graph \( G_{k,p}(n) \), up to second-order terms.

By symmetry, we focus on sketching the proof from the perspective of \( P_{\mathrm{red}, \ell} \).
Let $\Bx[\ell]=(\Bx_1,...,\Bx_\ell)$ be a random $\ell$-tuple in $(S^k)^\ell$. 
For every $r\in [\ell]$, we interpret $P_{\mathrm{red}, r}$ as the probability of the event $A_{\mathrm{red}, r}$ that $\Bx[r]=(\Bx_1,...,\Bx_r)$ induces a red clique $K_r$ in $G_{k,p}(n)$.
Then \( P_{\mathrm{red}, \ell} \) admits the following simple decomposition:
\begin{equation}\label{equ:decomp-of-kappas}
P_{\mathrm{red}, \ell}=\kappa_1 \cdots \kappa_{\ell-1}, \mbox{ where } \kappa_{r}=\frac{P_{\mathrm{red}, r+1}}{P_{\mathrm{red}, r}} \mbox{ for every } 1\leq r\leq \ell-1.
\end{equation}
Let $Y_r$ denote the event $\Bx_{r+1}\in N(\Bx[r])$, so \( A_{\mathrm{red}, r+1}=Y_r\wedge A_{\mathrm{red}, r}\).\footnote{Throughout this paper, the symbol \( \wedge \) denotes the joint occurrence of two or more events.}
We can rewrite $\kappa_{r}$ as
\begin{equation}\label{equ:kappa-r}
\kappa_{r}=\frac{P_{\mathrm{red}, r+1}}{P_{\mathrm{red}, r}}=\frac{\mathbb{P}\big(Y_r\wedge A_{\mathrm{red}, r}\big)}{\mathbb{P}(A_{\mathrm{red}, r})}
=\mathbb{P}\big(Y_r \ |\ A_{\mathrm{red}, r} \big)
=\mathbb{E}\big[\boldsymbol{1}_{Y_r} \ |\ A_{\mathrm{red}, r}\big]
=\mathbb{E}\big[P(\Bx[r]) \ |\ A_{\mathrm{red}, r}\big],
\end{equation}
where the last equality holds by the law of total expectation and the fact that $\mathbb{E}_{\Bx_{r+1}}\big[\boldsymbol{1}_{Y_r}\big]=\mathbb{P}(Y_r)=\mathbb{P}(N(\Bx[r]))=P(\Bx[r])$.
Similarly, we define 
\begin{equation}\label{equ:decomp-of-kappas-bar}
\overline{P}_{\mathrm{blue}, C\ell}=\overline{\kappa}_1 \cdots \overline{\kappa}_{C\ell-1}, \mbox{ where } \overline{\kappa}_{r}=\frac{\overline{P}_{\mathrm{blue}, r+1}}{\overline{P}_{\mathrm{blue}, r}} \mbox{ for every } 1\leq r\leq C\ell-1.
\end{equation}

\noindent{\bf Choice of the dimension $k$.}
Before we proceed to estimate the clique probabilities, we first discuss the choice of the dimension \( k \), which plays a critical role in the analysis. 
While its feasible value (i.e., $k\approx\ell^2$) will become clear from the precise expressions for clique probabilities later on, 
it is important to observe from the beginning that $k$ needs to be appropriately scaled in terms of $\ell$.
To motivate this, note that for two independent random vectors \( \boldsymbol{x}, \boldsymbol{y} \in S^k \), their typical inner product satisfies
\begin{equation}\label{equ:random-inner-product}
\big|\langle \boldsymbol{x}, \boldsymbol{y} \rangle\big| = \Theta\left(\frac{1}{\sqrt{k}}\right).
\end{equation}
If \( k \to \infty \), then such vectors become nearly orthogonal, 
and the random sphere graph \( G_{k,p}(n) \) approaches the Erdős–Rényi random graph \( \mathsf{ER}(n,p) \),
thereby losing the geometric structure essential to our analysis.
Conversely, if \( k \) is small (say $k\ll \ell^2)$, dependencies between edges become substantial, and the blue-clique probability $\overline{P}_{\mathrm{blue}, C\ell}$ deviates significantly from its expected value.
To see this, observe that from \eqref{equ:random-inner-product} and the discussion in Subsection~\ref{subsec:geometric_dependency}, one can derive that
\begin{equation}\label{equ:x[r]/x[r-1]}
\overline{P}(\boldsymbol{x}[r])
\approx \overline{P}(\boldsymbol{x}[r-1]) \cdot (1-p) \cdot \left(1 + \Omega\left(\tfrac{1}{\sqrt{k}}\right)\right)^{r - 1} \mbox{ for every } r\geq 3.
\end{equation}
This implies that \( \overline{P}(\boldsymbol{x}[r])\approx (1-p)^r \cdot \left(1 + \Omega\left(\tfrac{1}{\sqrt{k}}\right)\right)^{\binom{r}{2}} \) for $r=O(\sqrt{k})$, 
and \( \overline{P}(\boldsymbol{x}[r])\approx \overline{P}(\boldsymbol{x}[r-1]) \cdot (1-o(1))\approx (1-p)^{o(r)}\) for \( r\gg \sqrt{k} \).
Combing this with \eqref{equ:decomp-of-kappas-bar}, the analogous expression \eqref{equ:kappa-r} on $\overline{\kappa}_r$ and the assumption \( \ell \gg \sqrt{k} \), 
we obtain that
\[
\overline{P}_{\mathrm{blue}, C\ell} \gtrsim (1-p)^{o(\ell^2)} \gg (1-p)^{\binom{C\ell}{2}},
\]
which forces $n$ to be supremely small under the union bound requirement.
In summary, the dimension \( k \) must be carefully chosen in an intermediate regime, so as to balance between limiting dependencies and preserving geometric structure.

\bigskip

\noindent{\bf The Overall Strategy.}
The problem now reduces to estimating the quantities \( \kappa_r \) and \( \overline{\kappa}_r \).
Based on the earlier observations (e.g., \eqref{equ:x[r]/x[r-1]}), 
it is natural to expect that
$\frac{\kappa_{r+1}}{\kappa_r} < p$ and $\frac{\overline{\kappa}_{r+1}}{\overline{\kappa}_r} > 1 - p$.
What matters for our purposes, however, is not just these inequalities, but the relative size of the deviations: the gap \( p - \frac{\kappa_{r+1}}{\kappa_r} \) needs to significantly exceed the gap \( \frac{\overline{\kappa}_{r+1}}{\overline{\kappa}_r} - (1 - p) \).
Thus, the central part of our analysis is to estimate these gaps quantitatively in \( G_{k,p}(n) \).
We show that there exists a constant \( a_{k,p}=\left(\frac{e^{-c^2}}{2\pi}\right)^{3/2} > 0 \),  where $c:=c_{k,p}$ is from \eqref{equ:c_pk},\footnote{We will show in Lemma~\ref{prop:convergence_of_cpk} that $c_{k,p}$ converges as $k\to \infty$ so it is absolutely bounded.} such that for all \( r = O(\sqrt{k}) \),
\begin{equation}\label{equ:k/k-rough}
\frac{\kappa_{r+1}}{\kappa_r} \lesssim p - \frac{a_{k,p}}{p^2} \cdot \frac{r}{\sqrt{k}}
\quad \text{and} \quad
\frac{\overline{\kappa}_{r+1}}{\overline{\kappa}_r} \lesssim 1 - p + \frac{a_{k,p}}{(1 - p)^2} \cdot \frac{r}{\sqrt{k}}.
\end{equation}
Combining this with \eqref{equ:decomp-of-kappas} and \eqref{equ:decomp-of-kappas-bar}, and choosing \( k \approx \ell^2 \), we get
\[
P_{\mathrm{red}, \ell} \lesssim \left(p - \frac{a_{k,p}}{p^2} \cdot \frac{\ell}{3\sqrt{k}}\right)^{\binom{\ell}{2}}
\quad \text{and} \quad
\overline{P}_{\mathrm{blue}, C\ell} \lesssim \left(1 - p + \frac{a_{k,p}}{(1 - p)^2} \cdot \frac{C\ell}{3\sqrt{k}}\right)^{\binom{C\ell}{2}}.
\]
A key observation is that if \( p \) is sufficiently close to \( p_C \), then the sum of the two base terms
\begin{equation}\label{equ:sum-two-base}
\left(p - \frac{a_{k,p}}{p^2} \cdot \frac{\ell}{3\sqrt{k}}\right) + \left(1 - p + \frac{a_{k,p}}{(1 - p)^2} \cdot \frac{C\ell}{3\sqrt{k}}\right)
= 1 + \frac{a_{k,p} \cdot \ell}{3\sqrt{k}} \left(\frac{C}{(1 - p)^2} - \frac{1}{p^2}\right)
\end{equation}
is strictly less than 1, since
\begin{equation}\label{equ:p^2logp}
\frac{C \cdot p^2}{(1 - p)^2} \approx \frac{p_C^2 \log p_C}{(1 - p_C)^2 \log(1 - p_C)} < 1 \quad \text{for any } p_C \in (0, \tfrac{1}{2}).
\end{equation}
Therefore, we can choose some \( p \in (p_C, \tfrac{1}{2}) \) and a constant \( \varepsilon_0 = \varepsilon_0(C) > 0 \) such that
\[
p - \frac{a_{k,p}}{p^2} \cdot \frac{\ell}{3\sqrt{k}} \leq p_C - \varepsilon_0 \frac{\ell}{\sqrt{k}}
\quad \text{and} \quad
1 - p + \frac{a_{k,p}}{(1 - p)^2} \cdot \frac{C\ell}{3\sqrt{k}} \leq 1 - p_C - \varepsilon_0 \frac{\ell}{\sqrt{k}},
\]
completing the proof of Theorem~\ref{thm2}.

\bigskip

\noindent{\bf Breakdown into Key Quantities.}
We now examine \eqref{equ:k/k-rough} in more detail.
Recall from \eqref{equ:kappa-r} that \( \kappa_r=\mathbb{E}\big[P(\Bx[r]) \ |\ A_{\mathrm{red}, r}\big]\).
To estimate \(\kappa_r\), we consider \( P(\Bx[r]) \) for any given sequence of vectors $\Bx[r]=(\Bx_1,...,\Bx_r)\in (S^k)^r$ via the following expression:
\[
P(\Bx[r])=\prod_{s=0}^{r-1} Q_{[s]}(\Bx_{s+1}).
\]
Here, for every $0\leq s\leq r$ the quantity \( Q_{[s]}: S^k \to [0,1] \) is defined as
\[
Q_{[s]}(\boldsymbol{y}) := \frac{P(\boldsymbol{x}_1, \ldots, \boldsymbol{x}_s, \boldsymbol{y})}{P(\boldsymbol{x}_1, \ldots, \boldsymbol{x}_s)} \mbox{~ for each ~} \By\in S^k,
\]
where we use the notation from Definition~\ref{def:sphere_neighborhood}.\footnote{For the exact notation \( Q_{[s]}(\boldsymbol{y})\) used in the proof, see Definition~\ref{def:Qr}.}
Recall the projection $\pi_{[s]}(\cdot) := \pi_{\operatorname{span}(\Bx_1, \ldots, \Bx_s)}(\cdot)$ from Definition~\ref{def:projection}.
Under the assumption that $\Bx[s]\in (S^k)^s$ and $\By\in S^k$ are in generic position, 
we are able to bound $ Q_{[s]}(\boldsymbol{y})$ using the expectation of the inner product between the projection \(\pi_{[s]}(\By)\) of $\By$ and the projection of a random vector (see Theorem~\ref{ratiored}), as follows:
\[
 Q_{[s]}(\boldsymbol{y}) \lesssim p-\sqrt{\frac{k}{2\pi}} e^{-\frac{c^{2}}{2}}\cdot \mathbb{E}_{\boldsymbol{z}}[\langle \pi_{[s]}(\By),\pi_{[s]}(\Bz) \rangle],
\]
where the vector $\boldsymbol{z}$ is sampled uniformly from $N(\boldsymbol{x}[s])$, and the constant $c:=c_{k,p}$ is given by \eqref{equ:c_pk}.
Returning to the setting of \( \kappa_r=\mathbb{E}\big[P(\Bx[r]) \ |\ A_{\mathrm{red}, r}\big]\),
the expectation above becomes one taken over the joint distribution of the random vectors \(\By = \Bx_{s+1}\) and \(\Bz\).\footnote{More precisely, the expectation is taken over the following random variables: \(\Bx[r] \in (S^k)^r\) is sampled uniformly conditioned on the event \(A_{\mathrm{red}, r}\), and \(\Bz \in N(\Bx[s])\) is sampled uniformly and independently of \(\Bx[r] \setminus \Bx[s]\).}
Under these circumstances, \(\pi_{[s]}(\By)\) and \(\pi_{[s]}(\Bz)\) both behave almost like the Gaussian random vector \(\boldsymbol{w}\) (of dimension \(s\)), subject to the restrictions \(\langle \boldsymbol{w}, \Bx_i\rangle \leq -c/\sqrt{k}\) for all \(i \in [s]\).
We then prove in Section~\ref{sec:key-quan} that
\[
\mathbb{E}_{\By, \Bz}[\langle \pi_{[s]}(\By),\pi_{[s]}(\Bz) \rangle] \approx \frac{e^{-c^2}}{2\pi p^2}\cdot \frac{s}{k}.  
\]
The upper bound on \(\frac{\kappa_{r+1}}{\kappa_r}\) in \eqref{equ:k/k-rough} is essentially derived from the above estimates; 
similarly, so is \(\frac{\overline{\kappa}_{r+1}}{\overline{\kappa}_r}\).
However, random vectors may exhibit highly variable behavior, making computation intractable. 
To overcome this, we introduce a useful notion called a \textit{perfect sequence} (see Definition~\ref{dfn:perfect-sequence}). 
Roughly speaking, a sequence of vectors in \(S^k\) is said to be perfect if the vectors are nearly vertical to each other.
We then work with all of the above quantities in the corresponding ``perfect'' setting instead.
This turns out to be sufficient, as the probability that random vectors form a non-perfect sequence is sufficiently small (see Section~\ref{sec:from-perfect-to-general}).

\section{Auxiliary Lemmas}\label{sec:auxi-lemmas}
Before we proceed to the proof of Theorem~\ref{thm2}, we first establish several auxiliary but useful lemmas concerning the random sphere graph and the standard normal distribution.

Throughout the rest of the paper,
let \(\Phi(x) := \frac{1}{\sqrt{2\pi}} \int_{-\infty}^x e^{-t^2/2}\, dt\) denote the cumulative distribution function (CDF) of the standard normal distribution, and let \(\phi(x) := \Phi'(x) = \frac{1}{\sqrt{2\pi}} e^{-x^2/2}\) denote its probability density function (PDF).  
We denote by \(\Phi^{-1} : (0,1) \to \mathbb{R}\) the corresponding quantile function, defined by the property that \(\Phi\bigl(\Phi^{-1}(p)\bigr) = p\) for all \(p \in (0,1)\).

The first lemma shows that the constant $c_{k,p}$ given by \eqref{equ:c_pk} converges to $\Phi^{-1}(1-p)$ as $k\to \infty$.

\begin{lemma}\label{prop:convergence_of_cpk}
Fix any $p \in (0,1/2)$. 
The constant $c_{k,p}$ given by \eqref{equ:c_pk} satisfies
\begin{equation}\label{cpk}
c_{k,p} = \Phi^{-1}(1-p) + O\left(\frac{1}{{k}}\right).
\end{equation}
In particular, we have \(\lim_{k \to \infty} c_{k,p}= \Phi^{-1}(1-p).\)
\end{lemma}

\begin{proof}
Let \(\boldsymbol{e} \in \mathbb{R}^{k+1}\) be a fixed unit vector,
and let \(\boldsymbol{x}\) be a random point uniformly distributed on \(S^k\).
Then for any \(a \in (-1,1)\), the spherical cap probability \( \mathbb{P}(\langle \boldsymbol{x}, \boldsymbol{e} \rangle \leq a) \) satisfies that
\begin{equation}\label{sphereprob}
\mathbb{P}(\langle \boldsymbol{x}, \boldsymbol{e} \rangle \leq a) = \frac{\operatorname{Vol}\left(\{\boldsymbol{x} \in S^k : \langle \boldsymbol{x}, \boldsymbol{e} \rangle \leq a\}\right)}{\operatorname{Vol}(S^k)} = \frac{\displaystyle\int_{-1}^a (1 - t^2)^{\frac{k-2}{2}} dt}{\displaystyle\int_{-1}^1 (1 - t^2)^{\frac{k-2}{2}} dt}.
\end{equation}
Hence, the definition of the constant $c:= c_{k,p}$ in \eqref{equ:c_pk} can be translated into the following formula
\[
\int^{-\frac{c}{\sqrt k}}_{-1} (1-t^2)^{\frac{k-2}{2}} dt = p \int_{-1}^1 (1-t^2)^{\frac{k-2}{2}} dt.
\]
Substituting $t = \sin\theta$ and applying Wallis' formula, this becomes
\begin{equation}\label{equ:spherical-cap-prob}
\int_{\arcsin(\frac{c}{\sqrt k})}^{\frac{\pi}{2}} \cos^{k-1}\theta d\theta= p\int_{-\frac{\pi}{2}}^{\frac{\pi}{2}} \cos^{k-1}\theta d\theta = p \sqrt{\frac{2\pi}{k}}\left(1+O\left(\frac{1}{k}\right)\right).
\end{equation}
Using the following approximation which follows by the Taylor expansion of cosine,
\[
\left|\cos^{k-1}\theta - e^{-\frac{(k-1)\theta^2}{2}}\right| = O\left(\frac{1}{k}\right) \quad \text{for } \theta \in \left(0, \arcsin(\frac{c}{\sqrt k})\right),
\]
we can replace the cosine integral with a Gaussian integral and derive
\[
\int_0^{\arcsin(\frac{c}{\sqrt k})} e^{-\frac{(k-1)\theta^2}{2}} d\theta = \int_0^{\arcsin(\frac{c}{\sqrt k})}\cos^{k-1}\theta d\theta + O\left(\frac{1}{k\sqrt{k}}\right) = \left(\frac{1}{2}-p\right)\sqrt{\frac{2\pi}{k}}\left(1+O\left(\frac{1}{k}\right)\right),
\]
where the last equality uses \eqref{equ:spherical-cap-prob}.
Substituting \(x=\sqrt{k-1}\cdot \theta\), we have
\[
\frac{1}{\sqrt{2\pi}}\int_0^{\sqrt{k-1}\cdot\arcsin(\frac{c}{\sqrt k})}e^{-\frac{x^2}{2}}dx = \left(\frac{1}{2}-p\right)\sqrt{\frac{k-1}{k}}\left(1+O\left(\frac{1}{k}\right)\right)=\frac{1}{2}-p+O\left(\frac{1}{k}\right).
\]

This shows that 
\(\Phi\left(\sqrt{k-1}\cdot\arcsin(\frac{c}{\sqrt k})\right)=1-p+O\left(\frac{1}{k}\right)\), implying that
\[
\Phi^{-1}(1-p) = \sqrt{k-1}\arcsin\left(\frac{c}{\sqrt{k}}\right) + O\left(\frac{1}{k}\right) = c + O\left(\frac{1}{k}\right),
\]
where the last equality follows from the Taylor expansion of the function \( \arcsin \).  
\end{proof}

The following lemma gives a concentration inequality for the probability that random points sampled uniformly from the sphere \(S^k\) are nearly mutually orthogonal.
In particular, it serves as a key tool for the analysis of perfect and non-perfect sequences (to be defined in Section~\ref{Perfect sequences}). 

\begin{lemma}\label{Volumenonperfect}
Fix a constant \(C > 1\), and define \(\alpha_C := \max\left\{1000,\, 20\sqrt{C \log(10/p_C)}\right\}\).
Let \(\ell_0 \gg D \gg C\) be as specified in Theorem~\ref{thm2}. 
Suppose that \(\ell \geq \ell_0\) and \(k = D^2 \ell^2\).  
Let \(X_r \subset \mathbb{R}^{k+1}\) be an \(r\)-dimensional subspace, where \(1 \leq r \leq C\ell\).  
Then for a random vector \(\boldsymbol{y}\) sampled uniformly from the unit sphere \(S^k\), the orthogonal projection \(\pi_{[r]}: \mathbb{R}^{k+1} \to X_r\) satisfies the following tail bound:
\[
\mathbb{P}\left(|\pi_{[r]}(\boldsymbol{y})| > \frac{\alpha_C \sqrt{\ell}}{2 \sqrt{k}}\right) \leq \left(\frac{p_C}{10}\right)^{C \ell}.
\]
\end{lemma}

\begin{proof}
Let \(X' \subseteq \mathbb{R}^{k+1}\) be a \(C\ell\)-dimensional linear subspace such that \(X_r \subseteq X'\). Then,
\[
\mathbb{P}\left(\left|\pi_{[r]}(\boldsymbol{y})\right| > \frac{\alpha_C \sqrt{\ell}}{2\sqrt{k}}\right) \leq \mathbb{P}\left(\left|\pi_{X'}(\boldsymbol{y})\right| > \frac{\alpha_C \sqrt{\ell}}{2\sqrt{k}}\right).
\]
Hence, it suffices to consider the case when \(r = C\ell\). 
From now on, we assume $X_r=X'$.

Let $\By$ be a random vector sampled uniformly from \(S^k\).
Define \(u = |\pi_{X'}(\By)|^2\). 
Then, \(u \in [0,1]\) follows a Beta distribution with parameters \(\left(\frac{C\ell}{2}, \frac{k - C\ell + 1}{2}\right)\); see Appendix~\ref{app:beta} for a complete proof. 
That is, the PDF of \(u\) is  
\[
f(u) = \frac{u^{\frac{C\ell}{2} - 1} (1 - u)^{\frac{k - C\ell - 1}{2}}}{B\left(\frac{C\ell}{2}, \frac{k - C\ell + 1}{2}\right)}, \mbox{ where } B(x,y)=\int_{0}^1 t^{x-1}(1-t)^{y-1}dt \mbox{~denotes the beta function}.
\]
Let \(\theta_1:= \arcsin\left(\frac{\alpha_C \sqrt{\ell}}{2\sqrt{k}}\right)\). A straightforward calculation then yields that
\begin{equation}\label{integral111}
\begin{aligned}
\mathbb{P}\left(\left|\pi_{X'}(\boldsymbol{y})\right| > \frac{\alpha_C \sqrt{\ell}}{2\sqrt{k}}\right) 
&= \frac{\displaystyle\int_{\frac{\alpha_C^2 \ell}{4k}}^1 u^{\frac{C\ell}{2} - 1} (1 - u)^{\frac{k - C\ell - 1}{2}} \, du}{\displaystyle\int_0^1 u^{\frac{C\ell}{2} - 1} (1 - u)^{\frac{k - C\ell - 1}{2}} \, du} = \frac{\displaystyle\int_{\theta_1}^{\frac{\pi}{2}} \sin^{C\ell - 1}\theta \, \cos^{k - C\ell}\theta \, d\theta}{\displaystyle\int_0^{\frac{\pi}{2}} \sin^{C\ell - 1}\theta \, \cos^{k - C\ell}\theta \, d\theta},
\end{aligned}
\end{equation}
where the second equality follows from the substitution \(u = \sin^2\theta\).
We begin by bounding the numerator of \eqref{integral111} from above. 
Using the elementary inequalities that $2x\geq \arcsin x \geq x \geq \arctan x$ for all $0\leq x\ll 1$, together with the conditions that $k=D^2\ell^2\gg C\ell$ and $\alpha_C \geq 4\sqrt{C}$, we obtain
\begin{equation}\label{equ:theta_1}
\frac{\alpha_C \sqrt{\ell}}{\sqrt{k}} \geq \theta_1=\arcsin\left(\frac{\alpha_C\sqrt{\ell}}{2\sqrt{k}}\right)\geq \frac{\alpha_C\sqrt{\ell}}{2\sqrt{k}}\geq \sqrt{\frac{C\ell-1}{k-C\ell}}\geq \arctan\left(\sqrt{\frac{C\ell-1}{k-C\ell}}\right).
\end{equation}
This implies that for every $\theta_1 \leq \theta \leq \frac{\pi}{2}$, we have
\[
\frac{d}{d\theta}\left(\sin^{C\ell-1}\theta \cos^{k-C\ell}\theta\right)=\left((C\ell-1)-(k-C\ell)\tan^2\theta \right)\left(\sin^{C\ell-2}\theta \cos^{k-C\ell+1}\theta\right)\leq 0,
\]
and hence, \(\sin^{C\ell-1}\theta \cos^{k-C\ell}\theta\leq \sin^{C\ell-1}\theta_1 \cos^{k-C\ell}\theta_1\). 
Therefore, the numerator of \eqref{integral111} satisfies
\begin{equation}\label{eq:upper_bound}
\int_{\theta_1}^{\frac{\pi}{2}} \sin^{C\ell-1}\theta \cos^{k-C\ell}\theta \,d\theta \leq \frac{\pi}{2} \sin^{C\ell-1}\theta_1 \cos^{k-C\ell}\theta_1 \leq \frac{\pi}{2}\cdot \theta_1^{C\ell-1}\cdot \left(1 - \frac{\theta_1^2}{3}\right)^{k-C\ell},
\end{equation}
where the last inequality uses the facts that $\sin\theta_1 \leq \theta_1$ and $\cos\theta_1 \leq 1 - \frac{\theta_1^2}{3}$.

Next we bound the denominator of \eqref{integral111} from below. 
Since $k=D^2\ell^2\gg C\ell$, similarly as before,
\begin{equation}\label{equ:theta_2}
    1\gg \frac{2\sqrt{C\ell}}{\sqrt{k}}\geq \theta_2:= \arctan\left(\sqrt{\frac{C\ell-1}{k-C\ell}}\right)\geq \frac{2}{3}\sqrt{\frac{C\ell-1}{k-C\ell}}\geq \frac{1}{2}\sqrt{\frac{\ell}{k}}.
\end{equation}
For every $\theta_2 \leq \theta \leq 2\theta_2<\pi/2$, 
we have $\tan^2\theta\geq \tan^2\theta_2=\frac{C\ell-1}{k-C\ell}$ and thus
\[
\frac{d}{d\theta}\left(\sin^{C\ell-1}\theta \cos^{k-C\ell}\theta\right)
=\left((C\ell-1)-(k-C\ell)\tan^2\theta \right)\left(\sin^{C\ell-2}\theta \cos^{k-C\ell+1}\theta\right)\leq 0,
\]
implying that $\sin^{C\ell-1}\theta \cos^{k-C\ell}\theta\geq \sin^{C\ell-1}(2\theta_2) \cos^{k-C\ell}(2\theta_2)$. 
Using this, we can derive
\begin{equation}\label{eq:lower_bound}
\begin{aligned}
\int_0^{\pi/2} \sin^{C\ell-1}\theta \cos^{k-C\ell}\theta \,d\theta 
&\geq \int_{\theta_2}^{2\theta_2} \sin^{C\ell-1}\theta \cos^{k-C\ell}\theta \,d\theta \\
&\geq \theta_2 \sin^{C\ell-1}(2\theta_2) \cos^{k-C\ell}(2\theta_2) 
\geq \frac{1}{2}\sqrt{\frac{\ell}{k}}\cdot \theta_2^{C\ell-1}\left(1 - 2\theta_2^2\right)^{k-C\ell},
\end{aligned}
\end{equation}
where the last inequality uses \eqref{equ:theta_2} and the estimates
$\sin(2\theta_2) \geq \theta_2$ and $\cos(2\theta_2) \geq 1 - 2\theta_2^2$.

Substituting the bounds from \eqref{eq:upper_bound} and \eqref{eq:lower_bound} into \eqref{integral111}, we have
\begin{equation}\label{equ:pi_X'}
\begin{aligned}
\mathbb{P}\left(\left|\pi_{X'}(\boldsymbol{y})\right| > \frac{\alpha_C \sqrt{\ell}}{2\sqrt{k}}\right) 
&\leq \pi\sqrt{\frac{k}{\ell}}\cdot \frac{\theta_1^{C\ell-1}\left(1 - \frac{1}{3}\theta_1^2\right)^{k - C\ell}}{\theta_2^{C\ell-1}\left(1 - 2\theta_2^2\right)^{k - C\ell}}
\leq  \pi D\ell\cdot \left(\frac{\theta_1}{\theta_2}\right)^{C\ell} \exp\left(3k\theta_2^2 - k\theta_1^2/6\right) \\
&\leq \pi D\ell\cdot \exp\left(\ell\cdot \big(C\log(2 \alpha_C)+12C - \alpha_C^2/24\big)\right), 
\end{aligned}
\end{equation}
where the second inequality holds since \(\left(1 - \frac{1}{3}\theta_1^2\right)^{k-C\ell} \leq \left(\exp\left(-\frac{1}{3}\theta_1^2 \right)\right)^{k/2}=\exp(-\frac16k\theta_1^2)\) and \(\left(1 - 2\theta_2^2\right)^{k-C\ell} \geq \left(\exp(-3\theta_2^2)\right)^k=\exp(-3k\theta_2^2)\),
and the last inequality follows from \eqref{equ:theta_1} and \eqref{equ:theta_2}.
Recall that \( \alpha_C = \max\{1000, 20\sqrt{C \log(10/p_C)}\} \). 
We claim that
\begin{equation}\label{equ:-1/150}
    C\log(2 \alpha_C)+12C - \frac{\alpha_C^2}{24}\leq -\frac{\alpha_C^2}{60}.
\end{equation}
To see this, note that $p_C^2\cdot C\leq \frac{p_C^2\cdot C}{(1-p_C)^2}=\frac{p_C^2 \log p_C}{(1-p_C)^2 \log (1-p_C)}<1$ (because of \eqref{equ:p^2logp}), which implies that $\log C\leq 2\log (1/p_C)$.
Combining with $\alpha_C\geq 20\sqrt{C \log(10/p_C)}$, this implies that
\[
\frac{C\log(2\alpha_C)}{\alpha_C^2}\leq \frac{\log 40+\frac12\log C+\frac12\log\log(10/p_C)}{400\log(10/p_C)}\leq \frac{1}{200}+\frac{1}{400}+\frac{1}{800}<\frac{1}{100},
\]
where the second inequality follows by the facts that $\log 40<4$, $\log(10/p_C)>\log 10>2$ and $\log\log(10/p_C)<\log(10/p_C)$.
This, together with $\alpha_C^2\geq 800C$, would imply \eqref{equ:-1/150}.

Finally, substituting the bound \eqref{equ:-1/150} into \eqref{equ:pi_X'}, 
we derive 
\[
\mathbb{P}\left(\left|\pi_{X'}(\boldsymbol{y})\right| > \frac{\alpha_C \sqrt{\ell}}{2\sqrt{k}}\right) \leq \pi D\ell \cdot \exp\left(-\frac{\alpha_C^2\ell}{60}\right)\leq \exp\left(-\frac{\alpha_C^2\ell}{100}\right)\leq \left(\frac{p_C}{10}\right)^{C\ell},
\]
where the last inequality holds since $\alpha^2_C\geq 100C\log(10/p_C)$,
completing the proof.
\end{proof}

The next lemma is crucial in the proof of Theorem~\ref{ratiored}, particularly for obtaining estimates on the probability ratios \(Q_{[r]}(\By)\) and \(\overline{Q}_{[r]}(\By)\).

\begin{lemma}\label{limrk3}
Fix a constant $C>1$.
Let \(\ell_0 \gg D \gg C\) be as specified in Theorem~\ref{thm2}. 
Suppose that \(\ell \geq \ell_0\) and \(k = D^2 \ell^2\).  
Fix a constant $A_0 > 0$. Then for any \(A \in (-A_0, A_0)\), define
\[
H := -\frac{c}{\sqrt{k}} - \frac{A}{D\sqrt{k}},
\]
where $c:=c_{k,p}$ is given by \eqref{equ:c_pk}.
Let $1\leq r\leq C\ell$.
For a random vector $\By$ sampled uniformly from $S^{k-r} \subseteq \mathbb{R}^{k-r+1}$ and any fixed unit vector $\boldsymbol{e} \in \mathbb{R}^{k-r+1}$, we have
\[
\mathbb{P}\left(\langle \boldsymbol{y}, \boldsymbol{e} \rangle \leq H\right) = p - \frac{A}{\sqrt{2\pi}D}e^{-c^2/2} + O\left(\frac{1}{D^2}\right),
\]
where the convergence is uniform over all $A \in (-A_0, A_0)$, and the implied constant in $O(\frac{1}{D^2})$ depends only on $A_0$.
\end{lemma}

\begin{proof}
Let us begin by recalling two well-known estimates: first, Wallis' formula
\[
\int_0^1 (1 - t^2)^{\frac{k - r - 2}{2}} \, dt = \int_0^{\frac{\pi}{2}} \cos^{k - r - 1} x \, dx = \frac{\sqrt{\pi}}{\sqrt{2(k - r - 2)}} + O\left(\frac{1}{k \sqrt{k}}\right);
\]
second, the Gaussian integral
\(\int_{-\infty}^{\infty} e^{-\frac{k - r - 2}{2} t^2} \, dt = \frac{\sqrt{2\pi}}{\sqrt{k - r - 2}}\).
Using these, along with the fact that $e^{-\frac{k-r-2}{2} t^2} - \left(1 - t^2\right)^{\frac{k-r-2}{2}} \geq 0$, we derive that for all $t \in [-1,1]$, 
\begin{align*}
0\leq&\int_{-\infty}^{H} e^{-\frac{k-r-2}{2} t^2} dt - \int_{-1}^H \left(1 - t^2\right)^{\frac{k-r-2}{2}} dt\\
\leq&\int_{-\infty}^{\infty} e^{-\frac{k-r-2}{2} t^2} dt - \int_{-1}^1 \left(1 - t^2\right)^{\frac{k-r-2}{2}} dt = O\left(\frac{1}{k\sqrt k}\right).
\end{align*}
This, together with \eqref{sphereprob}, implies that
\begin{equation}\label{equ:xeleqH}
\begin{aligned}
\mathbb{P}\left(\langle \boldsymbol{y}, \boldsymbol{e} \rangle \leq H\right) 
&= \frac{\displaystyle\int_{-1}^{H} \left(1 - t^2\right)^{\frac{k-r-2}{2}} dt}{\displaystyle\int_{-1}^1 \left(1 - t^2\right)^{\frac{k-r-2}{2}} dt}= \frac{\displaystyle\int_{-\infty}^{H} e^{-\frac{k-r-2}{2} t^2} dt + O\left(\frac{1}{k\sqrt{k}}\right)}{\displaystyle\int_{-\infty}^{\infty} e^{-\frac{k-r-2}{2} t^2} dt + O\left(\frac{1}{k\sqrt{k}}\right)} \\
&= \frac{\displaystyle\int_{-\infty}^{H} e^{-\frac{k-r-2}{2} t^2} dt}{\displaystyle\int_{-\infty}^{\infty} e^{-\frac{k-r-2}{2} t^2} dt} + O\left(\frac{1}{k}\right)= \int_{-\infty}^{-(c + \frac{A}{D})\frac{\sqrt{k-r-2}}{\sqrt{k}}} \frac{1}{\sqrt{2\pi}} e^{-\frac{t^2}{2}} dt + O\left(\frac{1}{k}\right)\\
&=\Phi\left(-\left(c + \frac{A}{D}\right)\frac{\sqrt{k-r-2}}{\sqrt{k}}\right) + O\left(\frac{1}{k}\right), 
\end{aligned}
\end{equation}
where \(\Phi\) denotes the CDF of the standard normal distribution.
By Lemma~\ref{prop:convergence_of_cpk}, we have 
\begin{align*}
-\left(c + \frac{A}{D}\right)\sqrt{\frac{k-r-2}{k}} &= -\left(c+\frac{A}{D}\right)\left(1 - \frac{r+2}{2k} + O\left(\frac{r^2}{k^2}\right)\right)= -c - \frac{A}{D} + O\left(\frac{r}{k}\right) \\
&= -\Phi^{-1}(1-p) - \frac{A}{D} + O\left(\frac{\ell}{k}\right) = \Phi^{-1}(p) - \frac{A}{D} + O\left(\frac{1}{D\sqrt{k}}\right).
\end{align*}
Combining this with \eqref{equ:xeleqH} and using the Taylor expansion of $\Phi$, we obtain
\[
\begin{aligned}
\mathbb{P}\left(\langle \boldsymbol{y}, \boldsymbol{e} \rangle \leq H\right) &= \Phi\left(\Phi^{-1}(p) - \frac{A}{D} + O\left(\frac{1}{D\sqrt{k}}\right)\right) + O\left(\frac{1}{k}\right) \\
&= p - \frac{A}{D} \phi(\Phi^{-1}(p)) + O\left(\frac{1}{D^2}\right) = p - \frac{A}{\sqrt{2\pi} D} e^{-\frac{c^2}{2}} + O\left(\frac{1}{D^2}\right),
\end{aligned}
\]
where the last equality follows from Lemma~\ref{prop:convergence_of_cpk} that $\Phi^{-1}(p)=-c+O(\frac{1}{k})$, completing the proof.
\end{proof}

We conclude with a lemma on technical properties of the standard norm distribution, used to estimate expectations in Lemma~\ref{lem:coefficient_expectations}.

\begin{lemma}\label{lem:normal_conditional_expectation}
Let \(X \sim \mathcal{N}(0,1)\) denote a standard normal random variable.
Then the conditional expectation \(\mu(t) := \mathbb{E}[X \mid X \geq t]\) satisfies
\[
\mu(t) = \frac{\phi(t)}{1 - \Phi(t)}=\frac{\displaystyle e^{-\frac{t^2}{2}}}{\displaystyle\int_{t}^{\infty}e^{-\frac{s^2}{2}}ds}, 
\quad \mbox{and moreover,} \quad 
\left|\mu'(t)\right| \leq 100 \quad \text{for all } t \in \mathbb{R}.
\]
\end{lemma}

\begin{proof}
The formula for \(\mu(t)\) follows from a straightforward calculation:
\[\mu(t)=\frac{\displaystyle\frac{1}{\sqrt{2\pi}}\int_{t}^{\infty}se^{-\frac{s^2}{2}}ds}{\displaystyle \frac{1}{\sqrt{2\pi}}\int_{t}^{\infty}e^{-\frac{s^2}{2}}ds}=\frac{\displaystyle e^{-\frac{t^2}{2}}}{\displaystyle \int_{t}^{\infty}e^{-\frac{s^2}{2}}ds}=\frac{\phi(t)}{1-\Phi(t)}.\] 
Differentiating yields that
$\mu'(t) = \frac{\phi'(t)(1 - \Phi(t)) + \phi(t)^2}{(1 - \Phi(t))^2} = \mu(t)^2 - t\mu(t).$
For \(t \geq 1\), we obtain
 \[t \leq \mu(t)=\frac{\displaystyle e^{-\frac{t^2}{2}}}{\displaystyle\int_{t}^{\infty}e^{-\frac{s^2}{2}}ds}\leq\frac{\displaystyle(1+1/t^2)e^{-\frac{t^2}{2}}}{\displaystyle\int_{t}^{\infty}(1+1/s^2)e^{-\frac{s^2}{2}}ds} =\frac{\displaystyle(1+1/t^2)e^{-\frac{t^2}{2}}}{\displaystyle e^{-\frac{t^2}{2}}/t}= \frac{t^2+1}{t},\]
implying \(0\leq \mu'(t)=\mu(t)^2 - t\mu(t)\leq 2\). 
For \(t \leq 1\), since $1-\Phi(t)\geq1-\Phi(1)\geq\frac{1}{\sqrt{2\pi}}e^{-2}$, 
we have 
\[
|\mu'(t)| \leq \left|\frac{\phi'(t)}{1 - \Phi(t)}\right|+\left|\frac{ \phi(t)^2}{(1 - \Phi(t))^2}\right|\leq e^2\sup_{t\in\mathbb{R}}|te^{-\frac{t^2}{2}}|+e^4\sup_{t\in\mathbb{R}}|e^{-{t^2}}|\leq e^2+e^4<100.
\]
Putting everything together, we conclude that \(\left|\mu'(t)\right| \leq 100\) for all \(t \in \mathbb{R}\).
\end{proof}

\section{Perfect Sequences}\label{Perfect sequences}
Before defining perfect sequences, we first set up the parameters used throughout the remainder of the paper.  
We fix a constant \( C > 1 \).  
Let \( p_C \in (0, 1/2) \) be as defined earlier, and set \( M_C=p_C^{-1/2}\). 
Let \(\alpha_C\) be as given in Lemma~\ref{Volumenonperfect}.  
Let \(\ell_0 \gg D \gg C\) be as specified in Theorem~\ref{thm2}.  
Throughout, we assume \( p \in \bigl(p_C, \tfrac{1}{2}\bigr) \), \(\ell \ge \ell_0\), and \( k = D^2 \ell^2 \).  
Let \( c := c_{k, p} \) be defined by \eqref{equ:c_pk}.  
Our focus is on the Ramsey number \( r(\ell, C\ell) \) in the setting of the random sphere graph \( G_{k,p}(n) \).

Let \(\Bx[r] = (\Bx_1, \ldots, \Bx_r)\) be a sequence of points on the sphere \(S^k\subseteq \mathbb{R}^{k+1}\).  
For each \(i \in [r]\), define  
\[
X_i = \operatorname{span}(\Bx_1, \ldots, \Bx_i).
\]
For each \(1\leq i \leq r-1\), denote by \(\delta_{i+1} \in [0, \frac{\pi}{2}]\) the angle between the vector \(\Bx_{i+1}\) and the \(i\)-dimensional subspace \(X_{i}\).
Recall the projection $\pi_{[i]}(\cdot): \mathbb{R}^{k+1}\to X_i$ from Definition~\ref{def:projection}.

\begin{definition}[\bf{Perfect Sequences}]\label{dfn:perfect-sequence}
For \(2 \leq r \leq C \ell\),
a sequence $\boldsymbol{x}{[r]}= (\Bx_1, \ldots, \Bx_r)\in (S^k)^r$ is called \emph{perfect}, if for every $1 \leq i \leq r-1$,
\[
|\pi_{[i]}(\boldsymbol{x}_{i+1})|=\cos\delta_{i+1} \leq \frac{\alpha_C\sqrt{\ell}}{\sqrt{k}}.
\]
By convention, any singleton \((\boldsymbol{x}_1) \in S^k\) is considered perfect.
\end{definition}

It can be easily verified that any subsequence of a perfect sequence is perfect; we refer to this property as \emph{monotonicity}. 
In view of Lemma~\ref{Volumenonperfect}, we expect that the probability of a random sequence failing to be perfect is extremely small.  
With this in mind, we now introduce modified notations tailored to the setting of perfect sequences.  
First, we define the ``perfect'' variants of neighborhoods and their corresponding probability measures.

\begin{definition}\label{perfectneighborhood}
For any perfect sequence $\Bx[r] = (\Bx_1,\ldots,\Bx_r)$, we define
\begin{itemize}
    \item The \emph{perfect red-neighborhood}
    \[
    N_{\mathrm{per}}(\Bx{[r]}) := \left\{\boldsymbol{y} \in N(\Bx[r]) : (\Bx[r],\By) \text{ is perfect}\right\},\footnote{Here and throughout, we write \( (\Bx[r], \By) = (\Bx_1, \ldots, \Bx_r, \By) \).}
    \]
    with corresponding probability measure 
    \[
    P_{\mathrm{per}}(\Bx[r]) := \mathbb{P}\left( N_{\mathrm{per}}(\Bx[r])\right).
    \]
    \item The \emph{perfect blue-neighborhood}
    \[
    \overline{N}_{\mathrm{per}}(\Bx[r]) := \left\{\boldsymbol{y} \in \overline{N}(\Bx{[r]}) : (\Bx[r],\By) \text{ is perfect}\right\},
    \]
    with corresponding probability measure 
    \[
    \overline{P}_{\mathrm{per}}(\Bx[r]) := \mathbb{P}\left( \overline{N}_{\mathrm{per}}(\Bx[r])\right).
    \]
\end{itemize}
\end{definition}

The following definition provides the ``perfect'' analogues of Definition~\ref{def:prod_redblue} and of the quantities \(\kappa_r\) and \(\overline{\kappa}_r\) (see, e.g., \eqref{equ:kappa-r}).  
These are the primary objects of study in the rest of the paper.

\begin{definition}\label{def:perfect_probs}
For \(1 \leq r \leq C \ell\), let \( \boldsymbol{x}[r] = (\boldsymbol{x}_1, \ldots, \boldsymbol{x}_r) \) be a random \( r \)-tuple in \( (S^k)^r \).
\begin{itemize}
    \item Let \( A_{\mathrm{red},r} \) denote the event that \( G_p(\boldsymbol{x}[r]) \) forms a red clique, and let \( \overline{A}_{\mathrm{blue},r} \) denote the event that \( G_p(\boldsymbol{x}[r]) \) forms a blue clique.

    \item Let \( B_r \) denote the event that the sequence \( \boldsymbol{x}[r] \) is perfect.\footnote{For notational convenience, we treat each of \( A_{\mathrm{red},1} \), \( \overline{A}_{\mathrm{blue},1} \), and \( B_1 \) as events that always occur.}

    \item Define \( P_{\mathrm{red},r}^{\mathrm{per}}:= \mathbb{P}(A_{\mathrm{red},r} \wedge B_r) \) ~and~ \( \overline{P}_{\mathrm{blue},r}^{\mathrm{per}}:= \mathbb{P}(\overline{A}_{\mathrm{blue},r} \wedge B_r)\).

    \item Define $\kappa_r^{\mathrm{per}}:=\mathbb{E}\left[P_{\mathrm{per}}(\boldsymbol{x}{[r]})\ |\ A_{\mathrm{red},r}\wedge B_r\right]$ ~and~ $\overline{\kappa}_r^{\mathrm{per}}:=\mathbb{E}\left[\overline{P}_{\mathrm{per}}(\boldsymbol{x}{[r]})\ |\ \overline{A}_{\mathrm{blue},r}\wedge B_r\right].$
\end{itemize}
\end{definition}

We conclude this section with two simple lemmas.  
The first lemma connects the quantities \( P_{\mathrm{red}, r}^{\mathrm{per}} \) and \( \kappa_r^{\mathrm{per}} \),  
and is analogous to the decomposition in \eqref{equ:decomp-of-kappas} presented in the proof sketch of Theorem~\ref{thm2}.  
For any event \( A \), let \(\boldsymbol{1}_A\) denote its indicator random variable.

\begin{lemma}\label{lem:prob_of_perfect}
For any \(1 \leq r \leq C \ell\), we have 
\[
\kappa_r^{\mathrm{per}}=\frac{P_{\mathrm{red},r+1 }^{\mathrm{per}}}{P_{\mathrm{red},r }^{\mathrm{per}}} 
\quad \mbox{and} \quad 
\overline{\kappa}_r^{\mathrm{per}}=\frac{\overline{P}_{\mathrm{blue},r+1 }^{\mathrm{per}}}{\overline{P}_{\mathrm{blue},r }^{\mathrm{per}}}.
\]
Equivalently,
$$
P_{\mathrm{red},r+1}^{\mathrm{per}}= \kappa_1^{\mathrm{per}} \cdots \kappa_{r}^{\mathrm{per}} 
\quad \mbox{and} \quad 
\overline{P}_{\mathrm{blue},r+1}^{\mathrm{per}}= \overline{\kappa}_1^{\mathrm{per}} \cdots \overline{\kappa}_{r}^{\mathrm{per}}.
$$
\end{lemma}

\begin{proof}
We proceed by induction on $r$ to prove that \( P_{\mathrm{red},r+1}^{\mathrm{per}}= \kappa_1^{\mathrm{per}} \cdots \kappa_{r}^{\mathrm{per}} \). 
Since \( P_{\mathrm{red}, 1}^{\mathrm{per}} = 1 \), the base case \( r = 1 \) holds trivially.  
Assume the statement holds for some \( r \geq 1 \).
Let $\boldsymbol{x}{[r+1]} = (\boldsymbol{x}_1, \ldots, \boldsymbol{x}_{r+1})$ 
be a random \( (r+1)\)-tuple in \( (S^k)^{r+1} \). 
Define $A=A_{\mathrm{red},r}\wedge B_r$.
So $A$ depends only on $\boldsymbol{x}[r]$, with 
\[
P_{\mathrm{red},r}^{\mathrm{per}}=\mathbb{P}(A) \quad \mbox{and} \quad 
\kappa_r^{\mathrm{per}}=\mathbb{E}\left[P_{\mathrm{per}}(\boldsymbol{x}{[r]})\ |\ A\right].
\]
Let $Y$ be the event that $\boldsymbol{x}_{r+1} \in N_{\mathrm{per}}(\boldsymbol{x}{[r]})$. 
Then $\mathbb{E}[\boldsymbol{1}_{Y}|~\boldsymbol{x}{[r]}] = \mathbb{P}(Y|~\boldsymbol{x}{[r]}) = P_{\mathrm{per}}(\boldsymbol{x}{[r]})$. 
Using this, we can derive that $P_{\mathrm{red},r+1}^{\mathrm{per}} 
= \mathbb{P}(A \wedge Y)$ is equal to
\begin{equation*}
\begin{aligned}
\mathbb{E}[\boldsymbol{1}_{A} \cdot \boldsymbol{1}_{Y}] = \mathbb{E}\big[\boldsymbol{1}_{A}\cdot \mathbb{E}[\boldsymbol{1}_{Y} |~ \boldsymbol{x}[r]]\big] = \mathbb{E}\big[\boldsymbol{1}_{A} \cdot P_{\mathrm{per}}(\boldsymbol{x}[r])\big] 
= \mathbb{E}\big[P_{\mathrm{per}}(\boldsymbol{x}[r]) \mid {A}\big] \cdot \mathbb{P}(A) 
= \kappa_{r}^{\mathrm{per}} \cdot P_{\mathrm{red},r}^{\mathrm{per}},
\end{aligned}
\end{equation*}
where the first equality holds since $A$ depends only on $\boldsymbol{x}[r]$ and
the third equality follows from the definition of conditional expectation.
Now, the statement for $P_{\mathrm{red},r+1}^{\mathrm{per}}$ follows easily by induction. 
We omit the analogous proof for $\overline{P}_{\mathrm{blue},r+1}^{\mathrm{per}}$.
\end{proof}

The last lemma highlights an important property that partly reflects the coherent definition of perfect sequences. 
This will play a crucial role in the proofs presented in the next section.

\begin{lemma}\label{lem:hat(z)}
For \(1 \leq r \leq C\ell\), let \(\Bx[r] \in (S^k)^r\) be a given perfect sequence. 
Let \(\Bz\) be a random vector uniformly distributed in either \(N_{\mathrm{per}}(\Bx[r])\) or  \(\overline{N}_{\mathrm{per}}(\Bx[r])\), with the orthogonal decomposition \(\Bz = \tilde{\Bz} + \hat{\Bz}\), 
where \(\tilde{\Bz} = \pi_{[r]}(\Bz) \in X_r\) and \(\hat{\Bz} \in X_r^{\perp}\).
Then the normalized random vector \(\frac{\hat{\Bz}}{|\hat{\Bz}|}\) is independent of \(\tilde{\Bz}\) and is uniformly distributed on \(S^k\cap X_{r}^{\perp}\).
\end{lemma}

\begin{proof}
Let \(\Bx[r] = (\Bx_1, \ldots, \Bx_r)\).  
By symmetry, assume $\Bz$ is sampled uniformly at random from \(N_{\mathrm{per}}(\Bx[r])\).
Since \(\Bz - \tilde{\Bz} = \hat{\Bz} \in X_r^{\perp}\),  
we have \(\langle \Bz, \Bx_i \rangle = \langle \tilde{\Bz}, \Bx_i \rangle\) for each \(i \in [r]\).  
Combined with the definition of perfect sequences, this shows that the event \(\Bz \in N_{\mathrm{per}}(\Bx[r])\) depends only on \(\tilde{\Bz}\) and is thus independent of the choice of \(\frac{\hat{\Bz}}{|\hat{\Bz}|}\).  
The conclusion then follows directly.
\end{proof}

\section{Preliminary Estimates on Perfect Sequences}\label{sec:estimation_q}
In this section, we begin our study of perfect sequences.  
Recall the parameters we set up at the beginning of Section~\ref{Perfect sequences}.
As we proceed to estimate various probabilistic quantities, it is important to clarify the following convention used throughout:  
\begin{itemize}
    \item The big-O notation \( O(\cdot) \) denotes terms whose implicit constant may depend on \( C, p_C, \alpha_C, p, \) and \( c \), but \emph{not} on \( D \).
\end{itemize}

In this section, we aim to prove two results about perfect sequences, both of which rely on Lemma~\ref{limrk3} and Lemma~\ref{lem:hat(z)}.  
The first result establishes upper bounds for the central quantities \( Q_{[r]}(\By) \) and \( \overline{Q}_{[r]}(\By) \), whose formal definitions are given below.

\begin{definition}\label{def:Qr}
For $1\leq r\leq C\ell$,  let \((\Bx[r], \By)\in (S^k)^{r+1}\) be a perfect sequence.  Define
\[
Q_{[r]}(\boldsymbol{y}) := \frac{P_{\mathrm{per}}(\Bx[r], \By)}{P_{\mathrm{per}}(\Bx[r])}
\quad \mbox{and} \quad
\overline{Q}_{[r]}(\By):=\frac{\overline{P}_{\mathrm{per}}(\Bx[r],\By)}{\overline{P}_{\mathrm{per}}(\Bx[r])}.
\]
\end{definition}

Note that, in particular, we have \(Q_{[0]}(\By)=P_{\mathrm{per}}(\By)\leq p\) and \(\overline{Q}_{[0]}(\By)=\overline{P}_{\mathrm{per}}(\By)\leq 1-p\).

\begin{theorem}\label{ratiored}
Let \( (\Bx[r], \By) \in (S^k)^{r+1} \) be a given perfect sequence. Then the following hold:
\begin{equation}\label{eq:Qrx}
    Q_{[r]}(\By) \leq p - \sqrt{\frac{k}{2\pi}}\, e^{-\frac{c^2}{2}}\, \mathbb{E}_{\Bz}\bigl[\langle  \pi_{[r]}(\By) ,  \pi_{[r]}(\Bz) \rangle\bigr] + O\left(\frac{1}{D^2}\right),
\end{equation}
where \(\Bz\) denotes the random vector sampled uniformly from \( N_{\mathrm{per}}(\Bx[r]) \), and
\begin{equation}\label{eq:Qrx2}
    \overline{Q}_{[r]}(\By) \leq 1 - p + \sqrt{\frac{k}{2\pi}}\, e^{-\frac{c^2}{2}}\, \mathbb{E}_{\Bz}\bigl[\langle  \pi_{[r]}(\By) , \pi_{[r]}(\Bz) \rangle\bigr] + O\left(\frac{1}{D^2}\right),
\end{equation}
where \(\Bz\) denotes the random vector sampled uniformly from \(\overline{N}_{\mathrm{per}}(\Bx[r])\).
\end{theorem}

\begin{proof}
For conciseness, we present the full proof for \eqref{eq:Qrx} only, as the proof for \eqref{eq:Qrx2} follows analogously.  
Suppose that \((\Bx{[r]}, \By)\) is a perfect sequence,  
and let \(\Bz\) be sampled uniformly at random from \( N_{\mathrm{per}}(\Bx{[r]}) \). Consider the orthogonal decompositions:
\[
\By = \tilde{\By} + \hat{\By}
\quad \mbox{and} \quad
\Bz = \tilde{\Bz} + \hat{\Bz},
\]
where \(\tilde{\By} = \pi_{[r]}(\By)\) and \(\tilde{\Bz} = \pi_{[r]}(\Bz)\) belong to \(\operatorname{span}(\Bx{[r]})\), while \(\hat{\By}\) and \(\hat{\Bz}\) lie in \(\left(\operatorname{span}(\Bx{[r]})\right)^{\perp}\).
Since $\langle {\boldsymbol{y}}, {\boldsymbol{z}} \rangle=\langle \tilde{\boldsymbol{y}}, \tilde{\boldsymbol{z}} \rangle+\langle \hat{\boldsymbol{y}}, \hat{\boldsymbol{z}} \rangle$, we have $\boldsymbol{z} \in N(\boldsymbol{y})$ if and only if
\begin{equation}\label{equ:Q[r]-H}
\left\langle \frac{\hat{\boldsymbol{y}}}{|\hat{\boldsymbol{y}}|}, \frac{\hat{\boldsymbol{z}}}{|\hat{\boldsymbol{z}}|} \right\rangle \leq -\frac{1}{|\hat{\boldsymbol{y}}||\hat{\boldsymbol{z}}|} \left( \frac{c}{\sqrt{k}} + \langle \tilde{\boldsymbol{y}}, \tilde{\boldsymbol{z}} \rangle \right) =: H.
\end{equation}
Both $(\boldsymbol{x}{[r]},\boldsymbol{y})$ and $(\boldsymbol{x}{[r]},\boldsymbol{z})$ are perfect, so we have \(|\tilde{\boldsymbol{y}}|, |\tilde{\boldsymbol{z}}| \leq \frac{\alpha_C\sqrt{\ell}}{\sqrt{k}},\) which implies that
\begin{equation}\label{equ:2tilde}
|\langle \tilde{\boldsymbol{y}}, \tilde{\boldsymbol{z}} \rangle| \leq |\tilde{\boldsymbol{y}}||\tilde{\boldsymbol{z}}| = O\left(\frac{\ell}{k}\right)=O\left(\frac{1}{D\sqrt{k}}\right).
\end{equation}
By the unit norm condition (i.e., $|\tilde{\boldsymbol{y}}|^2 + |\hat{\boldsymbol{y}}|^2 =|\boldsymbol{y}|^2 = 1$), we obtain
\( 1 - \frac{\alpha_C^2 \ell}{k} \leq |\hat{\boldsymbol{y}}|^2, |\hat{\boldsymbol{z}}|^2 \leq 1, \)
yielding
\begin{equation}\label{equ:2hat}
\frac{1}{|\hat{\boldsymbol{y}}||\hat{\boldsymbol{z}}|} = 1 + O\left(\frac{\ell}{k}\right)=1+O\left(\frac{1}{D\sqrt{k}}\right).
\end{equation}
Note that all bounds above hold uniformly over all \(\By\) and \(\Bz\).
We can derive from (\ref{equ:2tilde}) and (\ref{equ:2hat}) that
\begin{equation}\label{H001}
H =-\left(\frac{c}{\sqrt{k}} + \langle \tilde{\boldsymbol{y}}, \tilde{\boldsymbol{z}} \rangle\right)\cdot\left(1+O\left(\frac{\ell}{k}\right)\right)= -\frac{c}{\sqrt{k}} - \langle \tilde{\boldsymbol{y}}, \tilde{\boldsymbol{z}} \rangle + O\left(\frac{\ell}{k\sqrt{k}}\right),
\end{equation}
where the term $- \langle \tilde{\boldsymbol{y}}, \tilde{\boldsymbol{z}} \rangle + O\left(\frac{\ell}{k\sqrt{k}}\right)$ is of the order $O\left(\frac{1}{D\sqrt{k}}\right)$.

We note that the random vector $\Bz$ is completely determined by the two random variables $\frac{\hat{\boldsymbol{z}}}{|\hat{\boldsymbol{z}}|}$ and  $\tilde{\boldsymbol{z}}$,
so $\Bz$ can be regarded as the joint distribution of  $\frac{\hat{\boldsymbol{z}}}{|\hat{\boldsymbol{z}}|}$ and  $\tilde{\boldsymbol{z}}$. 
By Lemma~\ref{lem:hat(z)}, 
\(\frac{\hat{\Bz}}{|\hat{\Bz}|}\) is independent of \(\tilde{\Bz}\), and it is uniformly distributed on \(S^{k - r}\).
We can derive the following for any fixed \(\tilde{\boldsymbol{z}}\):
\begin{equation}\label{equ:py}
\begin{aligned}
& \mathbb{P}\left(\boldsymbol{z} \in N(\boldsymbol{y}) \mid \tilde{\boldsymbol{z}}\right) 
=\mathbb{P}\left(\left\langle\frac{\hat{\boldsymbol{z}}}{|\hat{\boldsymbol{z}}|},\frac{\hat{\boldsymbol{y}}}{|\hat{\boldsymbol{y}}|}\right\rangle\leq H \Bigm| \tilde{\boldsymbol{z}}\right)\\
= & p - \sqrt{\frac{k}{2\pi}}e^{-\frac{1}{2}c^2}\cdot \left(\langle \tilde{\boldsymbol{y
}}, \tilde{\boldsymbol{z}} \rangle + O\left(\frac{\ell}{k\sqrt{k}}\right)\right) + O\left(\frac{1}{D^2}\right)
=p - \sqrt{\frac{k}{2\pi}}e^{-\frac{1}{2}c^2}\cdot \langle \tilde{\boldsymbol{y
}}, \tilde{\boldsymbol{z}} \rangle + O\left(\frac{1}{D^2}\right),
\end{aligned}
\end{equation}
where the first equality uses \eqref{equ:Q[r]-H}
and the second equality follows from Lemma~\ref{limrk3}, \eqref{H001} as above, and 
the fact that \(\frac{\hat{\Bz}}{|\hat{\Bz}|}\) remains uniformly distributed on \(S^{k - r}\) when fixing \(\tilde{\Bz}\) (because that \(\frac{\hat{\Bz}}{|\hat{\Bz}|}\) and \(\tilde{\Bz}\) are independent).
We emphasize that the $O\left(\frac{1}{D^2}\right)$ term above holds uniformly over all \(\tilde{\Bz}\). 
Viewing $\Bz$ as a random variable uniformly distributed in \(N_{\mathrm{per}}(\boldsymbol{x}[r])\),
we can derive that
\begin{align*}
Q_{[r]}(\boldsymbol{y}) &= \frac{P_{\mathrm{per}}(\boldsymbol{x}[r],\boldsymbol{y})}{P_{\mathrm{per}}(\boldsymbol{x}[r])}=\frac{\mathbb{P}(N_{\mathrm{per}}(\boldsymbol{x}[r],\boldsymbol{y}))}{\mathbb{P}(N_{\mathrm{per}}(\boldsymbol{x}[r]))} \leq\frac{\mathbb{P}(N_{\mathrm{per}}(\boldsymbol{x}[r])\cap N(\boldsymbol{y}))}{\mathbb{P}(N_{\mathrm{per}}(\boldsymbol{x}[r]))}= \mathbb{P}(\Bz\in N(\By))\\
&=\mathbb{E}_{\tilde{\boldsymbol{z}}}\left[\mathbb{P}(\Bz\in N(\boldsymbol{y}) \mid \tilde{\boldsymbol{z}}) \right] = p - \sqrt{\frac{k}{2\pi}}e^{-\frac{c^2}{2}} \mathbb{E}_{\tilde{\boldsymbol{z}}}[\langle \tilde{\boldsymbol{y}}, \tilde{\boldsymbol{z}} \rangle] + O\left(\frac{1}{D^2}\right),
\end{align*}
where the inequality follows from the monotonicity of perfect sequences (i.e., if \((\Bx[r], \By, \Bz)\) is perfect, then so is \((\Bx[r], \Bz)\)), and the last equality uses \eqref{equ:py}.
Note that \( \mathbb{E}_{\tilde{\Bz}}\bigl[\langle \tilde{\By}, \tilde{\Bz} \rangle\bigr] = \mathbb{E}_{\Bz}\bigl[\langle \tilde{\By}, \tilde{\Bz} \rangle\bigr], \)
where \(\tilde{\Bz}\) depends only on \(\Bz\), and \(\Bz\) is sampled uniformly from \(N_{\mathrm{per}}(\Bx[r])\). This proves \eqref{eq:Qrx}.
\end{proof}

Using similar arguments, we establish a lower bound on the probability of common neighborhoods for perfect sequences.
Recall that \( p \in \bigl(p_C, \tfrac{1}{2}\bigr) \) in the random sphere graph \( G_{k,p}(n) \).

\begin{lemma}\label{volume}
Let \(1\leq r\leq C\ell\), and let \(\boldsymbol{x}{[r]} = (\boldsymbol{x}_1, \ldots, \boldsymbol{x}_r)\in (S^k)^r\) be a perfect sequence. 
Then 
\begin{equation}\label{eq:vol_bounds2}
\frac{\operatorname{Vol}(N_{\mathrm{per}}(\boldsymbol{x}[r]))}{\operatorname{Vol}(N(\Bx[r-1]))} \geq \frac{p_C}{2} 
\quad \text{and} \quad 
\frac{\operatorname{Vol}(\overline{N}_{\mathrm{per}}(\boldsymbol{x}[r]))}{\operatorname{Vol}(\overline{N}(\Bx[r-1]))} \geq  \frac{p_C}{2}.
\end{equation}
In particular, this implies that
\begin{equation}\label{eq:vol_bounds}
\mathbb{P}{(N(\Bx[r]))} \geq \mathbb{P}(N_{\mathrm{per}}(\boldsymbol{x}[r])) \geq \left(\frac{p_C}{2}\right)^r 
\quad \text{and} \quad 
\mathbb{P}{(\overline{N}(\Bx[r]))} \geq \mathbb{P}(\overline{N}_{\mathrm{per}}(\boldsymbol{x}[r])) \geq \left(\frac{p_C}{2}\right)^r.
\end{equation}
\end{lemma}

\begin{proof}
In this proof, we write \(N_s=N(\Bx[s])\) and \(\overline{N}_s=\overline{N}(\Bx[s])\) for a fixed sequence \(\Bx[s]\).  
Recall from Definition \ref{def:projection} that \(\pi_{[s]}\) denotes the orthogonal projection from \(\mathbb{R}^{k+1}\) to \(\operatorname{span}(\boldsymbol{x}[s])\).  
Define
\[
M_s := \left\{\boldsymbol{y}\in S^k: \left|\pi_{[s]}(\boldsymbol{y})\right| \leq \frac{\alpha_C\sqrt{\ell}}{\sqrt{k}}\right\}.
\]
Let \(1\le s\le C\ell\) and \(\Bx[s]\) be a perfect sequence. Then \(N_{\mathrm{per}}(\boldsymbol{x}[s])=N_s \cap M_s\).  
By Lemma \ref{Volumenonperfect},
\begin{equation}\label{eq:vol_estimate}
\frac{\operatorname{Vol}(S^k\setminus M_s)}{\operatorname{Vol}(S^k)}=\mathbb{P}(\boldsymbol{y}\notin M_s)
= \mathbb{P}\left(\left|\pi_{[s]}(\boldsymbol{y})\right|> \frac{\alpha_C\sqrt{\ell}}{\sqrt{k}}\right)
\leq \mathbb{P}\left(\left|\pi_{[s]}(\boldsymbol{y})\right|> \frac{\alpha_C\sqrt{\ell}}{2\sqrt{k}}\right)
\leq \left(\frac{p_C}{10}\right)^{C \ell},
\end{equation}
where \(\By\) is sampled uniformly at random from \(S^k\).  
Then, as long as \(\operatorname{Vol}(N_s)\geq \left(\frac{p_C}{2}\right)^s \operatorname{Vol}(S^k)\),  
\begin{equation}\label{eq:ratio_bound}
\frac{\operatorname{Vol}(N_{\mathrm{per}}(\boldsymbol{x}[s]))}{\operatorname{Vol}(N_s)}
= \frac{\operatorname{Vol}(N_s \cap M_s)}{\operatorname{Vol}(N_s)}
\geq 1-\frac{\operatorname{Vol}(S^k\setminus M_s)}{\operatorname{Vol}(N_s)}
\geq 1 - \frac{\left(p_C/10\right)^{C \ell}}{\left(p_C/2\right)^{s}}
\geq 1 - 2^{-C \ell}.
\end{equation}

We proceed induction on \( r \) to prove the first inequality of \eqref{eq:vol_bounds2}, i.e.,
\( \frac{\operatorname{Vol}\bigl(N_{\mathrm{per}}(\boldsymbol{x}[r])\bigr)}{\operatorname{Vol}(N_{r-1})} \geq \frac{p_C}{2}.\)
The base case $r=1$ holds trivially by \eqref{eq:vol_estimate}. 
Suppose this holds for all integers $s$ up to some $r\geq 1$.
Consider any perfect sequence $\boldsymbol{x}{[r+1]}=(\boldsymbol{x}{[r]},\boldsymbol{x}_{r+1}) \in (S^k)^{r+1}$.
The inductive hypothesis implies that $\frac{\operatorname{Vol}(N_s)}{\operatorname{Vol}(N_{s-1})}\geq \frac{\operatorname{Vol}\bigl(N_{\mathrm{per}}(\boldsymbol{x}[s])\bigr)}{\operatorname{Vol}(N_{s-1})} \geq \frac{p_C}{2}$ for all $s\in [r]$, hence 
\begin{equation}\label{eq:vol_bounds3}
\operatorname{Vol}(N_r)\geq \operatorname{Vol}(N_{\mathrm{per}}(\boldsymbol{x}[r]))\geq \left(\frac{p_C}{2}\right)^r \cdot \operatorname{Vol}(S^k).
\end{equation}
By \eqref{eq:ratio_bound} (with $s=r$), this implies that 
\begin{equation}\label{equ:N_r-S^k}
\frac{\operatorname{Vol}(N_r \cap M_r)}{\operatorname{Vol}(N_r)}\geq 1-2^{-C\ell}.
\end{equation}
Let \(\boldsymbol{y}\) be a vector sampled uniformly at random from \( N_{\mathrm{per}}(\boldsymbol{x}{[r]}) \). 
Define 
\[
\tilde{\boldsymbol{y}} = \pi_{[r]}(\boldsymbol{y})
\mbox{ and } 
\hat{\boldsymbol{y}} = \boldsymbol{y} - \tilde{\boldsymbol{y}};
\quad \mbox{ similarly, define~} 
\tilde{\boldsymbol{x}}_{r+1} = \pi_{[r]}(\boldsymbol{x}_{r+1}) 
\mbox{ and } 
\hat{\boldsymbol{x}}_{r+1} = \boldsymbol{x}_{r+1} - \tilde{\boldsymbol{x}}_{r+1}.
\]
The following arguments are similar to that of Theorem~\ref{ratiored}.
Using the fact $\langle \boldsymbol{y}, \boldsymbol{x}_{r+1} \rangle = \langle \tilde{\boldsymbol{y}}, \tilde{\boldsymbol{x}}_{r+1} \rangle + \langle \hat{\boldsymbol{y}}, \hat{\boldsymbol{x}}_{r+1} \rangle$,
we derive that \( \By\in N_{r+1}\) (which is equivalent to $\By\in N(\Bx_{r+1})$) if and only if
\begin{equation}\label{eq:angle_condition}
\left\langle \frac{\hat{\boldsymbol{y}}}{|\hat{\boldsymbol{y}}|}, \frac{\hat{\boldsymbol{x}}_{r+1}}{|\hat{\boldsymbol{x}}_{r+1}|} \right\rangle \leq \frac{1}{|\hat{\boldsymbol{y}}||\hat{\boldsymbol{x}}_{r+1}|}\left(-\frac{c}{\sqrt{k}} - \langle \tilde{\boldsymbol{y}}, \tilde{\boldsymbol{x}}_{r+1} \rangle\right) =: H.
\end{equation}
Since \(\boldsymbol{y},\boldsymbol{x}_{r+1} \in M_r\), we have \(|\tilde{\boldsymbol{y}}|, |\tilde{\boldsymbol{x}}_{r+1}| \leq \frac{\alpha_C\sqrt{\ell}}{\sqrt{k}}\). 
Following the same bounds as in (\ref{equ:2tilde}) and (\ref{equ:2hat}), 
the right-hand side of \eqref{eq:angle_condition} simplifies to 
\( H = -\frac{c}{\sqrt{k}} + O\left(\frac{1}{D\sqrt{k}}\right).\)
By Lemma~\ref{lem:hat(z)}, 
\(\frac{\hat{\boldsymbol{y}}}{|\hat{\boldsymbol{y}}|}\) is uniformly distributed in \(S^{k-r}\) and is independent of \(\tilde{\boldsymbol{y}}\).
Therefore, for any fixed $\tilde{\boldsymbol{y}}$, by Lemma~\ref{limrk3} we have
\begin{equation}\label{eq:conditional_vol_tildex}
\frac{\operatorname{Vol}(\pi_{[r]}^{-1}(\tilde{\boldsymbol{y}}) \cap N_{r+1})}{\operatorname{Vol}(\pi_{[r]}^{-1}(\tilde{\boldsymbol{y}})\cap S^k)}= \mathbb{P}\left(\left\langle \frac{\hat{\boldsymbol{y}}}{|\hat{\boldsymbol{y}}|}, \frac{\hat{\boldsymbol{x}}_{r+1}}{|\hat{\boldsymbol{x}}_{r+1}|} \right\rangle\leq -\frac{c}{\sqrt{k}}+O\left(\frac{1}{D\sqrt{k}}\right)\right)
= ~ p - O\left(\frac{1}{D}\right) > \frac{2p_C}{3},
\end{equation}
where the last inequality holds since \( p\in (p_C,1/2) \) and \(D\gg C\). 
Using \eqref{eq:conditional_vol_tildex} and integrating over all possible $\tilde{\By}=\pi_{[r]}(\By)\in \pi_{[r]}( N_{\mathrm{per}}(\boldsymbol{x}{[r]}))$, we can show that
\[
\operatorname{Vol}(N_{r+1} \cap M_r)\geq \frac{2p_C}{3} \cdot\operatorname{Vol}(N_r \cap M_r).
\]
This, together with \eqref{equ:N_r-S^k}, implies that
\begin{equation}\label{eq:final_estimate}
\frac{\operatorname{Vol}(N_{r+1})}{\operatorname{Vol}(N_r)} 
\geq \frac{\operatorname{Vol}(N_{r+1} \cap M_r)}{\operatorname{Vol}(N_r)} = \frac{\operatorname{Vol}(N_r \cap M_r)}{\operatorname{Vol}(N_r)}\cdot \frac{\operatorname{Vol}(N_{r+1} \cap M_r)}{\operatorname{Vol}(N_r \cap M_r)} \geq \left(1 - 2^{-C \ell}\right) \frac{2p_C}{3} > \frac{3p_C}{5}
\end{equation}
and \( \operatorname{Vol}(N_{r+1})\geq \left(\frac{p_C}{2}\right)^{r+1} \operatorname{Vol}(S^k) \).
Applying \eqref{eq:ratio_bound} again (with $s=r+1$), and using \eqref{eq:final_estimate}, we have
\[
\frac{\operatorname{Vol}(N_{\mathrm{per}}(\boldsymbol{x}[r+1]))}{\operatorname{Vol}(N_r)}=\frac{\operatorname{Vol}(N_{\mathrm{per}}(\boldsymbol{x}[r+1]))}{\operatorname{Vol}(N_{r+1})}\cdot \frac{\operatorname{Vol}(N_{r+1})}{\operatorname{Vol}(N_r)}\geq \left(1-2^{-C\ell}\right)\cdot \frac{\operatorname{Vol}(N_{r+1})}{\operatorname{Vol}(N_r)} > \frac{p_C}{2}.
\]
This completes the proof of the first inequality in \eqref{eq:vol_bounds2}.  
The second inequality in \eqref{eq:vol_bounds2} follows analogously.\footnote{For the blue case, the sign in \eqref{eq:angle_condition} is reversed, and consequently the estimate in \eqref{eq:conditional_vol_tildex} becomes \(1 - p + O(1/D)\), which remains at least \(p + O(1/D) > \tfrac{2}{3} p_C\); the other arguments remain unchanged.}
As for \eqref{eq:vol_bounds}, its proof proceeds in exactly the same way as that of \eqref{eq:vol_bounds3}.
\end{proof}

\section{From Perfect to General}\label{sec:from-perfect-to-general}
In this section, we show that the ``perfect'' variants (see, e.g., Definition~\ref{def:perfect_probs}) are the principal cases for monochromatic clique probabilities (Definition~\ref{def:prod_redblue}) and their generalizations.
In particular, as a special case of a more general result, we obtain the following.

\begin{theorem}\label{thm:perfect_nonperfect2}
For every $2\leq r\leq C\ell$, we have
\begin{equation}
P_{\mathrm{red},r}\leq \left(1+2^{-C\ell}\right)P_{\mathrm{red},r}^{\mathrm{per}}\quad \mbox{and} \quad \overline{P}_{\mathrm{blue},r}\leq \left(1+2^{-C\ell}\right)\overline{P}_{\mathrm{blue},r}^{\mathrm{per}}.
\end{equation}
\end{theorem}

To facilitate our estimates in the non-perfect cases, we introduce the notion of the \emph{non-perfect number} for sequences of points in $S^k$.

\begin{definition}
For any $r \in \mathbb{N}^*$, we say a sequence $\boldsymbol{x}{[r]} = (\boldsymbol{x}_1, \ldots, \boldsymbol{x}_r)$ of points in $S^k$ has \emph{non-perfect number} $t$, if there exists a subset $J\subseteq \{2,\ldots,r\}$ of size $t$ satisfying the following dichotomy:
\begin{align*}
|\pi_{[i-1]}(\boldsymbol{x}_i)| > \alpha_C\frac{\sqrt\ell}{\sqrt k} & \quad \text{for } i \in J, \\
|\pi_{[i-1]}(\boldsymbol{x}_i)| \leq \alpha_C\frac{\sqrt\ell}{\sqrt k} & \quad \text{for } i \in \{2,\ldots,r\}\setminus J.
\end{align*}
The set \( J \) is called the \emph{non-perfect profile} of $\Bx[r]$.
If \(J=\emptyset\) or \( J = \{r-t + 1, \dots, r\} \) consists of $t\geq 1$ consecutive indices ending at \( r \), then we say the sequence \(\boldsymbol{x}[r]\) is \emph{faithful}.
\end{definition}

Note that a sequence has non-perfect number \( 0 \) if and only if it is a perfect sequence.
The following definition generalizes both Definition~\ref{def:prod_redblue} and Definition~\ref{def:perfect_probs}, and will be used in the next section.

\begin{definition}\label{def:rsperfect}
Let \( 0\leq s<r\leq C\ell\) be integers, and consider a sequence \( \Bx[r] = (\Bx_1, \ldots, \Bx_r) \in S^k \) in which the first \( s \) points form a fixed perfect sequence $\Bx[s]$, while the remaining \( r - s \) points are sampled uniformly and independently from \( S^k \).
Under this setup, we define the following: 
\begin{itemize}
    \item Let \( P_{\mathrm{red},r}(\boldsymbol{x}{[s]}) \) be the probability that \( G_p(\boldsymbol{x}{[r]}) \) forms a red clique, and let \( P_{\mathrm{red},r}^{\mathrm{per}}(\boldsymbol{x}{[s]}) \) be the probability that \( G_p(\boldsymbol{x}{[r]}) \) forms a red clique while \(\boldsymbol{x}{[r]}\) forms a perfect sequence.\footnote{We point out that if the induced subgraph \( G_p(\boldsymbol{x}{[s]}) \) on the given sequence \(\boldsymbol{x}{[s]}\) is not a red clique, then both probabilities defined here equal zero.}
    \item Let $\overline{P}_{\mathrm{blue},r}(\boldsymbol{x}{[s]})$ be the probability that $G_p(\boldsymbol{x}{[r]})$ forms a blue clique, and $\overline{P}_{\mathrm{blue},r}^{\mathrm{per}}(\boldsymbol{x}{[s]})$ be the probability that $G_p(\boldsymbol{x}{[r]})$ forms a blue clique while $\boldsymbol{x}{[r]}$ forms a perfect sequence.
\end{itemize}
\end{definition}

For the case $s=0$, we view $\boldsymbol{x}{[0]}$ as the empty initial condition, and thus we have 
\begin{equation}\label{equ:P(x[0])}
P_{\mathrm{red},r}(\boldsymbol{x}{[0]}) = P_{\mathrm{red},r}
\quad \text{and} \quad 
\overline{P}_{\mathrm{blue},r}(\boldsymbol{x}{[0]}) = \overline{P}_{\mathrm{blue},r}.
\end{equation}

\begin{lemma}\label{lem:prob_nonperfect}
Let \( 0\leq s<r\leq C\ell\) and $\Bx[s]$ be a fixed perfect sequence. Then the following hold:
\[
P_{\mathrm{red},r}(\boldsymbol{x}{[s]})\leq \sum_{t=s}^r \binom{r-s}{r-t}\left(\frac{p_C}{10}\right)^{C\ell(r-t)}P_{\mathrm{red},t}^{\mathrm{per}}(\boldsymbol{x}{[s]}),
\]
\[
\overline{P}_{\mathrm{blue},r}(\boldsymbol{x}{[s]})\leq \sum_{t=s}^r \binom{r-s}{r-t}\left(\frac{p_C}{10}\right)^{C\ell(r-t)}\overline{P}_{\mathrm{blue},t}^{\mathrm{per}}(\boldsymbol{x}{[s]}).
\]
\end{lemma}

\begin{proof}
For any $r\in \mathbb{N}$, we define a map \( \tilde\varphi: (S^k)^r \to (S^k)^r \) as follows.
Consider any $\Bx[r]\in (S^k)^r$. 
Let $J=\{j_1,\ldots, j_{r-t}\}$ denote the non-perfect profile of $\Bx[r]$ for some $1\leq t\leq r$, with $2\leq j_1< \ldots <j_{r-t}\leq r$. 
Let $[r]\setminus J=\{\ell_1,\ldots,\ell_t\}$ with $1=\ell_1<\ldots<\ell_t\leq r$.
Define $\By[r]=\tilde\varphi(\Bx[r])$ by letting
\[
(\By_1,\ldots, \By_r)=\tilde\varphi(\Bx{[r]}) := (\Bx_{\ell_1},\ldots, \Bx_{\ell_t}, \Bx_{j_1},\ldots, \Bx_{j_{r-t}}),
\]
where we also define \(\varphi:[r]\to[r]\) satisfying $\By_i =\boldsymbol{x}_{\varphi(i)}$ for all $i\in [r]$. 
For each $i\in [r]$, define 
\[ 
Y_i:= \operatorname{span}(\By_1,\By_2,\ldots,\By_i)=\operatorname{span}(\Bx_{\varphi(1)}, \Bx_{\varphi(2)}, \ldots, \Bx_{\varphi(i)}).
\]
Then we have the following properties:
\begin{itemize}
    \item[(1).] For each \(2\leq i\leq t\), we have $\varphi(i)=\ell_i\in \{2,\ldots, r\}\setminus J$ and $Y_{i-1} \subseteq \operatorname{span}(\Bx_1, \Bx_2, ...,\Bx_{\varphi(i)-1})$, implying
    \[
    |\pi_{Y_{i-1}}(\By_i)| =|\pi_{Y_{i-1}}(\Bx_{\varphi(i)})| \leq |\pi_{[\varphi(i)-1]}(\boldsymbol{x}_{\varphi(i)})| \leq \alpha_C\frac{\sqrt\ell}{\sqrt k}.\footnote{Here, $\pi_{[q]}(\cdot)$ denotes the projection of points in $S^k$ onto $\operatorname{span}(\Bx_1,\ldots,\Bx_{q})$.}
    \]
    This shows that $(\By_1, \By_2, \ldots, \By_t)$ is a perfect sequence.
    \item[(2).] For each \( t+1\leq i\leq r \), since $\varphi(i)=j_{i-t}\in J$ and $Y_{i-1}\supseteq \operatorname{span}(\Bx_1, \Bx_2, ...,\Bx_{\varphi(i)-1})$, we have
    \[
    |\pi_{Y_{i-1}}(\By_i)| =|\pi_{Y_{i-1}}(\Bx_{\varphi(i)})| \geq |\pi_{[\varphi(i)-1]}(\boldsymbol{x}_{\varphi(i)})| > \alpha_C\frac{\sqrt\ell}{\sqrt k}.
    \]
\end{itemize}
Combining both properties, we see that the non-perfect profile of $\By[r]$ is $J'=\{t+1,t+2,\ldots,r\}$, 
hence $\By[r]=\tilde\varphi(\Bx[r])$ is always a faithful sequence with the same non-perfect number of $\boldsymbol{x}{[r]}$.

We now aim to prove the first inequality of this lemma.
Consider the setting from  Definition~\ref{def:rsperfect}, where $\Bx[s]$ is a fixed perfect sequence, and the remaining $r-s$ points in $\Bx[r]$ are sampled uniformly and independently from $S^k$. 
If \( \boldsymbol{x}{[s]} \) does not induce a red clique, then both sides of the first inequality are zero, and the inequality holds trivially.
So we assume that \( \boldsymbol{x}{[s]} \) induces a red clique $G_p(\Bx[s])$.

Let $r-t$ denote the non-perfect number of $\Bx[r]$. 
Since $\Bx[s]$ is perfect, we have $s\leq t\leq r$, and every element in the non-perfect profile $J=\{j_1<\ldots < j_{r-t}\}$ of $\Bx[r]$ must be strictly greater than \( s \). 
Hence, the number of distinct choices for \( J \) is at most \( \binom{r - s}{r - t} \).
Note that \( \boldsymbol{y}{[r]} = \tilde\varphi(\boldsymbol{x}{[r]}) \) is a faithful sequence with the same non-perfect number $r-t$, obtained from \( \boldsymbol{x}{[r]} \) by reordering the positions of the points. 
It is important to observe that both sequences induce the same graph, i.e., \( G_p(\boldsymbol{x}{[r]})= G_p(\boldsymbol{y}{[r]}) \).
Moreover, the map \( \tilde\varphi \), when restricted to all sequences \( \boldsymbol{x}{[r]} \) with a given non-perfect profile \( J \), is injective and preserves volume measure (because it is an isometric embedding with the standard Euclidean metric).
Therefore, we can derive that the probability that \( G_p\bigl(\boldsymbol{x}{[r]}\bigr) \) forms a red clique with a given non-perfect profile \( J \) is at most the probability \( p^*_{|J|} \) that \( G_p\bigl(\boldsymbol{x}{[r]}\bigr) \) forms a red clique and \( \boldsymbol{x}{[r]} \) is faithful with non-perfect number \( |J| \).
Putting all above together, and applying Lemma~\ref{Volumenonperfect}, we obtain 
\begin{align*}
& P_{\mathrm{red},r}(\boldsymbol{x}{[s]}) = \sum_{t=s}^r \mathbb{P}\bigl(  G_p(\Bx[r]) \mbox{ forms a red clique with } |J|=r-t \bigr) \leq \sum_{t=s}^r \binom{r-s}{r-t} \cdot p^*_{r-t}\\
\leq & \sum_{t=s}^r \binom{r-s}{r-t} \cdot \mathbb{P}\left(
\boldsymbol{x}{[t]} \mbox{ is perfect, } G_p(\boldsymbol{x}{[t]}) \mbox{ is a red clique, and }
|\pi_{[i-1]}(\boldsymbol{x}_i)|>\tfrac{\alpha_C\sqrt{\ell}}{\sqrt{k}}\ \forall t<i\leq r \right) \\
\leq & \sum_{t=s}^r \binom{r-s}{r-t}\left(\frac{p_C}{10}\right)^{C\ell(r-t)}P_{\mathrm{red},t}^{\mathrm{per}}(\boldsymbol{x}{[s]}).
\end{align*}
This proves the first inequality.
The second inequality follows by identical arguments under the analogous condition, so we omit the details.
\end{proof}

Now we are able to prove the main result of this section, as follows.

\begin{theorem}\label{thm:perfect_nonperfect}
Let $2\leq r\leq C\ell$ and $0\leq s\leq r-1$. 
For any fixed perfect sequence $\boldsymbol{x}{[s]}$, 
we have
\begin{equation}\label{equ:P_{red,r}x[s]}
P_{\mathrm{red},r}(\boldsymbol{x}{[s]})\leq \left(1+2^{-C\ell}\right)P_{\mathrm{red},r}^{\mathrm{per}}(\boldsymbol{x}{[s]})
\quad \mbox{and} \quad
\overline{P}_{\mathrm{blue},r}(\boldsymbol{x}{[s]})\leq \left(1+2^{-C\ell}\right)\overline{P}_{\mathrm{blue},r}^{\mathrm{per}}(\boldsymbol{x}{[s]}).
\end{equation}
\end{theorem}

\begin{proof}
We consider the first inequality.  
We may assume that \( \boldsymbol{x}{[s]} \) induces the red clique \( G_p(\boldsymbol{x}{[s]}) \); otherwise, both sides equal zero, and the inequality holds trivially.
For any \( s\leq t <r \),
consider any perfect sequence \( \boldsymbol{x}{[t]} \) whose first \( s \) points form the given perfect sequence \( \boldsymbol{x}{[s]} \).  
Using \eqref{eq:vol_bounds} in Lemma~\ref{volume}, we have
\(
\operatorname{Vol}\bigl(N_{\mathrm{per}}(\boldsymbol{x}{[t]})\bigr)
\geq 
\left( \frac{p_C}{2} \right)^{r} 
\operatorname{Vol}(S^k).
\)
This implies that for any $s\leq t < r$,
\[
P^{\mathrm{per}}_{\mathrm{red},\,t+1}\bigl(\boldsymbol{x}{[s]}\bigr)
\geq 
\inf_{\Bx[t]\ \mathrm{ perfect}}\left(\frac{
    \operatorname{Vol}\bigl(N_{\mathrm{per}}(\boldsymbol{x}{[t]})\bigr)
}{
    \operatorname{Vol}(S^k)
}\right)
\cdot
P^{\mathrm{per}}_{\mathrm{red},\,t}\bigl(\boldsymbol{x}{[s]}\bigr)
\geq 
\left( \frac{p_C}{2} \right)^{r}
\cdot
P^{\mathrm{per}}_{\mathrm{red},\,t}\bigl(\boldsymbol{x}{[s]}\bigr),
\]
and thus
\[
P^{\mathrm{per}}_{\mathrm{red},\,r}\bigl(\boldsymbol{x}{[s]}\bigr)
\geq 
\left( \frac{p_C}{2} \right)^{r(r-t)}
\cdot
P^{\mathrm{per}}_{\mathrm{red},\,t}\bigl(\boldsymbol{x}{[s]}\bigr).
\]
Combining this with Lemma~\ref{lem:prob_nonperfect}, we can then derive
\begin{equation}
\begin{aligned}
P_{\mathrm{red},r}(\boldsymbol{x}{[s]})\leq& P_{\mathrm{red},r}^{\mathrm{per}}(\boldsymbol{x}{[s]})+\sum_{t=s}^{r-1} \binom{r-s}{r-t}\left(\frac{p_C}{10}\right)^{C\ell(r-t)}P_{\mathrm{red},t}^{\mathrm{per}}(\boldsymbol{x}{[s]})\\
\leq& P_{\mathrm{red},r}^{\mathrm{per}}(\boldsymbol{x}{[s]})+\left(\sum_{t=s}^{r-1} r^{r-t}\left(\frac{p_C}{10}\right)^{C\ell(r-t)}\left(\frac{p_C}{2}\right)^{-r(r-t)}\right)\cdot P_{\mathrm{red},r}^{\mathrm{per}}(\boldsymbol{x}{[s]})\\
\leq& \left(1+\sum_{t=s}^{r-1} 4^{-C\ell(r-t)}\right)\cdot P_{\mathrm{red},r}^{\mathrm{per}}(\boldsymbol{x}{[s]})
\leq \left(1+2^{-C\ell}\right)\cdot P_{\mathrm{red},r}^{\mathrm{per}}(\boldsymbol{x}{[s]}),\notag
\end{aligned}
\end{equation}
where the third inequality follows from the analysis below (as \( r \leq C\ell \) and \( \ell \) is sufficiently large)
$$r^{r-t}\left(\frac{p_C}{10}\right)^{C\ell(r-t)}\cdot\left(\frac{p_C}{2}\right)^{-r(r-t)}\leq\left(C\ell\cdot 5^{-C\ell}\right)^{r-t}\leq 4^{-C\ell(r-t)}.$$
This completes the proof of the first inequality of Theorem~\ref{thm:perfect_nonperfect}. 
The proof of the second inequality follows analogously through similar arguments.
\end{proof}

In view of \eqref{equ:P(x[0])}, we see that Theorem~\ref{thm:perfect_nonperfect2} follows directly from the case \( s = 0 \) of Theorem~\ref{thm:perfect_nonperfect}.

\section{Estimates on Key Quantities}\label{sec:key-quan}
This section focuses on deriving a crucial estimate for computing $\kappa_{r}^{\mathrm{per}}$ and $\overline{\kappa}_{r}^{\mathrm{per}}$. 
Let $\Bx[r]$ be a random $r$-tuple in \( (S^k)^r \).
Recall Definitions~\ref{perfectneighborhood} and \ref{def:Qr}, which imply that
\begin{equation}\label{equ:kappa=P_per}
\kappa_r^{\mathrm{per}}=\mathbb{E}\left[P_{\mathrm{per}}(\boldsymbol{x}{[r]})\ |\ A_{\mathrm{red},r}\wedge B_r\right],
\quad \mbox{ where } \quad
P_{\mathrm{per}}(\boldsymbol{x}{[r]})=\prod_{s=0}^{r-1} Q_{[s]}(\Bx_{s+1}),
\end{equation}
and an analogous expression for \(\overline{\kappa}_{r}^{\mathrm{per}}\).
Since our analysis proceeds by fixing each perfect subsequence \(\boldsymbol{x}[s]\) for \(0 \le s < r\), 
the central quantity to estimate, as guided by Theorem~\ref{ratiored}, is
\[
\mathbb{E}\left[\langle \pi_{[s]}(\boldsymbol{x}_{s+1}), \pi_{[s]}(\boldsymbol{z}) \rangle \mid A_{\mathrm{red},r}\wedge B_r\right],
\]
where $\boldsymbol{z}$ is sampled uniformly at random from $N_{\mathrm{per}}(\boldsymbol{x}{[s]})$, independently of $\Bx_{s+1},\ldots, \Bx_r$.

Our goal in this section is to prove the following statement.

\begin{theorem}\label{thm:projection-expectation}
Let \(0 \leq s < r \leq C\ell\), and fix a perfect sequence \(\boldsymbol{x}[s] = (\boldsymbol{x}_1, \ldots, \boldsymbol{x}_s)\). 
Independently sample \(\boldsymbol{x}_{s+1}, \ldots, \boldsymbol{x}_r\) uniformly from \(S^k\), and recall \(A_{\mathrm{red},r}\), \(\overline{A}_{\mathrm{blue},r}\), and \(B_r\) from Definition~\ref{def:perfect_probs}.
\begin{itemize}
\item If \(G_p(\boldsymbol{x}[s])\) forms a red clique, and \(\boldsymbol{z}\) is sampled uniformly at random from \(N_{\mathrm{per}}(\boldsymbol{x}[s])\), independently of \(\boldsymbol{x}_{s+1}, \ldots, \boldsymbol{x}_r\), then
\begin{equation}\label{eq:proj-expectation3}
\mathbb{E}\big[\langle \pi_{[s]}(\boldsymbol{x}_{s+1}), \pi_{[s]}(\boldsymbol{z}) \rangle\mid A_{\mathrm{red},r}\wedge B_r\big]=  \frac{e^{-c^2}}{2\pi p^2} \cdot \frac{s}{k} + O\left(\frac{\ell}{Dk}\right).
\end{equation}
\item If \(G_p(\boldsymbol{x}{[s]})\) forms a blue clique, and \(\boldsymbol{z}\) is sampled uniformly at random from \(\overline{N}_{\mathrm{per}}(\boldsymbol{x}[s])\), independently of \(\boldsymbol{x}_{s+1},\ldots,\boldsymbol{x}_r\), then
\begin{equation}\label{eq:proj-expectation4}
\mathbb{E}\big[\langle \pi_{[s]}(\boldsymbol{x}_{s+1}), \pi_{[s]}(\boldsymbol{z}) \rangle\mid \overline{A}_{\mathrm{blue},r}\wedge B_r\big] = \frac{e^{-c^2}}{2\pi (1-p)^2} \cdot \frac{s}{k} + O\left(\frac{\ell}{Dk}\right).
\end{equation}
\end{itemize}
\end{theorem}

The remainder of this section is devoted to proving this result, divided into three subsections.

\subsection{Spectral Properties of Perfect Sequences}\label{subsec:eigenvalues}
To facilitate the analysis of the desired expectations, we employ spectral arguments to study relevant vectors associated with perfect sequences. 
We begin with the definitions of these concepts. 

\begin{definition}\label{def:eigenvalues}
Let \(1 \leq r \leq C\ell\) and \( \Bx[r]=(\Bx_1,\ldots, \Bx_r)\) be a perfect sequence in \( (S^k)^r\subseteq (\mathbb{R}^{k+1})^r \). 
\begin{itemize}
    \item Let \(\mathbf{X} \in \mathbb{R}^{(k+1)\times r}\) be the matrix whose \( i^{\text{th}} \) column is the vector \(\boldsymbol{x}_i\) for each \( i \in [r]\).\footnote{We will view all vectors in \( \mathbb{R}^{k+1}\) as column vectors.}
    \item For each \( i \in [r] \), let \(\boldsymbol{v}_i \in \operatorname{span}(\boldsymbol{x}_1, \ldots, \boldsymbol{x}_r)\) be the unique vector satisfying \(\langle \boldsymbol{v}_i, \boldsymbol{x}_i \rangle = 1\) and orthogonal to \(\operatorname{span}\bigl(\{\boldsymbol{x}_j : j \in [r] \setminus \{i\}\}\bigr)\); see Figures~\ref{fig2} and \ref{fig3}.  
    Let \(\mathbf{V} \in \mathbb{R}^{(k+1)\times r}\) denote the matrix whose \( i^{\text{th}} \) column is the vector \(\boldsymbol{v}_i\) for each \( i \in [r] \). Equivalently, \(\mathbf{V} = \mathbf{X}(\mathbf{X}^T\mathbf{X})^{-1}\), i.e., \(\mathbf{V}\) is the Moore–Penrose pseudoinverse of \(\mathbf{X}^T\).
    \item For each \( i \in [r] \), let \( X_i := \operatorname{span}(\boldsymbol{x}_1, \ldots, \boldsymbol{x}_i) \), and define \( X_0 := \{\boldsymbol{0}\} \). Construct the orthonormal basis \(\{\boldsymbol{e}_i\}_{i=1}^r\) for the space \( X_r \), where for each \( i \in [r] \), the vector \(\boldsymbol{e}_i \in X_i \cap X_{i-1}^\perp\) is the unique unit vector satisfying \(\langle \boldsymbol{x}_i, \boldsymbol{e}_i \rangle > 0\).
    \item Let \(\lambda_1 \ge \cdots \ge \lambda_r \ge 0\) denote the eigenvalues of \(\mathbf{X}^T \mathbf{X}\), and let \(0 \le \mu_1 \le \cdots \le \mu_r\) denote the eigenvalues of \(\mathbf{V}^T \mathbf{V}\).
\end{itemize}
\end{definition}

We first show that the common neighborhoods \(N(\Bx[r])\) and \(\overline{N}(\Bx[r])\) can be characterized as a corner generated by the previously defined vectors \(\Bv_1,\ldots,\Bv_r\).

\begin{proposition}\label{prop:common_neighborhood}
Let \(\Bx[r]=(\boldsymbol{x}_1,\ldots,\boldsymbol{x}_r)\) be a perfect sequence.
Let \(\mathbf{X}\), \(\boldsymbol{v}_1,\ldots,\boldsymbol{v}_r\), and \(\mathbf{V}\) be defined as in Definition~\ref{def:eigenvalues}. 
Then we have
\begin{align*}
N(\boldsymbol{x}{[r]}) &= \left\{\boldsymbol{y} \in S^k : \pi_{[r]}(\boldsymbol{y}) = \sum_{i=1}^r a_i\boldsymbol{v}_i \text{ with } \ a_i \leq -\frac{c}{\sqrt{k}} \text{ for all }  i \in [r]\right\}, \\
\overline{N}(\boldsymbol{x}{[r]}) &= \left\{\boldsymbol{y} \in S^k : \pi_{[r]}(\boldsymbol{y}) = \sum_{i=1}^r a_i\boldsymbol{v}_i \text{ with } \ a_i > -\frac{c}{\sqrt{k}} \text{ for all }  i \in [r]\right\}.
\end{align*}
\end{proposition}

\begin{proof}
Consider the projection \(\pi_{[r]}(\boldsymbol{y}) = \sum_{i=1}^r a_i \boldsymbol{v}_i\) onto the space $X_r$. 
Then we have
\begin{equation}\label{equ:yx_i=a_i}
\begin{aligned}
(\langle \boldsymbol{y}, \boldsymbol{x}_1 \rangle, \ldots, \langle \boldsymbol{y}, \boldsymbol{x}_r \rangle) &= \boldsymbol{y}^T \mathbf{X} = (\pi_{[r]}(\boldsymbol{y}))^T \cdot \mathbf{X} \\
&= (a_1, \ldots, a_r) \mathbf{V}^T \cdot \mathbf{X}  = (a_1, \ldots, a_r) \cdot (\mathbf{X}^T\mathbf{X})^{-1} \mathbf{X}^T \mathbf{X} = (a_1, \ldots, a_r).
\end{aligned}
\end{equation}
By Definition~\ref{def:sphere_neighborhood}, 
we obtain that \(\boldsymbol{y} \in N(\boldsymbol{x}{[r]})\) if and only if \(a_i=\langle \boldsymbol{y}, \boldsymbol{x}_i \rangle \leq -\frac{c}{\sqrt{k}}\) for all \(i\in [r]\);
similarly, \(\boldsymbol{y} \in \overline{N}(\boldsymbol{x}{[r]})\) if and only if \(a_i=\langle \boldsymbol{y}, \boldsymbol{x}_i \rangle > -\frac{c}{\sqrt{k}}\) for all \(i\in [r]\),
finishing the proof.
\end{proof}

The following lemma will be useful in the coming proof. 
It offers multiple handy and broadly applicable formulas for various vectors under the perfect-sequence condition.

\begin{lemma}\label{lem:eigenvalues}
Let \(\Bx[r]=(\boldsymbol{x}_1,\ldots,\boldsymbol{x}_r)\) be a perfect sequence.
Let \(\mathbf{X}\), \(\boldsymbol{v}_1,\ldots,\boldsymbol{v}_r\), \(\mathbf{V}\), \(\lambda_1 \ge \cdots \ge \lambda_r\), and \(\mu_1 \le \cdots \le \mu_r\)  be defined as in Definition~\ref{def:eigenvalues}. Then the following hold:
\begin{itemize}
\item[(1)] For all $i \in [r]$, we have $\mu_i = \lambda_i^{-1}$ with $\lambda_i, \mu_i = 1 + O\left(\frac{1}{D}\right)$, and $|\boldsymbol{v}_i| = 1 + O\left(\frac{1}{D}\right)$.

\item[(2)] For a matrix $A=(a_{ij})_{m\times n}$,  let \( \| A \|_F = \sqrt{\sum_{i=1}^m \sum_{j=1}^n a_{ij}^2}\) denote its {\it Frobenius norm}. Then 
\begin{align*}
\sum_{i=1}^r (\lambda_i - 1)^2 = \|\mathbf{X}^T\mathbf{X} - \mathbf{I}\|_F^2 = O\left(\frac{1}{D^2}\right), \\
\sum_{i=1}^r (\mu_i - 1)^2 = \|\mathbf{V}^T\mathbf{V} - \mathbf{I}\|_F^2 = O\left(\frac{1}{D^2}\right).
\end{align*}
In particular, for each $i \in [r]$, let $V_r(i) = \operatorname{span}\bigl(\{\boldsymbol{v}_j : j \in [r] \setminus \{i\}\}\bigr)$. 
Then we have
\[
\sum_{i=1}^r |\pi_{V_r(i)}(\boldsymbol{v}_i)|^2 = O\left(\frac{1}{D^2}\right).
\]
\item[(3)] Let \(\boldsymbol{y} \in S^k\) be a vector such that \((\boldsymbol{x}[r], \boldsymbol{y})\) is perfect, and let \(\pi_{[r]}(\boldsymbol{y}) = \sum_{i=1}^r a_i \boldsymbol{v}_i\). Then $\sum_{i=1}^r a_i^2=O\left(\frac{\ell}{k}\right)$ and for any $I\subseteq [r]$,
\[
\sum_{i\in I}a_i^2 = \left(1+O\left(\frac{1}{D}\right)\right)\cdot \left|\sum_{i\in I} a_i\boldsymbol{v}_i\right|^2.
\]
\end{itemize}
\end{lemma}

\begin{proof}
Let $\boldsymbol{x}[r]=(\boldsymbol{x}_1,\ldots,\boldsymbol{x}_r)$ be a fixed perfect sequence.
For item (1), we first observe that
\begin{align*}
\mathbf{V}^T\mathbf{V} &= (\mathbf{X}^T\mathbf{X})^{-1}\mathbf{X}^T\mathbf{X}(\mathbf{X}^T\mathbf{X})^{-1}= (\mathbf{X}^T\mathbf{X})^{-1}.
\end{align*}
This proves that \(\mu_i = \lambda_i^{-1} \) for all $i \in [r]$.
Note that $\By^{T} \mathbf{X}\mathbf{X}^T \By=\sum_{i=1}^r \langle \boldsymbol{y}, \boldsymbol{x}_i \rangle^2$, and $\mathbf{X}^T\mathbf{X}$ and $\mathbf{X}\mathbf{X}^T$ share identical nonzero eigenvalues. 
By the Courant-Fischer theorem, we have
\begin{align}
\lambda_1 &= \sup_{\substack{\boldsymbol{y} \in X_r \\ |\boldsymbol{y}|=1}} \sum_{i=1}^r \langle \boldsymbol{y}, \boldsymbol{x}_i \rangle^2 
\qquad \mbox{and} \qquad
\lambda_r = \inf_{\substack{\boldsymbol{y} \in X_r \\ |\boldsymbol{y}|=1}} \sum_{i=1}^r \langle \boldsymbol{y}, \boldsymbol{x}_i \rangle^2. \label{eq:max_eigenvalue}
\end{align}
Recall the orthonormal basis $\{\boldsymbol{e}_i\}_{i=1}^r$ defined in Definition \ref{def:eigenvalues}.  
For all $1\leq i, j\leq r$, we have
\begin{equation}\label{equ:projx}
\langle \boldsymbol{x}_i, \boldsymbol{e}_j \rangle^2 = 
\begin{cases}
O\left(\frac{\ell}{k}\right), & 1 \leq j < i, \\
1 - O\left(\frac{\ell}{k}\right), & j = i, \\
0, & i<j\leq r.
\end{cases}
\end{equation}
Consider any unit vector $\boldsymbol{y} = \sum_{i=1}^r b_i \boldsymbol{e}_i\in X_r$ (with $\sum_{i=1}^r b_i^2 = 1$). 
It follows that
\begin{align*}
\sum_{i=1}^r \langle \boldsymbol{y}, \boldsymbol{x}_i \rangle^2 
&= \sum_{i=1}^r \left(\sum_{j=1}^r b_j \langle \boldsymbol{e}_j, \boldsymbol{x}_i \rangle\right)^2 =\sum_{i=1}^r\left(b_i\langle\boldsymbol{e}_i,\boldsymbol{x}_i\rangle+\left\langle \boldsymbol{x}_i,\sum_{j=1}^{i-1}b_j\boldsymbol{e}_j\right\rangle \right)^2\\
&=\underbrace{\sum_{i=1}^r b_i^2\langle\boldsymbol{e}_i,\boldsymbol{x}_i\rangle^2}_{(\mathrm{I})} + \underbrace{2\sum_{i=1}^{r}b_i\langle\boldsymbol{e}_i,\boldsymbol{x}_i\rangle \left\langle \boldsymbol{x}_i,\sum_{j=1}^{i-1}b_j\boldsymbol{e}_j\right\rangle}_{(\mathrm{II})} + \underbrace{\sum_{i=1}^r \left\langle \boldsymbol{x}_i,\sum_{j=1}^{i-1}b_j\boldsymbol{e}_j\right\rangle^2}_{(\mathrm{III})}\\
&= \left(1+O\left(\frac{1}{D}\right)\right)\cdot \sum_{i=1}^r b_i^2\langle\boldsymbol{e}_i,\boldsymbol{x}_i\rangle^2 + \left(1+O(D)\right)\cdot \sum_{i=1}^r \left\langle \boldsymbol{x}_i,\sum_{j=1}^{i-1}b_j\boldsymbol{e}_j\right\rangle^2,
\end{align*}
where the last inequality holds because the absolute value of \(\mathrm{(II)}\) is at most \( \frac{1}{D}\,\mathrm{(I)} + D\,\mathrm{(III)}\), by the AM–GM inequality (i.e., \(2ab \le \frac{1}{D}a^2 + D b^2\)).
Using \eqref{equ:projx} and the identity \(\sum_{i=1}^r b_i^2 = 1\), 
we have \( \sum_{i=1}^r b_i^2\langle\boldsymbol{e}_i,\boldsymbol{x}_i\rangle^2 = 1 + O(\frac{\ell}{k}) \).
Moreover, since \(\sum_{j=1}^{i-1} b_j\boldsymbol{e}_j \in X_{i-1}\) and \(\sum_{j=1}^{i-1} b_j^2 \leq 1\),  we can get
\[
\left\langle\boldsymbol{x}_i,\sum_{j=1}^{i-1}b_j\boldsymbol{e}_j\right\rangle^2 = \left\langle \pi_{[i-1]}(\boldsymbol{x}_i),\sum_{j=1}^{i-1}b_j\boldsymbol{e}_j\right\rangle^2
\leq|\pi_{[i-1]}(\boldsymbol{x}_i)|^2\cdot \left|\sum_{j=1}^{i-1} b_j\boldsymbol{e}_j\right|^2 \leq \alpha_C^2 {\frac{\ell}{ k}}\cdot\left(\sum_{j=1}^{i-1} b_j^2\right)\leq \alpha_C^2 {\frac{\ell}{ k}},
\] 
where the second last inequality follows from the fact that $\Bx[r]$ is a perfect sequence.
Combining all above bounds, and noting that \(r \le C \ell\) and \(k = D^2 \ell^2\), we obtain that for any unit vector $\By\in X_r$,
\[
\sum_{i=1}^r \langle \boldsymbol{y}, \boldsymbol{x}_i \rangle^2 = \left(1+O\left(\frac{1}{D}\right)\right)\cdot \left(1 + O\left(\frac{\ell}{k}\right)\right)+(1+O(D))\cdot O\left(\frac{r\ell}{k}\right)=1+O\left(\frac{1}{D}\right).
\]
This, together with \eqref{eq:max_eigenvalue}, implies that for every $i\in [r]$,
\( \lambda_i = 1 + O\left(\frac{1}{D}\right) \) and thus \( \mu_i=\lambda_i^{-1} = 1 + O\left(\frac{1}{D}\right) \).
We apply the Courant-Fischer theorem for the matrix \( \mathbf{V}^T\mathbf{V} \):
for any unit vector $\By\in \mathbb{R}^r$, we have $\mu_1\leq \By^{T} \mathbf{V}^T\mathbf{V} \By\leq \mu_r$.
For each $i\in [r]$,
by taking $\By$ to be the unit vector $\By_i\in \mathbb{R}^r$, where the $i^{\text{th}}$ entry is 1,
we derive $|\boldsymbol{v}_i|^2=\By_i^{T} \mathbf{V}^T\mathbf{V} \By_i=1+O\left(\frac{1}{D}\right)$ and thus $|\boldsymbol{v}_i|=1+O\left(\frac{1}{D}\right)$, proving item (1).

For any \(I \subseteq [r]\), let \(\mathbf{V}_I \in \mathbb{R}^{(k+1)\times |I|}\) denote the matrix whose columns are the vectors \(\boldsymbol{v}_i\) for all \(i \in I\); define \(\mathbf{X}_I\) analogously.
By Cauchy’s eigenvalue interlacing theorem, it is easy to obtain that

\medskip

{\noindent \bf Property \((\star)\).} all \(|I|\) eigenvalues of \(\mathbf{V}_I^T \mathbf{V}_I\) (respectively, \(\mathbf{X}_I^T \mathbf{X}_I\)) are equal to \(1 + O\bigl(\frac{1}{D}\bigr)\), and the same holds for the nonzero eigenvalues of \(\mathbf{V}_I \mathbf{V}_I^T\) (respectively, \(\mathbf{X}_I \mathbf{X}_I^T\)).

\medskip

Now we prove item (2). 
Fix $s\in [r-1]$ and let $\boldsymbol{z}_s := \pi_{[s]}(\boldsymbol{x}_{s+1})$. 
We have
\begin{equation}\label{equ:proj_xs+1}
\sum_{i=1}^s\langle \boldsymbol{x}_{s+1}, \boldsymbol{x}_i \rangle^2=\sum_{i=1}^s \langle \boldsymbol{z}_s, \boldsymbol{x}_i \rangle^2=\boldsymbol{z}_s^T\mathbf{X}_{[s]}\mathbf{X}_{[s]}^T\boldsymbol{z}_{s}
= \left(1 + O\left(\frac{1}{D}\right)\right)\cdot |\boldsymbol{z}_s|^2
=O\left(\frac{\ell}{k}\right), 
\end{equation}
where the second last equality follows from Property \((\star)\) above and the Courant–Fischer theorem, and the final equality holds because \(\boldsymbol{x}[r]\) is perfect.
Since \(\lvert \boldsymbol{x}_i \rvert = 1\) for all \(1 \le i \le r\), it follows from \eqref{equ:proj_xs+1} that the following Frobenius norm satisfies
\begin{equation*}
    \sum_{i=1}^r (\lambda_i - 1)^2=\|\mathbf{X}^T\mathbf{X} - \mathbf{I}\|_F^2 = \sum_{\substack{1 \leq i,j \leq r \\ i \neq j}} \langle \boldsymbol{x}_i, \boldsymbol{x}_j \rangle^2 = 2\sum_{j=1}^r \sum_{1 \leq i < j} \langle \boldsymbol{x}_i, \boldsymbol{x}_j \rangle^2 = O\left(\frac{\ell^2}{k}\right) = O\left(\frac{1}{D^2}\right). 
\end{equation*}
Using the fact that \(\mu_i^{-1} = \lambda_i = 1 + O\left(\frac{1}{D}\right)\) for each \(i \in [r]\), we deduce from the above equation that
\begin{equation*}
\|\mathbf{V}^T\mathbf{V} - \mathbf{I}\|_F^2 = \sum_{i=1}^r (\mu_i - 1)^2 = \sum_{i=1}^r \frac{(\lambda_i - 1)^2}{\lambda_i^2} = O\left(\frac{1}{D^2}\right). \label{eq:dual_frob_bound}
\end{equation*}
Next, consider the projections $\boldsymbol{u}_i := \pi_{V_r(i)}(\boldsymbol{v}_i)$, 
where $V_r(i) = \operatorname{span}\bigl(\{\boldsymbol{v}_j : j \in [r] \setminus \{i\}\}\bigr)$.
Using the same discussion as in \eqref{equ:proj_xs+1}, we can obtain that for any $i\in [r]$, 
\[
\sum_{j\in [r]\backslash \{i\}} \langle \boldsymbol{v}_i, \boldsymbol{v}_j \rangle^2 = \boldsymbol{u}_i^T\mathbf{V}_{[r]\backslash \{i\}}\mathbf{V}_{[r]\backslash \{i\}}^T \boldsymbol{u}_{i}
= \left(1 + O\left(\frac{1}{D}\right)\right)\cdot |\boldsymbol{u}_i|^2.
\]
This implies the final desired equality of item (2) as follows:
\begin{align*}
\sum_{i=1}^r |\pi_{V_r(i)}(\boldsymbol{v}_i)|^2=\left(1 + O\left(\frac{1}{D}\right)\right)\cdot \sum_{1\leq i\neq j\leq r }\langle \boldsymbol{v}_i, \boldsymbol{v}_j \rangle^2\leq \left(1 + O\left(\frac{1}{D}\right)\right)\cdot \|\mathbf{V}^T\mathbf{V} - \mathbf{I}\|_F^2 = O\left(\frac{1}{D^2}\right). 
\end{align*}

It remains to prove item (3). 
Consider any \(\boldsymbol{y} \in S^k\) such that \((\boldsymbol{x}[r], \boldsymbol{y})\) is perfect.
Let \(\pi_{[r]}(\boldsymbol{y}) = \sum_{i=1}^r a_i \boldsymbol{v}_i\). 
Let $\boldsymbol{a}=(a_1,\ldots,a_r)^T$. 
Then by the same arguments as above, we have 
\[
\left(1+O\left(\frac{1}{D}\right)\right)\cdot\boldsymbol{a}^T \boldsymbol{a}=\boldsymbol{a}^T \mathbf{V}^T\mathbf{V}\boldsymbol{a}= \left|\sum_{i=1}^r a_i\Bv_i\right|^2=\left|\pi_{[r]}(\boldsymbol{y})\right|^2=O\left(\frac{\ell}{k}\right),
\]
where the last equality follows by the perfectness of the sequence \((\boldsymbol{x}[r], \boldsymbol{y})\). 
This implies  $\sum_{i=1}^r a_i^2=O\left(\frac{\ell}{k}\right)$.
For any $I\subseteq [r]$, by replacing \(\boldsymbol{a}\) and \(\mathbf{V}\) with the vector \(\boldsymbol{a}_I = (a_i : i \in I)^T\) and the matrix \(\mathbf{V}_I\) in the leftmost equality above, it follows that
\(
\sum_{i \in I} a_i^2 
= \left( 1 + O\left( \frac{1}{D} \right) \right) 
\cdot \left\lvert \sum_{i \in I} a_i \boldsymbol{v}_i \right\rvert^2.
\)
This completes the proof of Lemma~\ref{lem:eigenvalues}. 
\end{proof}

We conclude this subsection with the next lemma,
which provides an estimate of the difference between the orthonormal basis \(\Be_1,\ldots, \Be_r\) and \(\Bv_1,\ldots, \Bv_r\). 

\begin{lemma}\label{lem:dist_e_v}
Let \(\Bx[r]=(\boldsymbol{x}_1,\ldots,\boldsymbol{x}_r)\) be a perfect sequence.
Let 
\(\mathbf{X}\), \(\boldsymbol{e}_1,\ldots,\boldsymbol{e}_r\), \(\mathbf{V}\), and  \(\boldsymbol{v}_1,\ldots,\boldsymbol{v}_r\) be defined as in Definition~\ref{def:eigenvalues}. 
Let \(V_r(i) = \operatorname{span}\bigl(\{\boldsymbol{v}_j : j \in [r] \setminus \{i\}\}\bigr) \). Then for all \(i \in [r] \), we have
\[
\langle \boldsymbol{v}_i, \boldsymbol{e}_i \rangle = 1 + O\left(\frac{\ell}{k}\right
)\quad\text{and}\quad\left| \pi_{V_r(i)}(\boldsymbol{e}_i) \right| = O\left(\frac{\sqrt{\ell}}{\sqrt{k}}\right
).
\]
\end{lemma}

\begin{proof}
Recall that for each \(i \in [r]\), \(\boldsymbol{e}_i\) is the unit vector in \(X_i \cap X_{i-1}^\perp\) satisfying \(\langle \boldsymbol{x}_i, \boldsymbol{e}_i \rangle >0\). So 
\begin{equation}\label{eq:xi_for_ei}
\boldsymbol{e}_i = \frac{\boldsymbol{x}_i - \pi_{[i-1]}(\boldsymbol{x}_i)}{\left| \boldsymbol{x}_i - \pi_{[i-1]}(\boldsymbol{x}_i) \right|
}.
\end{equation}
By Definition~\ref{def:eigenvalues},
we have \(\langle\boldsymbol{v}_i,\boldsymbol{x}_j\rangle=0\) for all \(j\neq i\), 
thus by \eqref{eq:xi_for_ei}
\begin{align*}
1 = \langle \boldsymbol{v}_i, \boldsymbol{x}_i \rangle = \langle \boldsymbol{v}_i, \boldsymbol{x}_i - \pi_{[i-1]}(\boldsymbol{x}_i) \rangle + \langle \boldsymbol{v}_i, \pi_{[i-1]}(\boldsymbol{x}_i) \rangle = \left| \boldsymbol{x}_i - \pi_{[i-1]}(\boldsymbol{x}_i) \right| \cdot \langle \boldsymbol{v}_i, \boldsymbol{e}_i \rangle.
\end{align*}
Then using the fact that $\Bx[r]$ is a perfect sequence, we can derive
\[
\langle \boldsymbol{v}_i, \boldsymbol{e}_i \rangle = \frac{1}{\left| \boldsymbol{x}_i - \pi_{[i-1]}(\boldsymbol{x}_i) \right|}
= \frac{1}{\sqrt{1 - \left| \pi_{[i-1]}(\boldsymbol{x}_i) \right|^2}} = 
\frac{1}{\sqrt{1 - O\left({\ell}/{k}\right)}} = 1 + O\left(\frac{\ell}{k}\right).
\]
By the property \(\langle \boldsymbol{v}_j, \boldsymbol{x}_i \rangle = 0\) for all \(j \neq i\), we have \(\pi_{V_r(i)}(\boldsymbol{x}_i) = \boldsymbol{0}\).
This, together with  \eqref{eq:xi_for_ei}, implies
\begin{align*}
\left| \pi_{V_r(i)}(\boldsymbol{e}_i) \right| 
&= 
\left| \pi_{V_r(i)} \left( \frac{\boldsymbol{x}_i - \pi_{[i-1]}(\boldsymbol{x}_i)}{\left| \boldsymbol{x}_i - \pi_{[i-1]}(\boldsymbol{x}_i) \right|} \right) \right| = 
\frac{1}{\left| \boldsymbol{x}_i - \pi_{[i-1]}(\boldsymbol{x}_i) \right|} \left| \pi_{V_r(i)} \left( -\pi_{[i-1]}(\boldsymbol{x}_i) \right) \right| \\
&
\leq \frac{ \left| \pi_{[i-1]}(\boldsymbol{x}_i) \right| }{ \left| \boldsymbol{x}_i - \pi_{[i-1]}(\boldsymbol{x}_i) \right| } 
\leq \left(1 + O\left(\frac{\ell}{k}\right)\right) \cdot O\left(\frac{\sqrt{\ell}}{\sqrt{k}}\right) = 
O\left(\frac{\sqrt{\ell}}{\sqrt{k}}\right),
\end{align*}
completing the proof. \end{proof}

\subsection{Estimate of \(\mathbb{E}[\langle\pi_{[r]}(\boldsymbol{y}),\boldsymbol{e}_i\rangle]\):  Single Random Projection}\label{sec:estimation_of_xy}

In this subsection, we present the core technical step in proving Theorem~\ref{thm:projection-expectation}, as follows. 

\begin{lemma}\label{lem:key-tech}
Let \(1 \leq r \leq C\ell\), and let \(\boldsymbol{x}{[r]} = (\boldsymbol{x}_1,\ldots,\boldsymbol{x}_r)\) be a perfect sequence. 
Let \(\boldsymbol{v}_1,\ldots,\boldsymbol{v}_r\), \( \Be_1,\ldots, \Be_r\) and \(V_r(s) = \operatorname{span}\bigl(\{\boldsymbol{v}_i : i \in [r] \setminus \{s\}\}\bigr) \) for $s\in [r]$ be defined as in Definition~\ref{def:eigenvalues}.
If \(\boldsymbol{y}\) is sampled uniformly at random from  \(N_{\mathrm{per}}(\boldsymbol{x}{[r]})\),
then for each $s\in [r]$,
\begin{equation}\label{eq:expection2}
\mathbb{E}[\langle\pi_{[r]}(\boldsymbol{y}),\boldsymbol{e}_s\rangle]= -\left(1 + O\left(\frac{1}{D}\right)\right)\cdot\frac{e^{-c^2/2}}{p\sqrt{2\pi k}} + O\left(\frac{\sqrt{\ell}}{\sqrt{k}}|\pi_{V_r(s)}(\boldsymbol{v}_s)|\right).
\end{equation}
Analogously, if \(\boldsymbol{y}\) is sampled uniformly at random from \(\overline{N}_{\mathrm{per}}(\boldsymbol{x}{[r]})\), then for each \( s\in [r]\), 
\begin{equation}\label{eq:expection2b}
\mathbb{E}[\langle\pi_{[r]}(\boldsymbol{y}),\boldsymbol{e}_s\rangle]= \left(1 + O\left(\frac{1}{D}\right)\right)\cdot\frac{e^{-c^2/2}}{(1-p)\sqrt{2\pi k}} + O\left(\frac{\sqrt{\ell}}{\sqrt{k}}|\pi_{V_r(s)}(\boldsymbol{v}_s)|\right).
\end{equation}
\end{lemma}

The essential part of this lemma is formulated in the next lemma, which we prove first.

\begin{lemma}\label{lem:coefficient_expectations}
Let \( \Bx[r],\boldsymbol{v}_1,\ldots,\boldsymbol{v}_r\) and \( V_r(s) \) be given as in Lemma~\ref{lem:key-tech}.
Consider the vector \(\boldsymbol{y}\) sampled uniformly at random from  \(N_{\mathrm{per}}(\boldsymbol{x}{[r]})\), with \(\pi_{[r]}(\boldsymbol{y}) = \sum_{i=1}^r a_i \boldsymbol{v}_i\).
Then for each \( s\in [r]\), 
\begin{equation}\label{equ:Eai}
\mathbb{E}[a_s] = -\left(1 + O\left(\frac{1}{D}\right)\right)\cdot\frac{e^{-c^2/2}}{p \sqrt{2\pi k}} + O\left(\frac{\sqrt{\ell}}{\sqrt{k}}|\pi_{V_r(s)}(\boldsymbol{v}_s)|\right).
\end{equation}

Analogously, consider the vector \(\boldsymbol{y}\) sampled uniformly at random from \(\overline{N}_{\mathrm{per}}(\boldsymbol{x}{[r]})\), with \(\pi_{[r]}(\boldsymbol{y}) = \sum_{i=1}^r a_i \boldsymbol{v}_i\). Then for each \( s\in [r]\), 
\begin{equation}\label{equ:Eai3}
\mathbb{E}[a_s] = \left(1 + O\left(\frac{1}{D}\right)\right)\cdot \frac{e^{-c^2/2}}{(1-p)\sqrt{2\pi k}} + O\left(\frac{\sqrt{\ell}}{\sqrt{k}}|\pi_{V_r(s)}(\boldsymbol{v}_s)|\right).
\end{equation}
\end{lemma}

\begin{proof}
Fix a perfect sequence \(\boldsymbol{x}{[r]} = (\boldsymbol{x}_1,\ldots,\boldsymbol{x}_r) \) and 
let \(\boldsymbol{v}_1,\ldots,\boldsymbol{v}_r\) be defined from \(\boldsymbol{x}{[r]}\) as in Definition~\ref{def:eigenvalues}.
Let \(\boldsymbol{y}\) be the vector sampled uniformly at random from \(N_{\mathrm{per}}(\boldsymbol{x}{[r]})\), and let \(\pi_{[r]}(\boldsymbol{y}) = \sum_{i=1}^r a_i \boldsymbol{v}_i\).
From now on, we regard \(\boldsymbol{y}\) as the joint distribution of the random variables \(a_i\) for \(i \in [r]\).
We fix an index \(s \in [r]\), and aim to prove \eqref{equ:Eai} for this choice of \(s\). 
Write \( \boldsymbol{\omega}_{\overline{s}} \triangleq \bigl(a_1,\ldots,a_{s-1},a_{s+1},\ldots, a_r\bigr) \) and
define the event \(\mathcal{A}\triangleq \left\{ \boldsymbol{\omega}_{\overline{s}}: \left|\sum_{i \in [r]\backslash \{s\} } a_i \boldsymbol{v}_i\right| \leq (\alpha_C - 2)\frac{\sqrt{\ell}}{\sqrt{k}}\right\}.\)

We first claim that to prove \eqref{equ:Eai}, it suffices to show the following for the conditional expectation
\begin{equation}\label{equ:E(a_s)}
\mathbb{E}\bigl[ a_s \mid \boldsymbol{\omega}_{\overline{s}} \bigr]
=
-\frac{e^{-c^2/2}}{p\sqrt{2\pi k}}
\left( 1 + O\left(\frac{1}{D}\right) \right)
+ O\left( \frac{\sqrt{\ell}}{\sqrt{k}} \,\bigl| \pi_{V_r(s)}(\boldsymbol{v}_s) \bigr| \right),
\end{equation}
where \( \boldsymbol{\omega}_{\overline{s}}\in \mathcal{A} \) denotes an arbitrary but fixed outcome \( \bigl(a_1,\ldots,a_{s-1},a_{s+1},\ldots, a_r\bigr)\) satisfying the event \(\mathcal{A}\).
To see this, we begin by estimating the probability \(\mathbb{P}(\mathcal{A})\).
By item (3) of Lemma~\ref{lem:eigenvalues},
\[
|\pi_{[r]}(\boldsymbol{y})|=\left(1+O\left(\frac{1}{D}\right)\right)\cdot \sqrt{\sum_{i\in [r]} a_i^2}\geq \left(1+O\left(\frac{1}{D}\right)\right)\cdot \sqrt{\sum_{i\in [r]\backslash \{s\}} a_i^2}= \left(1+O\left(\frac{1}{D}\right)\right)\cdot \left| \sum_{i\in [r]\backslash \{s\}} a_i\Bv_i\right|.
\]
Hence, the complement $\mathcal{A}^c$ of the event $\mathcal{A}$ satisfies that
\begin{equation*}\label{equ:1-pF}
\begin{aligned}
\mathbb{P}(\mathcal{A}^c)\leq &\ \mathbb{P}\left(\By\in S^k: |\pi_{[r]}(\boldsymbol{y})|\geq  \left(1+O\left(\frac{1}{D}\right)\right)\cdot \bigl(\alpha_C-2\bigr)\frac{\sqrt{\ell}}{\sqrt{k}}\right) \\
\leq &\, \operatorname{Vol}\left(\boldsymbol{y} \in S^k : |\pi_{[r]}(\boldsymbol{y})| \geq \frac{\alpha_C\sqrt{\ell}}{2\sqrt{k}}\right) \left. \middle/ \right. \operatorname{Vol}(N_{\mathrm{per}}(\boldsymbol{x}[r]))
\leq {\left(\frac{p_C}{10}\right)^{C\ell}}\cdot{\left(\frac{p_C}{2}\right)^{-r}} \leq 5^{-C\ell},
\end{aligned}
\end{equation*}
where the second inequality uses the fact that $D \gg \alpha_C\geq 10$ and the second last inequality follows from Lemma~\ref{Volumenonperfect} and Lemma~\ref{volume}. 
Using item (3) of Lemma~\ref{lem:eigenvalues} again, we have
\[
1=|\boldsymbol{y}|^2\geq |\pi_{[r]}(\boldsymbol{y})|^2=\left|\sum_{i=1}^r a_i \boldsymbol{v}_i\right|^2
=\left(1+O\left(\frac{1}{D}\right)\right)\cdot \sum_{i=1}^r a_i^2\geq \frac14 \sum_{i=1}^r a_i^2,
\]
implying that \(|a_i|\leq 2\) for all \(i\in [r]\). 
Under the assumption that \eqref{equ:E(a_s)} holds for every \( \boldsymbol{\omega}_{\overline{s}}\in \mathcal{A} \), 
applying the law of total expectation, we can derive the desired estimation \eqref{equ:Eai} in the following:
\begin{align*}
\mathbb{E}[a_s] &= \mathbb{P}(\mathcal{A}) \cdot \mathbb{E}\left[\ a_s \mid \mathcal{A}\ \right] +(1-\mathbb{P}(\mathcal{A}))\cdot\mathbb{E}[\ a_s \mid \mathcal{A}^c\ ] \\
&= \left(1 - O(5^{-C\ell})\right)\cdot\left(-\frac{e^{-c^2/2}}{p\sqrt{2\pi k}}\left(1 + O\left(\frac{1}{D}\right)\right) + O\left(\frac{\sqrt{\ell}}{\sqrt{k}}|\pi_{V_r(s)}(\boldsymbol{v}_s)|\right)\right)\ + O(5^{-C\ell}) \\
&= -\frac{e^{-c^2/2}}{p\sqrt{2\pi k}}\left(1 + O\left(\frac{1}{D}\right)\right) + O\left(\frac{\sqrt{\ell}}{\sqrt{k}}|\pi_{V_r(s)}(\boldsymbol{v}_s)|\right),
\end{align*}
where the last equality holds by the fact that \(O(5^{-C\ell})=O\left(\frac{1}{D\sqrt{k}}\right) \). 

The remainder of this proof is devoted to establishing \eqref{equ:E(a_s)} for every \(\boldsymbol{\omega}_{\overline{s}} \in \mathcal{A}\).
By Proposition~\ref{prop:common_neighborhood}, 
\[
N_{\mathrm{per}}(\boldsymbol{x}{[r]}) = \left\{\boldsymbol{y}=(a_1,\ldots, a_r) \in S^k \, \left| \, a_i \leq -\frac{c}{\sqrt{k}} \text{ for all }  i \in [r] \text{ and } \left|\pi_{[r]}(\boldsymbol{y})\right|\right.=\left|\sum_{i=1}^r a_i \boldsymbol{v}_i\right| \leq \alpha_C\frac{\sqrt{\ell}}{\sqrt{k}} \right\}.
\]
Let \(I(\boldsymbol{\omega}_{\overline{s}})\) denote the set of all values \(a_s\) such that \(\boldsymbol{y} = (a_1, \ldots, a_r) \in N_{\mathrm{per}}(\boldsymbol{x}{[r]})\) while  \( \boldsymbol{\omega}_{\overline{s}}\) is fixed.

We claim that for any fixed \(\boldsymbol{\omega}_{\overline{s}} \in \mathcal{A}\), \(I(\boldsymbol{\omega}_{\overline{s}})=[-A_s, -\frac{c}{\sqrt{k}}]\) is an interval with
\(
\frac{\sqrt{\ell}}{\sqrt{k}} \leq A_s \leq \frac{2\alpha_C\sqrt{\ell}}{\sqrt{k}}.
\)
Note that, from the above expression of \(N_{\mathrm{per}}(\boldsymbol{x}{[r]})\), 
the domain \(I(\boldsymbol{\omega}_{\overline{s}})\) of \(a_s\) is determined by the conditions \(a_s \le -\frac{c}{\sqrt{k}}\) and 
\(
\left|\sum_{i=1}^r a_i \boldsymbol{v}_i\right| \le \alpha_C \frac{\sqrt{\ell}}{\sqrt{k}},
\)
the latter of which can be viewed as a quadratic inequality in the variable \(a_s\).
Thus, \(I(\boldsymbol{\omega}_{\overline{s}})\) must be an interval \( [-A_s, -\frac{c}{\sqrt{k}}] \) for some constant $A_s>0$. 
Suppose that  $ -\frac{\sqrt\ell}{\sqrt k}\leq a_s\leq -\frac{c}{\sqrt{k}}$.
Then, since \(\boldsymbol{\omega}_{\overline{s}} \in \mathcal{A}\) and \(|\boldsymbol{v}_s|=1+O(\frac{1}{D})\) (by Lemma~\ref{lem:eigenvalues}), we have 
\[
\left|\pi_{[r]}(\boldsymbol{y})\right| \leq \left(\alpha_C - 2 + 1 + O\left(\frac{1}{D}\right)\right)\frac{\sqrt{\ell}}{\sqrt{k}} \leq \alpha_C\frac{\sqrt{\ell}}{\sqrt{k}},
\]
implying that $[-\frac{\sqrt\ell}{\sqrt k},-\frac{c}{\sqrt{k}}]\subseteq I(\boldsymbol{\omega}_{\overline{s}})$ and thus $A_s\geq \frac{\sqrt\ell}{\sqrt k}$. 
On the other hand, if $a_s<-\frac{2\alpha_C\sqrt{\ell}}{\sqrt{k}}$, then
\[
\left|\pi_{[r]}(\boldsymbol{y})\right| \geq |a_s \Bv_s|-\left|\sum_{i \in [r]\backslash \{s\} } a_i \boldsymbol{v}_i\right|\geq  \left( 2\alpha_C+O\left(\frac{1}{D}\right)-(\alpha_C-2)\right)\frac{\sqrt{\ell}}{\sqrt{k}} > \alpha_C \frac{\sqrt{\ell}}{\sqrt{k}}.
\]
Combining with the above bounds, we obtain that \(\frac{\sqrt{\ell}}{\sqrt{k}} \leq A_s \leq \frac{2\alpha_C\sqrt{\ell}}{\sqrt{k}},\) proving the claim. 

While fixing \( \boldsymbol{\omega}_{\overline{s}} = \bigl(a_1,\ldots,a_{s-1},a_{s+1},\ldots, a_r\bigr)\in \mathcal{A} \), 
the surface area measure satisfies
\[
\operatorname{Vol}\left(\left\{\boldsymbol{y}\in S^k : \pi_{[r]}(\boldsymbol{y})=\sum_{i\in [r]\backslash \{s\}}a_i\boldsymbol{v}_i+a_s\boldsymbol{v}_s,\ a_s\in[a,a+da]\right\}\right) \propto \left(1-\left| \pi_{[r]}(\boldsymbol{y})\right|^2\right)^{\frac{k-r-1}{2}}da.
\]
Therefore, \(\mathbb{E}\bigl[ a_s \mid \boldsymbol{\omega}_{\overline{s}} \bigr]\) can be written as
\begin{equation}\label{equ:Eaicalculation1}
\frac{\displaystyle\int_{-A_s}^{-\frac{c}{\sqrt{k}}} a_s \left(1 - |\pi_{[r]}(\boldsymbol{y})|^2\right)^{\frac{k-r-1}{2}} da_s}{\displaystyle\int_{-A_s}^{-\frac{c}{\sqrt{k}}} \left(1 - |\pi_{[r]}(\boldsymbol{y})|^2\right)^{\frac{k-r-1}{2}} da_s} = \frac{\displaystyle\int_{-A_s}^{-\frac{c}{\sqrt{k}}} a_s \exp\left(-\frac{k}{2}|\pi_{[r]}(\boldsymbol{y})|^2\right) da_s}{\displaystyle\int_{-A_s}^{-\frac{c}{\sqrt{k}}} \exp\left(-\frac{k}{2}|\pi_{[r]}(\boldsymbol{y})|^2\right) da_s} \left(1 + O\left(\frac{1}{D^2}\right)\right),
\end{equation}
where the equality holds since for all \(1 \leq r \leq C\ell\) and \(x:=|\pi_{[r]}(\boldsymbol{y})|=O(\frac{\sqrt{\ell}}{\sqrt{k}})\),
\begin{equation}\label{equ:exp(x^2)}
\begin{aligned}
\exp\left(\frac{k}{2}x^2\right)(1-x^2)^{\frac{k-r-1}{2}} 
&=\exp\left( \frac{k}{2}x^2 - \frac{k-r-1}{2}x^2+O(k x^4)\right)\\
&=\exp\left(O(\ell x^2+kx^4)\right)= 1 + O\left(\frac{1}{D^2}\right).  
\end{aligned}
\end{equation}
Here, \(\pi_{[r]}(\boldsymbol{y})\) is viewed as a function of \(a_s\), that is,
\[
\bigl|\pi_{[r]}(\boldsymbol{y})\bigr|^2 
=
a_s^2 \, \lvert \boldsymbol{v}_s \rvert^2 
+ 2a_s\lvert \boldsymbol{v}_s \rvert\cdot \left\langle  \frac{\boldsymbol{v}_s}{|\boldsymbol{v}_s|}, \sum_{i \neq s} a_i \boldsymbol{v}_i \right\rangle 
+ \sum_{i,j \ne s} a_i \, a_j \, \langle \boldsymbol{v}_i, \boldsymbol{v}_j \rangle.
\]

To simplify the calculation of \eqref{equ:Eaicalculation1}, we introduce the following shifted variables:
\[
t := a_s \, |\boldsymbol{v}_s| 
+ \left\langle \frac{\boldsymbol{v}_s}{|\boldsymbol{v}_s|}, \sum_{i \neq s} a_i \boldsymbol{v}_i \right\rangle,
\quad
A_s' := A_s \, |\boldsymbol{v}_s| 
- \left\langle \frac{\boldsymbol{v}_s}{|\boldsymbol{v}_s|}, \sum_{i \neq s} a_i \boldsymbol{v}_i \right\rangle,
\quad
B_s' := -\frac{c}{\sqrt{k}} \, |\boldsymbol{v}_s| 
+ \left\langle \frac{\boldsymbol{v}_s}{|\boldsymbol{v}_s|}, \sum_{i \neq s} a_i \boldsymbol{v}_i \right\rangle.
\]
Before substituting these parameters into \eqref{equ:Eaicalculation1}, we derive some estimates for them.
By Lemma~\ref{lem:eigenvalues}, 
\begin{equation}\label{equ:<vs,sum>}
\begin{aligned}
\left|\left\langle  \frac{\boldsymbol{v}_s}{|\boldsymbol{v}_s|}, \sum_{i \neq s} a_i \boldsymbol{v}_i \right\rangle \right|&=\left|\left\langle \frac{\pi_{V_r(s)}(\boldsymbol{v}_s)}{|\boldsymbol{v}_s|}, \sum_{i \neq s} a_i \boldsymbol{v}_i \right\rangle\right|\leq \frac{|\pi_{V_r(s)}(\boldsymbol{v}_s)|}{1+O\left(\frac{1}{D}\right)} \cdot \left|\sum_{i \neq s} a_i \boldsymbol{v}_i\right|\\
&=O\bigl(|\pi_{V_r(s)}(\boldsymbol{v}_s)|\bigr) \cdot \sqrt{\sum_{i \neq s} a_i^2}=O\left(\frac{\sqrt{\ell}}{\sqrt{k}}\cdot |\pi_{V_r(s)}(\boldsymbol{v}_s)|\right)=O\left(\frac{\sqrt{\ell}}{D\sqrt{k}}\right),
\end{aligned}
\end{equation}
and using the bound \(\frac{\sqrt{\ell}}{\sqrt{k}} \leq A_s \leq \frac{2\alpha_C\sqrt{\ell}}{\sqrt{k}}\), we have
\begin{equation}\label{orderAs}
0<A_s' = \left(1+O\left(\frac{1}{D}\right)\right)\cdot A_s + O\left(\frac{\sqrt{\ell}}{D\sqrt{k}}\right) = \Theta\left(\frac{\sqrt{\ell}}{\sqrt{k}}\right) \quad \text{and} \quad |B_s'| = O\left(\frac{\sqrt{\ell}}{D\sqrt{k}}\right).
\end{equation}
Now, substituting \(t\), \(A_s'\), and \(B_s'\) for \(a_s\), \(A_s\), and \(-\frac{c}{\sqrt{k}}\), respectively, in \eqref{equ:Eaicalculation1} yields
\begin{equation}\label{equ:Eaicalculation2}
\begin{aligned}
\mathbb{E}\bigl[ a_s \mid \boldsymbol{\omega}_{\overline{s}} \bigr]
&= \frac{\displaystyle\int_{-A_s'}^{B_s'} \left(t-\left\langle  \frac{\boldsymbol{v}_s}{|\boldsymbol{v}_s|}, \sum_{i \neq s} a_i \boldsymbol{v}_i \right\rangle\right) \exp\left(-\frac{k}{2}t^2\right) dt}{\displaystyle\int_{-A_s'}^{B_s'} \exp\left(-\frac{k}{2}t^2\right) dt} \left(1 + O\left(\frac{1}{D}\right)\right) \\
&=\frac{\displaystyle\int_{-A_s'}^{B_s'} t \exp\left(-\frac{k}{2}t^2\right) dt}{\displaystyle\int_{-A_s'}^{B_s'} \exp\left(-\frac{k}{2}t^2\right) dt} \left(1 + O\left(\frac{1}{D}\right)\right)+O\left(\left\langle \frac{\boldsymbol{v}_s}{|\boldsymbol{v}_s|}, \sum_{i \neq s} a_i \boldsymbol{v}_i \right\rangle\right).
\end{aligned}    
\end{equation}
Using the basic equality \( \int_{-\infty}^{\beta} t \exp\left(-\frac{t^2}{2}\right) dt= -\exp\left(-\frac{\beta^2}{2}\right)\),
we have 
\begin{align*}
\frac{\left|\displaystyle\int_{-\infty}^{-A_s'} t \exp\left(-\frac{k}{2}t^2\right) dt \right|}{\left|\displaystyle\int_{-\infty}^{B_s'} t \exp\left(-\frac{k}{2}t^2\right) dt\right|}
= \frac{\displaystyle\exp\left(-\frac{k}{2}(A_s')^2\right)}{\displaystyle\exp\left(-\frac{k}{2}(B_s')^2\right)}
= O\left(\exp\left(-\frac{\ell}{10}\right)\right)
= O\left(\frac{1}{D}\right),
\end{align*}
where the second equality follows from \eqref{orderAs} that $(A_s')^2-(B_s')^2=(A_s'-B_s')(A_s'+B_s')\geq \frac{(A_s)^2}{4}\geq \frac{\ell}{4k}$. 
Furthermore, \eqref{orderAs} gives
\(
\left(|B_s'|+\frac{\sqrt\ell}{D\sqrt k}\right)^2=O\left(\frac{\ell}{D^2k}\right)<\frac{\ell}{16k}\leq \frac{(A_s')^2}{4},
\)
which implies that 
\begin{align*}
\int_{-\infty}^{-A_s'} \exp\left(-\frac{k}{2}t^2\right) dt \leq & \frac{1}{|A_s'|}\int_{-\infty}^{-A_s'} (-t) \exp\left(-\frac{k}{2}t^2\right) dt = \frac{1}{|A_s'|k}\cdot\exp\left(-\frac{k}{2}(A_s')^2\right)\\
\leq&\frac{\sqrt \ell}{D\sqrt k}\cdot\exp\left(-\frac{k}{2}\left(|B_s'|+\frac{\sqrt\ell}{D\sqrt k}\right)^2\right)\cdot\exp\left(-\frac{k}{4}(A_s')^2\right)\\
\leq& \int_{-|B_s'|-\frac{\sqrt{\ell}}{D\sqrt{k}}}^{-|B_s'|} \exp\left(-\frac{k}{2}t^2\right) dt\cdot \exp\left(-\frac{\ell}{16}\right)\leq \int_{-\infty}^{B_s'} \exp\left(-\frac{k}{2}t^2\right) dt\cdot O\left(\frac{1}{D}\right).
\end{align*}
Putting the above two bounds into \eqref{equ:Eaicalculation2}, and using \eqref{equ:<vs,sum>}, we have
\begin{equation}\label{equ:Eaicalculation2.2}
\mathbb{E}\bigl[ a_s \mid \boldsymbol{\omega}_{\overline{s}} \bigr]
=\frac{\displaystyle\int_{-\infty}^{B_s'} t \exp\left(-\frac{k}{2}t^2\right) dt}{\displaystyle\int_{-\infty}^{B_s'} \exp\left(-\frac{k}{2}t^2\right) dt} \left(1 + O\left(\frac{1}{D}\right)\right)+ O\left(\frac{\sqrt{\ell}}{\sqrt{k}}\cdot |\pi_{V_r(s)}(\boldsymbol{v}_s)|\right).
\end{equation}

Next, applying \eqref{equ:<vs,sum>} and Lemma~\ref{lem:eigenvalues} once more, we obtain the following equation 
\[
B_s'=-\frac{c}{\sqrt{k}}|\boldsymbol{v}_s|+\left\langle\frac{\boldsymbol{v}_s}{|\boldsymbol{v}_s|},\sum_{i\neq s}a_i\boldsymbol{v}_i \right\rangle= -\frac{c}{\sqrt{k}}\cdot \left(1+O\left(\frac{1}{D}\right)\right)+O\left(\frac{\sqrt{\ell}}{\sqrt{k}}\cdot |\pi_{V_r(s)}(\boldsymbol{v}_s)|\right).
\] 
Let \( X \) be the standard normal random variable. 
By Lemma~\ref{lem:normal_conditional_expectation}, we have that
\begin{equation*}\label{normalderivate}
\left|\frac{d}{dx}\left(\frac{\displaystyle\int_{-\infty}^xt\exp(-\frac{k}{2}t^2)dt}{\displaystyle\int_{-\infty}^x\exp(-\frac{k}{2}t^2)dt}\right)\right|
=\left|\frac{1}{\sqrt{k}}\frac{d}{dx}\mathbb{E}[X|X\leq x\sqrt{k}]\right|\leq 100.
\end{equation*}
Hence, we can derive from \eqref{equ:Eaicalculation2.2} that
\begin{equation}\label{equ:Eaicalculation2.3}
\mathbb{E}\bigl[ a_s \mid \boldsymbol{\omega}_{\overline{s}} \bigr]
=\frac{\displaystyle\int_{-\infty}^{-\frac{c}{\sqrt{k}}} t \exp\left(-\frac{k}{2}t^2\right) dt}{\displaystyle\int_{-\infty}^{-\frac{c}{\sqrt{k}}} \exp\left(-\frac{k}{2}t^2\right) dt} \left(1 + O\left(\frac{1}{D}\right)\right) + O\left(\frac{1}{D\sqrt k}\right) + O\left(\frac{\sqrt{\ell}}{\sqrt{k}}\cdot |\pi_{V_r(s)}(\boldsymbol{v}_s)|\right).
\end{equation}
Recall that \(\Phi\) is the standard normal CDF.
By Lemma~\ref{prop:convergence_of_cpk}, 
\( c = \Phi^{-1}(1-p) + O\left(\frac{1}{k}\right) \).
This yields
\begin{equation}\label{equ:E[X|X>c]}
\frac{\displaystyle\int_{-\infty}^{-\frac{c}{\sqrt{k}}} t \exp\left(-\frac{k}{2}t^2\right) dt}{\displaystyle\int_{-\infty}^{-\frac{c}{\sqrt{k}}} \exp\left(-\frac{k}{2}t^2\right) dt}=-\frac{1}{\sqrt{k}}\mathbb{E}[X|X\geq c]=-\frac{1}{\sqrt{k}}\frac{\phi(c)}{1-\Phi(c)}=-\frac{e^{-c^2/2}}{p\sqrt{2\pi k}}+O\left(\frac{1}{k\sqrt k}\right),
\end{equation}
where the second equality follows from Lemma~\ref{lem:normal_conditional_expectation}
and the last equality holds by \(\phi(c)=\frac{e^{-c^2/2}}{\sqrt{2\pi}}\) and \(1-\Phi(c)=p+O(\frac{1}{k})\).
Combining \eqref{equ:Eaicalculation2.3} and \eqref{equ:E[X|X>c]}, we obtain \eqref{equ:E(a_s)}, implying \eqref{equ:Eai} (as shown earlier).

The proof of \eqref{equ:Eai3} proceeds analogously by considering the random vector 
\(\By\in \overline{N}_{\mathrm{per}}(\Bx[r])\) in place of \(N_{\mathrm{per}}(\Bx[r])\). 
In this blue case, the counterpart of \eqref{equ:E[X|X>c]} becomes 
\(\frac{1}{\sqrt{k}}\,\mathbb{E}[X \mid X \ge -c],\)
which contributes the main term 
\( \frac{e^{-c^2/2}}{(1-p)\sqrt{2\pi k}}.\)
All arguments carry over with corresponding changes, except for \eqref{equ:Eaicalculation1}, which requires some modification. 
We provide a detailed proof below.
Define \(\pi_{[r]}(\boldsymbol{y}) = \sum_{i=1}^r a_i \boldsymbol{v}_i\),
and the same event \(\mathcal{A}\) as before.
Then for any fixed \(\boldsymbol{\omega}_{\overline{s}} \in \mathcal{A}\),
in this case we have
\[
I(\boldsymbol{\omega}_{\overline{s}})=\left[-\frac{c}{\sqrt{k}}, A_s\right], \mbox{ where }
\frac{\sqrt{\ell}}{\sqrt{k}} \leq A_s \leq \frac{2\alpha_C\sqrt{\ell}}{\sqrt{k}}.
\]
Since this interval is not positive everywhere, 
the blue version of \eqref{equ:Eaicalculation1} cannot be obtained directly from the previous estimation \eqref{equ:exp(x^2)}.
However, we are able to derive the following (weaker) analogue of \eqref{equ:Eaicalculation1}, which nevertheless suffices for the purposes of the remaining proof:
\begin{equation*}\label{equ:Eaicalculation3}
\frac{\displaystyle\int_{-\frac{c}{\sqrt{k}}}^{A_s} a_s \left(1 - |\pi_{[r]}(\boldsymbol{y})|^2\right)^{\frac{k-r-1}{2}} da_s}{\displaystyle\int_{-\frac{c}{\sqrt{k}}}^{A_s} \left(1 - |\pi_{[r]}(\boldsymbol{y})|^2\right)^{\frac{k-r-1}{2}} da_s} = \frac{\displaystyle\int_{-\frac{c}{\sqrt{k}}}^{A_s} a_s \exp\left(-\frac{k}{2}|\pi_{[r]}(\boldsymbol{y})|^2\right) da_s}{\displaystyle\int_{-\frac{c}{\sqrt{k}}}^{A_s} \exp\left(-\frac{k}{2}|\pi_{[r]}(\boldsymbol{y})|^2\right) da_s} \left(1 + O\left(\frac{1}{D^2}\right)\right)+O\left(\frac{1}{D^2\sqrt{k}}\right).
\end{equation*}
To see this, we define
\begin{align*}
&M:=-\int_{-\frac{c}{\sqrt{k}}}^0 a_s\left(1 - |\pi_{[r]}(\boldsymbol{y})|^2\right)^{\frac{k-r-1}{2}} da_s, \quad
N:=\int_{0}^{A_s} a_s\left(1 - |\pi_{[r]}(\boldsymbol{y})|^2\right)^{\frac{k-r-1}{2}} da_s,\\
&M':=-\int_{-\frac{c}{\sqrt{k}}}^0 a_s\exp\left(-\frac{k}{2}|\pi_{[r]}(\boldsymbol{y})|^2\right) da_s, \quad
N':=\int_{0}^{A_s} a_s\exp\left(-\frac{k}{2}|\pi_{[r]}(\boldsymbol{y})|^2\right) da_s,\\
& P:=\int_{-\frac{c}{\sqrt{k}}}^{A_s} \left(1 - |\pi_{[r]}(\boldsymbol{y})|^2\right)^{\frac{k-r-1}{2}} da_s\quad\mbox{and}\quad
P':=\int_{-\frac{c}{\sqrt{k}}}^{A_s} \exp\left(-\frac{k}{2}|\pi_{[r]}(\boldsymbol{y})|^2\right) da_s.
\end{align*}
Using \eqref{equ:exp(x^2)}, we have \(M'=\bigl(1+O(\frac{1}{D^2})\bigr)M\), \(N'=\bigl(1+O(\frac{1}{D^2})\bigr)N\), and \(P'=\bigl(1+O(\frac{1}{D^2})\bigr)P\). 
Then 
\[
-M'+N'=(-M+N)\cdot \left(1+O\left(\frac{1}{D^2}\right)\right)+O\left(\frac{M}{D^2}\right).
\]
Now using  \(P'=\bigl(1+O(\frac{1}{D^2})\bigr)P\), we have
\[
\frac{-M'+N'}{P'}= \frac{-M+N}{P}\cdot \left(1+O\left(\frac{1}{D^2}\right)\right)+O\left(\frac{M}{D^2\cdot P}\right).
\]
This, together with an easy observation that \(\left|M\right|=O\left(\frac{1}{\sqrt{k}}\right)\left|P\right|\), implies the above desired equality.
We point out that the extra \( O\left(\frac{1}{D^2\sqrt{k}}\right) \) term is negligible for the remaining proof of the blue case, and therefore this analogue of \eqref{equ:Eaicalculation1} suffices for the proof of Lemma~\ref{lem:coefficient_expectations}. 
\end{proof}

We are ready to prove the main result - Lemma~\ref{lem:key-tech} of this subsection.

\begin{proof}[Proof of Lemma~\ref{lem:key-tech}.]
Let \(\boldsymbol{x}[r]\), \(\boldsymbol{v}_1, \ldots, \boldsymbol{v}_r\), and \(\boldsymbol{e}_1, \ldots, \boldsymbol{e}_r\) be as given in the conditions.
By symmetry, 
we only consider the vector \(\boldsymbol{y}\) sampled uniformly at random from \(N_{\mathrm{per}}(\boldsymbol{x}[r])\).
Recall that $\pi_{[r]}(\boldsymbol{y})=a_1\boldsymbol{v}_1+\cdots+a_r\boldsymbol{v}_r$. 
Using  Lemma~\ref{lem:dist_e_v} and Lemma~\ref{lem:coefficient_expectations},
we can show that for each $s\in [r]$,
\begin{align*}
\left|\left\langle\sum_{i\neq s}a_i\boldsymbol{v}_i,\pi_{V_r(s)}(\boldsymbol{e}_s)\right\rangle\right|\leq&\left|\sum_{i \neq s} a_i \boldsymbol{v}_i\right|\cdot|\pi_{V_r(s)}(\boldsymbol{e}_s)|
\leq\left(1+O\left(\frac{1}{D}\right)\right)\cdot\left(\sum_{i=1}^{r}a_i^2\right)^{\frac12}\cdot O\left(\frac{\sqrt{\ell}}{\sqrt{k}}\right)\\
\leq&\left(1+O\left(\frac{1}{D}\right)\right)\cdot\left|\sum_{i=1}^r a_i \boldsymbol{v}_i\right|\cdot O\left(\frac{\sqrt{\ell}}{\sqrt{k}}\right) \leq |\pi_{[r]}(\boldsymbol{y})| \cdot O\left(\frac{\sqrt{\ell}}{\sqrt{k}}\right)=O\left(\frac{1}{D\sqrt{k}} \right),
\end{align*}
where the final equality follows by \(|\pi_{[r]}(\boldsymbol{y})|=O\left(\frac{\sqrt{\ell}}{\sqrt{k}}\right) \) and \(\frac{\ell}{k}=\frac{1}{D\sqrt{k}}\). This implies that
\begin{align*}
\mathbb{E}[\langle\pi_{[r]}(\boldsymbol{y}),\boldsymbol{e}_s\rangle]=&\mathbb{E}[a_s]\cdot\langle\boldsymbol{v}_s,\boldsymbol{e}_s \rangle+\mathbb{E}\left[\left\langle\sum_{i\neq s}a_i\boldsymbol{v}_i,\boldsymbol{e}_s\right\rangle\right]\\
=&\mathbb{E}[a_s]\cdot\left(1+O\left(\frac{1}{D^2\ell}\right)\right)+\mathbb{E}\left[\left\langle\sum_{i\neq s}a_i\boldsymbol{v}_i,\pi_{V_r(s)}(\boldsymbol{e}_s)\right\rangle\right]\\
=&-\frac{e^{-c^2/2}}{p\sqrt{2\pi k}}\left(1 + O\left(\frac{1}{D}\right)\right) + O\left(\frac{\sqrt{\ell}}{\sqrt{k}}|\pi_{V_r(s)}(\boldsymbol{v}_s)|\right) + O\left(\frac{1}{D\sqrt{k}} \right)\\
=&-\frac{e^{-c^2/2}}{p\sqrt{2\pi k}}\left(1 + O\left(\frac{1}{D}\right)\right) + O\left(\frac{\sqrt{\ell}}{\sqrt{k}}|\pi_{V_r(s)}(\boldsymbol{v}_s)|\right),
\end{align*}
where the second equality follows from Lemma~\ref{lem:dist_e_v} and the fact $\sum_{i\neq s}a_i\boldsymbol{v}_i\in V_r(s)$, and the third equality holds by Lemma~\ref{lem:coefficient_expectations}.
This completes the proof of \eqref{eq:expection2} (and similarly, of \eqref{eq:expection2b}).
\end{proof}

\subsection{Estimate of $\mathbb{E}\left[\langle \pi_{[s]}(\boldsymbol{y}), \pi_{[s]}(\boldsymbol{z}) \rangle\right]$: Two Random Projections}

In this subsection, we complete the proof of Theorem~\ref{thm:projection-expectation}. To do so, we first establish the following approximate version, involving two independent random vectors.

\begin{lemma}\label{lem:projection-expectation2}
Let \(0 \le s \le r \le C\ell\) and fix a perfect sequence \(\boldsymbol{x}[r]\). 
Consider independent random vectors \(\boldsymbol{y}\) and \(\boldsymbol{z}\), where \(\boldsymbol{y}\) is uniformly distributed in \(N_{\mathrm{per}}(\boldsymbol{x}[r])\) and \(\boldsymbol{z}\) in \(N_{\mathrm{per}}(\boldsymbol{x}[s])\). Then
\begin{equation}\label{eq:proj-expectation}
\mathbb{E}\left[\langle \pi_{[s]}(\boldsymbol{y}), \pi_{[s]}(\boldsymbol{z}) \rangle\right]
=
\frac{e^{-c^2}}{2\pi p^2}\cdot \frac{s}{k}
+ O\left(\frac{\ell}{Dk}\right).
\end{equation}

Analogously, consider independent random vectors \(\boldsymbol{y}\) and \(\boldsymbol{z}\), where \(\boldsymbol{y}\) is uniformly distributed in \(\overline{N}_{\mathrm{per}}(\boldsymbol{x}[r])\) and \(\boldsymbol{z}\) in \(\overline{N}_{\mathrm{per}}(\boldsymbol{x}[s])\). Then
\begin{equation}\label{eq:proj-expectation_b}
\mathbb{E}\left[\langle \pi_{[s]}(\boldsymbol{y}), \pi_{[s]}(\boldsymbol{z}) \rangle\right]
=
\frac{e^{-c^2}}{2\pi (1-p)^2} \cdot \frac{s}{k}
+ O\left(\frac{\ell}{Dk}\right).
\end{equation}
\end{lemma}

\begin{proof}
We consider \eqref{eq:proj-expectation}. 
Fix \(0\leq s\leq r\leq C\ell \) and a perfect sequence $\Bx[r]=(\Bx_1,\ldots, \Bx_r)\in (S^k)^r$. 
As before, we define \( X_i = \operatorname{span}(\boldsymbol{x}_1, \ldots, \boldsymbol{x}_i) \) for each \( i \in [r] \).
Analogous to the basis \(\boldsymbol{v}_1, \ldots, \boldsymbol{v}_r\) of \(X_r\) defined in Definition~\ref{def:eigenvalues}, we define a corresponding basis \(\boldsymbol{w}_1, \ldots, \boldsymbol{w}_s\) for \(X_s\), such that
\[
\langle \boldsymbol{x}_i, \boldsymbol{w}_j \rangle
=
\begin{cases}
1, & \text{if } i = j;\\
0, & \text{if } i \ne j
\end{cases}
\]
for all \(1 \le i, j \le s\).
Similarly, analogous to the subspaces \( V_r(i) \) of $X_r$ for $i\in [r]$, we define \(W_s(i):=\operatorname{span}(\boldsymbol{w}_j:j\in [s]\backslash \{i\}).\) 
Finally, we recall the vectors $\boldsymbol{e}_1, \ldots, \boldsymbol{e}_r$ from Definition~\ref{def:eigenvalues}. 
Note that $\boldsymbol{e}_1, \ldots, \boldsymbol{e}_s$ form an orthonormal basis for $X_s$ and depend only on \(\boldsymbol{x}[s]\).

Let \(\boldsymbol{z}\) be a random vector uniformly distributed in \(N_{\mathrm{per}}(\boldsymbol{x}[s])\). By Lemma~\ref{lem:key-tech}, we have
\begin{equation}\label{equ:pi(z),ei}
    \mathbb{E}[\langle\pi_{[s]}(\boldsymbol{z}),\boldsymbol{e}_i\rangle]= -\frac{e^{-c^2/2}}{p\sqrt{2\pi k}}\left(1 + O\left(\frac{1}{D}\right)\right) + O\left(\frac{\sqrt{\ell}}{\sqrt{k}}|\pi_{W_s(i)}(\boldsymbol{w}_i)|\right)
\end{equation}
for all \(i\in [s]\).
Let \(\boldsymbol{y}\) be a random vector uniformly distributed in \(N_{\mathrm{per}}(\boldsymbol{x}[r])\).
For each \(i\in [s]\), we have \( \langle\pi_{[s]}(\boldsymbol{y}),\boldsymbol{e}_i\rangle=\langle\pi_{[r]}(\boldsymbol{y}),\boldsymbol{e}_i\rangle \), and thus, applying Lemma~\ref{lem:key-tech} again, we can derive
\begin{equation}\label{equ:pi(y)pi(z)}
\mathbb{E}[\langle\pi_{[s]}(\boldsymbol{y}),\boldsymbol{e}_i\rangle]=\mathbb{E}[\langle\pi_{[r]}(\boldsymbol{y}),\boldsymbol{e}_i\rangle]= -\frac{e^{-c^2/2}}{p\sqrt{2\pi k}}\left(1 + O\left(\frac{1}{D}\right)\right) + O\left(\frac{\sqrt{\ell}}{\sqrt{k}}|\pi_{V_r(i)}(\boldsymbol{v}_i)|\right).
\end{equation}
Since \(\boldsymbol{y}\) and \(\boldsymbol{z}\) are independent, we have
\begin{align*}
\mathbb{E}\left[\left\langle \pi_{[s]}(\boldsymbol{y}), \pi_{[s]}(\boldsymbol{z}) \right\rangle\right]=&\mathbb{E}\left[\left\langle \sum_{i=1}^{s}\langle\pi_{[s]}(\boldsymbol{y}),\boldsymbol{e}_i\rangle \boldsymbol{e_i}, \sum_{i=1}^{s}\langle\pi_{[s]}(\boldsymbol{z}),\boldsymbol{e}_i\rangle \boldsymbol{e}_i \right\rangle\right]=\sum_{i=1}^{s}\mathbb{E}[\left\langle\pi_{[s]}(\boldsymbol{y}),\boldsymbol{e}_i\right\rangle]\cdot \mathbb{E}[\left\langle\pi_{[s]}(\boldsymbol{z}),\boldsymbol{e}_i\right\rangle].
\end{align*}
Substituting \eqref{equ:pi(z),ei} and \eqref{equ:pi(y)pi(z)} into the right-hand side of the above expression and simplifying yield
\begin{align*}
\mathbb{E}\left[\left\langle \pi_{[s]}(\boldsymbol{y}), \pi_{[s]}(\boldsymbol{z}) \right\rangle\right] =& \sum_{i=1}^s \left(1+O\left(\frac{1}{D}\right)\right)\cdot \frac{e^{-c^2}}{p^2 (2\pi k)}+O\left(\frac{\ell}{k}\right)\cdot \sum_{i=1}^s |\pi_{V_r(i)}(\boldsymbol{v}_i)|\cdot|\pi_{W_s(i)}(\boldsymbol{w}_i)|\\
&+ O\left(\frac{\sqrt{\ell}}{k}\right)\cdot \sum_{i=1}^s \bigl(|\pi_{V_r(i)}(\boldsymbol{v}_i)|+|\pi_{W_s(i)}(\boldsymbol{w}_i)|\bigr)\\
=
& \left(\frac{e^{-c^2}}{2\pi p^2}\cdot \frac{s}{k}+ O\left(\frac{s}{Dk}\right)\right)
+ O\left(\frac{\ell}{k}\right)\cdot \left(\sum_{i=1}^{s}|\pi_{V_r(i)}(\boldsymbol{v}_i)|^2\right)^{\frac{1}{2}}\left(\sum_{i=1}^{s}|\pi_{W_s(i)}(\boldsymbol{w}_i)|^2\right)^{\frac{1}{2}}\\
&+ O\left(\frac{{\sqrt{\ell}}}{{k}}\right)\cdot\left( \sqrt{s}\cdot \left(\sum_{i=1}^{s}|\pi_{V_r(i)}(\boldsymbol{v}_i)|^2\right)^{\frac12}+ \sqrt{s}\cdot \left(\sum_{i=1}^{s}|\pi_{W_s(i)}(\boldsymbol{w}_i)|^2\right)^{\frac12}\right)\\
= &\frac{e^{-c^2}}{2\pi p^2} \cdot \frac{s}{k} + O\left(\frac{\ell}{Dk}\right),
\end{align*}
where the second equality follows from the Cauchy–Schwarz inequality, and the last equality uses the fact that \(s \le C\ell\) and item (2) of Lemma~\ref{lem:eigenvalues}, which gives 
\(\sum_{i=1}^{s} \left|\pi_{V_r(i)}(\boldsymbol{v}_i)\right|^2 = O\left(\frac{1}{D^2}\right)\) 
and 
\(\sum_{i=1}^{s} \left|\pi_{W_s(i)}(\boldsymbol{w}_i)\right|^2 = O\left(\frac{1}{D^2}\right)\).
This completes the proof of \eqref{eq:proj-expectation}, and similarly of \eqref{eq:proj-expectation_b}.
\end{proof}

Finally, we are ready to present the proof of Theorem~\ref{thm:projection-expectation}.

\begin{proof}[Proof of Theorem~\ref{thm:projection-expectation}]
We aim to show \eqref{eq:proj-expectation3}. 
Let us restate the setting:  
Given \(0 \le s < r \le C\ell\) and a perfect sequence \(\boldsymbol{x}[s]\) such that \(G_p(\boldsymbol{x}[s])\) forms a red clique, we independently sample \(\boldsymbol{x}_{s+1}, \ldots, \boldsymbol{x}_r\) uniformly from \(S^k\), and sample \(\boldsymbol{z}\) uniformly from \(N_{\mathrm{per}}(\boldsymbol{x}[s])\), independently of \(\{\boldsymbol{x}_{s+1}, \ldots, \boldsymbol{x}_r\}\).

We begin by considering a fixed perfect sequence \(\boldsymbol{x}[r-1]\) such that \(G_p(\boldsymbol{x}[r-1])\) forms a red clique.  
Under this assumption, the event \(A_{\mathrm{red},r} \wedge B_r\) occurs if and only if \(\boldsymbol{x}_r \in N_{\mathrm{per}}(\boldsymbol{x}[r-1])\), 
so \(\boldsymbol{x}_r\) corresponds to the random vector \(\boldsymbol{y}\) sampled uniformly from \(N_{\mathrm{per}}(\boldsymbol{x}[r-1])\).  
Applying Lemma~\ref{lem:projection-expectation2} for the fixed perfect sequence \(\boldsymbol{x}[r-1]\), we have
\begin{equation*}
\mathbb{E}\left[\langle \pi_{[s]}(\Bx_{r}), \pi_{[s]}(\boldsymbol{z}) \rangle \mid A_{\mathrm{red},r}\wedge B_r \wedge \{\Bx[r-1]\} \right]
= \mathbb{E}\left[\langle \pi_{[s]}(\By), \pi_{[s]}(\boldsymbol{z}) \rangle \right]=
\frac{e^{-c^2}}{2\pi p^2}\cdot \frac{s}{k}
+ O\left(\frac{\ell}{Dk}\right).
\end{equation*}
Integrating all such perfect sequences \(\boldsymbol{x}[r-1]\), we can get
\begin{equation}\label{eq:proj-expectation-fix-x[r-1]}
\mathbb{E}\left[\langle \pi_{[s]}(\Bx_{r}), \pi_{[s]}(\boldsymbol{z}) \rangle \mid A_{\mathrm{red},r}\wedge B_r \right]= \frac{e^{-c^2}}{2\pi p^2}\cdot \frac{s}{k}
+ O\left(\frac{\ell}{Dk}\right).
\end{equation}

Note that under this setting, we can derive from Definition~\ref{def:rsperfect} and Theorem~\ref{thm:perfect_nonperfect} that
\[
\frac{\mathbb{P}(A_{\mathrm{red},r}\wedge B_r)}{\mathbb{P}(A_{\mathrm{red},r})}=\frac{P_{\mathrm{red},r}^{\mathrm{per}}(\boldsymbol{x}{[s]})}{P_{\mathrm{red},r}(\boldsymbol{x}{[s]})}=1-O(2^{-C\ell}).
\]
This, together with the law of conditional probability and the fact $|\langle\pi_{[s]}(\boldsymbol{x}_{s+1}),\pi_{[s]}(\boldsymbol{z})\rangle|\leq 1$, implies
\begin{equation}\label{equ:expection-1}
\mathbb{E}[\langle\pi_{[s]}(\boldsymbol{x}_{s+1}),\pi_{[s]}(\boldsymbol{z})\rangle\mid A_{\mathrm{red},r}]=\mathbb{E}[\langle\pi_{[s]}(\boldsymbol{x}_{s+1}),\pi_{[s]}(\boldsymbol{z})\rangle\mid A_{\mathrm{red},r}\wedge{B_r}]+O(2^{-C\ell}).
\end{equation}

We now reorder the sequence \(\boldsymbol{x}[r]\) as \(\boldsymbol{y}[r] = (\boldsymbol{y}_1, \ldots, \boldsymbol{y}_r) := (\boldsymbol{x}_1, \ldots, \boldsymbol{x}_s, \boldsymbol{x}_{s+2}, \ldots, \boldsymbol{x}_r, \boldsymbol{x}_{s+1})\).  
Define \(A'_{\mathrm{red},r}\) as the event that \(G_p(\boldsymbol{y}[r])\) forms a red clique, and \(B_r'\) as the event that \(\boldsymbol{y}[r]\) forms a perfect sequence.
Evidently, \(A'_{\mathrm{red},r}= A_{\mathrm{red},r}\);
by the independence of random vectors, we deduce
\begin{equation}\label{equ:expection-2}
\mathbb{E}[\langle\pi_{[s]}(\boldsymbol{x}_{s+1}),\pi_{[s]}(\boldsymbol{z})\rangle\mid A_{\mathrm{red},r}]= \mathbb{E}[\langle\pi_{[s]}(\boldsymbol{y}_{r}),\pi_{[s]}(\boldsymbol{z})\rangle\mid A'_{\mathrm{red},r}].
\end{equation}
On the other hand, similar to the proof of \eqref{equ:expection-1}, we have
\begin{equation}\label{equ:expection-3}
\mathbb{E}[\langle\pi_{[s]}(\boldsymbol{y}_{r}),\pi_{[s]}(\boldsymbol{z})\rangle\mid A'_{\mathrm{red},r}]=\mathbb{E}[\langle\pi_{[s]}(\boldsymbol{y}_{r}),\pi_{[s]}(\boldsymbol{z})\rangle\mid A'_{\mathrm{red},r}\wedge B'_r]+O(2^{-C\ell}).
\end{equation}
The final equation we need is as follows, which holds because \(\boldsymbol{x}[r]\) and \(\boldsymbol{y}[r]\) are identically distributed:
\begin{equation}\label{equ:expection-4}
\mathbb{E}[\langle\pi_{[s]}(\boldsymbol{y}_{r}),\pi_{[s]}(\boldsymbol{z})\rangle\mid A'_{\mathrm{red},r}\wedge B'_r]= 
\mathbb{E}[\langle\pi_{[s]}(\boldsymbol{x}_{r}),\pi_{[s]}(\boldsymbol{z})\rangle\mid A_{\mathrm{red},r}\wedge B_r]
\end{equation}
Combining \eqref{equ:expection-1}, \eqref{equ:expection-2}, \eqref{equ:expection-3}, and \eqref{equ:expection-4},  
we obtain
\begin{equation*}
\mathbb{E}\bigl[\langle \pi_{[s]}(\boldsymbol{x}_{s+1}), \pi_{[s]}(\boldsymbol{z}) \rangle 
\mid A_{\mathrm{red},r} \wedge B_r \bigr]
=
\mathbb{E}\bigl[\langle \pi_{[s]}(\boldsymbol{x}_{r}), \pi_{[s]}(\boldsymbol{z}) \rangle 
\mid A_{\mathrm{red},r} \wedge B_r \bigr]
+ O(2^{-C\ell})
=
\frac{e^{-c^2}}{2\pi p^2} \cdot \frac{s}{k}
+ O\left(\frac{\ell}{Dk}\right),
\end{equation*}
where the last equality follows from \eqref{eq:proj-expectation-fix-x[r-1]} and the fact that 
\(2^{-C\ell} \ll \frac{1}{D^3 \ell} = \frac{\ell}{Dk}\).
This proves \eqref{eq:proj-expectation3}. 
The proof of \eqref{eq:proj-expectation4} proceeds similarly, 
and thus we complete the proof of Theorem~\ref{thm:projection-expectation}.
\end{proof}

\section{Proof of Theorem~\ref{thm2}}\label{sec:final-proof}
In this section, we complete the proof of Theorem~\ref{thm2}.
The proof proceeds in two parts.
First, we derive upper bounds for \( P_{\mathrm{red},r}^{\mathrm{per}} \) and \( \overline{P}_{\mathrm{blue},r}^{\mathrm{per}} \),
by estimating \( \kappa_{r}^{\mathrm{per}} \) and \( \overline{\kappa}_{r}^{\mathrm{per}} \) with the help of Theorems~\ref{ratiored} and~\ref{thm:projection-expectation}.
We then conclude the proof of Theorem~\ref{thm2} with some further calculations.

\subsection{Estimates of $P_{\mathrm{red},r}^{\mathrm{per}}$ and $\overline{P}_{\mathrm{blue},r}^{\mathrm{per}}$ (via $\kappa_{r}^{\mathrm{per}}$ and $\overline{\kappa}_{r}^{\mathrm{per}}$)}\label{subsec:kappa}

We first establish the following estimates for $\kappa_{r}^{\mathrm{per}}$ and $\overline{\kappa}_{r}^{\mathrm{per}}$. 

\begin{lemma}\label{lem:Erper}
For every $1\leq r\leq C\ell,$ we have
\begin{equation}\label{equ:kappa-final}
\kappa_r^{\mathrm{per}}\leq \prod_{i=0}^{r-1}\left(p - \left(\frac{e^{-c^2}}{2\pi}\right)^{\frac{3}{2}} \frac{1}{p^2}\cdot\frac{i}{\sqrt{k}} + O\left(\frac{1}{D^2}\right)\right),
\end{equation}
\begin{equation}\label{equ:kappa-final-2}
\overline{\kappa}_r^{\mathrm{per}}\leq \prod_{i=0}^{r-1}\left(1-p + \left(\frac{e^{-c^2}}{2\pi}\right)^{\frac{3}{2}} \frac{1}{(1-p)^2}\cdot\frac{i}{\sqrt{k}} + O\left(\frac{1}{D^2}\right)\right).
\end{equation}
\end{lemma}

\begin{proof}
We focus on proving the first inequality. 
Let \(\boldsymbol{x}[r]\) be a random \(r\)-tuple in \((S^k)^r\).  
To proceed, we show the following inequality by induction on \(s\), descending from \(s = r-1\) to \(s = 0\):
\begin{equation}\label{eq:recursion2}
\kappa_r^{\mathrm{per}}\leq \prod_{i=s+1}^{r-1}\left(p - \left(\frac{e^{-c^2}}{2\pi}\right)^{\frac{3}{2}} \frac{1}{p^2}\cdot\frac{i}{\sqrt{k}} + O\left(\frac{1}{D^2}\right)\right)\cdot\mathbb{E}\left[\left.\prod_{j=0}^{s}Q_{[j]}(\boldsymbol{x}_{j+1})\right| A_{\mathrm{red},r}\wedge B_r\right].
\end{equation}
The base case $s = r-1$ follows immediately from \eqref{equ:kappa=P_per}.
We now assume that \eqref{eq:recursion2} holds for some $1\leq s\leq r-1$, and we need to establish its validity for $s-1$. 

We begin by considering \(\boldsymbol{x}[s]\) as a \emph{fixed} perfect sequence for which \(G_p(\boldsymbol{x}[s])\) forms a red clique.
Let \(\boldsymbol{x}_{s+1}, \ldots, \boldsymbol{x}_r\) be independent random vectors uniformly distributed on \(S^k\), and let \(\boldsymbol{z}\) be chosen uniformly from \(N_{\mathrm{per}}(\boldsymbol{x}[s])\), independently of \(\boldsymbol{x}_{s+1}, \ldots, \boldsymbol{x}_r\).
By Theorem~\ref{ratiored}, we have
\begin{equation*}
\mathbb{E}\left[Q_{[s]}(\boldsymbol{x}_{s+1})\mid A_{\mathrm{red},r}\wedge B_r \right]
\leq p-\sqrt{\frac{k}{2\pi}} e^{-\frac{c^{2}}{2}}\cdot \mathbb{E}[\langle\pi_{[s]}(\boldsymbol{x}_{s+1}),\pi_{[s]}(\boldsymbol{z})\rangle\mid A_{\mathrm{red},r}\wedge B_r]+O\left(\frac{1}{D^2}\right).
\end{equation*}
This, together with Theorem~\ref{thm:projection-expectation}, shows that
\begin{equation}\label{eq:expection-Q[s]}
\mathbb{E}\left[Q_{[s]}(\boldsymbol{x}_{s+1})\mid A_{\mathrm{red},r}\wedge B_r \right]\leq p - \left(\frac{e^{-c^2}}{2\pi}\right)^{\frac{3}{2}} \frac{s}{p^2\sqrt{k}} + O\left(\frac{1}{D^2}\right).
\end{equation}
Here, it is worth noting that the error term $O\left(\frac{\ell}{Dk}\right)$ of Theorem~\ref{thm:projection-expectation} contributes an error of order $O\left(\frac{\ell}{D\sqrt{k}}\right)=O\left(\frac{1}{D^2}\right)$ as well.
Given that the sequence $\boldsymbol{x}{[s]}$ is fixed, we further derive from \eqref{eq:expection-Q[s]} that
\begin{equation*}
\mathbb{E}\left[\left.\prod_{j=0}^{s}Q_{[j]}(\boldsymbol{x}_{j+1})\right| A_{\mathrm{red},r}\wedge B_r  \right]
\leq \left(\prod_{j=0}^{s-1}Q_{[j]}(\boldsymbol{x}_{j+1})\right)\cdot \left(p - \left(\frac{e^{-c^2}}{2\pi}\right)^{\frac{3}{2}} \frac{s}{p^2\sqrt{k}} + O\left(\frac{1}{D^2}\right)\right).
\end{equation*}
By averaging over all possible $\boldsymbol{x}{[s]}$, we have
\begin{equation*}
\mathbb{E}\left[\left.\prod_{j=0}^{s}Q_{[j]}(\boldsymbol{x}_{j+1})\right| A_{\mathrm{red},r}\wedge B_r\right]
\leq\left(p - \left(\frac{e^{-c^2}}{2\pi}\right)^{\frac{3}{2}} \frac{s}{p^2\sqrt{k}} + O\left(\frac{1}{D^2}\right)\right)\mathbb{E}\left[\left.\prod_{j=0}^{s-1}Q_{[j]}(\boldsymbol{x}_{j+1})\right| A_{\mathrm{red},r}\wedge B_r\right].
\end{equation*}
Combining this with the induction hypothesis for \(s\), we obtain
\begin{align*}
\kappa_r^{\mathrm{per}}\leq& \prod_{i=s+1}^{r-1}\left(p - \left(\frac{e^{-c^2}}{2\pi}\right)^{\frac{3}{2}} \frac{1}{p^2}\cdot \frac{i}{\sqrt{k}} + O\left(\frac{1}{D^2}\right)\right)\cdot\mathbb{E}\left[\left.\prod_{j=0}^{s}Q_{[j]}(\boldsymbol{x}_{j+1})\right| A_{\mathrm{red},r}\wedge B_r\right]\\
\leq& \prod_{i=s}^{r-1}\left(p - \left(\frac{e^{-c^2}}{2\pi}\right)^{\frac{3}{2}} \frac{1}{p^2}\cdot \frac{i}{\sqrt{k}} + O\left(\frac{1}{D^2}\right)\right)\cdot\mathbb{E}\left[\left.\prod_{j=0}^{s-1}Q_{[j]}(\boldsymbol{x}_{j+1})\right| A_{\mathrm{red},r}\wedge B_r\right].
\end{align*}
This completes the induction.  
Note that since \(Q_{[0]}(\Bx) = P_{\mathrm{per}}(\Bx) \le p\) for all \(\Bx \in S^k\), the case \(s = 0\) of \eqref{eq:recursion2} gives the desired inequality \eqref{equ:kappa-final}.
The proof of \eqref{equ:kappa-final-2} proceeds analogously.  
\end{proof}

Using the above lemma, we can readily obtain the following upper bounds for \( P_{\mathrm{red},\ell}^{\mathrm{per}} \) and \( \overline{P}_{\mathrm{blue},C\ell}^{\mathrm{per}} \).
It is worth pointing out that the term of order \(\frac{\ell}{\sqrt{k}}\) is actually of order \(\frac{1}{D}\), thus serving as the first-order correction in these bounds.

\begin{theorem}\label{thm:perfect-prob}
\[
P_{\mathrm{red},\ell}^{\mathrm{per}} \leq \left( 
p - \left( \frac{e^{-c^2}}{2\pi} \right)^{\frac{3}{2}} \frac{1}{p^2} \cdot \frac{\ell}{3\sqrt{k}} 
+ O\left( \frac{1}{D^2} \right) 
\right)^{\binom{\ell}{2}},
\]
\[
\overline{P}_{\mathrm{blue}, C\ell}^{\mathrm{per}} \leq  \left( 
1 - p + \left( \frac{e^{-c^2}}{2\pi} \right)^{\frac{3}{2}} \frac{C}{(1-p)^2} \cdot \frac{\ell}{3\sqrt{k}} 
+ O\left( \frac{1}{D^2} \right) 
\right)^{\binom{C\ell}{2}}.
\]
\end{theorem}

\begin{proof}
Using Lemmas~\ref{lem:prob_of_perfect} and \ref{lem:Erper}, we obtain that for any \( 2 \le r \le \ell \),
\begin{align*}
P_{\mathrm{red},r}^{\mathrm{per}}&=\prod_{i=1}^{r-1}\kappa_{i}^{\mathrm{per}}  \leq \prod_{i=0}^{r-1}\left(p - \left(\frac{e^{-c^2}}{2\pi}\right)^{\frac{3}{2}} \frac{1}{p^2} \cdot \frac{i}{\sqrt{k}} + O\left(\frac{1}{D^2}\right)\right)^{r-1-i}\\
&=p^{\binom{r}{2}}\cdot \exp\left(-\left(\frac{e^{-c^2}}{2\pi}\right)^{\frac{3}{2}}\cdot \sum_{i=0}^{r-1}\frac{i(r-1-i)}{p^3\sqrt{k}}+O\left(\frac{r^2}{D^2}\right)\right)\\
&=p^{\binom{r}{2}}\cdot \exp\left(-\left(\frac{e^{-c^2}}{2\pi}\right)^{\frac{3}{2}}\cdot\frac{1}{p^3\sqrt{k}} \cdot\binom{r}{3}+O\left(\frac{r^2}{D^2}\right)\right)\\
&=\left(p-\left(\frac{e^{-c^2}}{2\pi}\right)^{\frac{3}{2}} \frac{1}{p^2}\cdot \frac{r-2}{3\sqrt{k}}+O\left(\frac{1}{D^2}\right)\right)^{\binom{r}{2}}.
\end{align*}
It is clear that setting \( r = \ell \) in the above inequality gives the desired bound for \( P_{\mathrm{red},\ell}^{\mathrm{per}} \).
By the same reasoning, the desired bound for \( \overline{P}_{\mathrm{blue}, C\ell}^{\mathrm{per}} \) follows.
\end{proof}

\subsection{Proof Completion}
Before concluding the proof of Theorem~\ref{thm2}, we state a simple proposition (proved in Appendix \ref{app:p+1-p}).
Recall \(c=c_{k,p}\) from \eqref{equ:c_pk}.
Define \(c_0=c_{k,p_C}\) and  \( f(x):=\frac{1}{x^2}-\frac{C}{(1-x)^2} \).
Note that $f(p_C)>0$ by \eqref{equ:p^2logp}.

\begin{proposition}\label{prop:bound=p+1-p}
Let \(D\gg C\). Then there exists \(p\in \bigl(p_C,p_C+1/(p_C^2\cdot D)\bigl)\subseteq \bigl(p_C,1/2\bigr)\) such that
\begin{equation}\label{eq:bound_Fp}
p-\left(\frac{e^{-c^2}}{2\pi}\right)^{\frac{3}{2}}\frac{1}{3Dp^2}\leq p_C-\left(\frac{e^{-c_0^2}}{2\pi}\right)^{\frac{3}{2}}\frac{f(p_C)}{9D},
\end{equation}
\begin{equation}\label{equ:bound_1-p}
1-p+\left(\frac{e^{-c^2}}{2\pi}\right)^{\frac{3}{2}}\frac{C}{3D(1-p)^2}\leq 1-p_C-\left(\frac{e^{-c_0^2}}{2\pi}\right)^{\frac{3}{2}}\frac{f(p_C)}{9D}.
\end{equation}    
\end{proposition}

\begin{proof}[\bf Proof of Theorem \ref{thm2}.]
Fix any constant \( C > 1 \).  
Let \( D = D(C) \) and \( \ell_0 = \ell_0(C) \) be constants satisfying \( \ell_0 \gg D \gg C \).  
Assume \( \ell \ge \ell_0 \) and set \( k = D^2 \ell^2 \).  
Let \( p \in \bigl(p_C, 1/2 \bigr) \) be as given in Proposition~\ref{prop:bound=p+1-p}. 
We consider the random sphere graph \( G_{k,p}(n) \).

By Lemma~\ref{prop:convergence_of_cpk}, we have \(c_0=\Phi^{-1}(1-p_C)+O(1/k)\) and thus
\begin{equation}\label{eq:phi_inverse_Phi}
\phi(\Phi^{-1}(p_C))= \frac{1}{\sqrt{2\pi}}\cdot e^{-\frac{c_0^2}{2}} + O\left(\frac{1}{k}\right)\leq \frac{1.1}{\sqrt{2\pi}}\cdot e^{-\frac{c_0^2}{2}}.
\end{equation}
Recall that $\frac{\sqrt{k}}{\ell}=D\gg C$ is sufficiently large. 
Applying Theorems~\ref{thm:perfect_nonperfect2} and \ref{thm:perfect-prob}, we have
\begin{align*}
P_{\mathrm{red},\ell}&\leq \left(1+2^{-C\ell}\right)P_{\mathrm{red},\ell}^{\mathrm{per}}\leq \left(1+2^{-C\ell}\right)\cdot 
\left( p - \left( \frac{e^{-c^2}}{2\pi} \right)^{\frac{3}{2}} \frac{1}{3Dp^2}  + O\left( \frac{1}{D^2} \right) 
\right)^{\binom{\ell}{2}}\\
& \leq \left(p_C-\left(\frac{e^{-c_0^2}}{2\pi}\right)^{\frac{3}{2}} \frac{f(p_C)}{9D} + O\left(\frac{1}{D^2}\right)\right)^{\binom{\ell}{2}}
\leq\left(p_C - \phi(\Phi^{-1}(p_C))^3\cdot \frac{f(p_C)}{18D}\right)^{\binom{\ell}{2}},
\end{align*}
where the third inequality follows from \eqref{eq:bound_Fp} and the fact that $2^{-C\ell}\ll 1/D^2$, and the last inequality holds by \eqref{eq:phi_inverse_Phi} and the fact that \( D \gg C \).
Similarly, we can derive 
\begin{align*}
\overline{P}_{\mathrm{blue},C\ell}&\leq \left(1+2^{-C\ell}\right)\overline{P}_{\mathrm{blue},C\ell}^{\mathrm{per}}\leq \left(1+2^{-C\ell}\right)\cdot \left( 
1 - p + \left( \frac{e^{-c^2}}{2\pi} \right)^{\frac{3}{2}} \frac{C}{3D(1-p)^2}
+ O\left( \frac{1}{D^2} \right) 
\right)^{\binom{C\ell}{2}}\\
&\leq \left(1-p_C - \left(\frac{e^{-c_0^2}}{2\pi}\right)^{\frac{3}{2}}\frac{f(p_C)}{9D}+O\left(\frac{1}{D^2}\right)\right)^{\binom{C\ell}{2}}
\leq\left(1-p_C - \phi(\Phi^{-1}(p_C))^3\cdot\frac{f(p_C)}{18D}\right)^{\binom{C\ell}{2}}.
\end{align*}
Hence, setting the positive constant
\begin{equation}\label{equ:epsilon_0}
\varepsilon_0 = \varepsilon_0(C) := \frac{1}{18} \, \phi\bigl(\Phi^{-1}(p_C)\bigr)^3 \cdot f(p_C).
\end{equation}
completes the proof of Theorem~\ref{thm2}.
\end{proof}

Let us comment on the value of \(\varepsilon(C)\) as given in Theorem~\ref{thm1}. 
A key related parameter is the constant \(D(C)\). 
To simplify the presentation and calculations in our proofs, 
we take \(D(C)\) sufficiently large relative to \(C\). 
However, a more careful analysis shows that it suffices to take
\[
D(C) = 10^5 \cdot \alpha_C^4 \cdot p_C^{-3} \cdot \bigl(f(p_C)\bigr)^{-1},
\]
where \(\alpha_C = \max\left\{1000,\, 20\sqrt{C \log(10/p_C)}\right\}\) is as defined in Lemma~\ref{Volumenonperfect}.
Combining \eqref{equ:ramsey-n} and~\eqref{equ:epsilon_0}, we can choose the parameter \(\varepsilon(C)\) in Theorem~\ref{thm1} as
\[
\varepsilon(C) = \frac{\varepsilon_0(C) \cdot p_C^{-3/2}}{6D(C)} 
= \frac{p_C^{-3/2} \cdot \phi(\Phi^{-1}(p_C))^3 \cdot f(p_C)}{108 \cdot D(C)} 
= \frac{\alpha_C^{-4} \cdot p_C^{3/2}}{1.08 \times 10^7} \cdot \phi(\Phi^{-1}(p_C))^3 \cdot f(p_C)^2.
\]
In particular, as \(C \to 1^+\), we have \(\varepsilon(C) = \Omega\bigl((C - 1)^2\bigr)\), whereas as \(C \to \infty\),
\(
\varepsilon(C) = \Omega\left(\frac{\operatorname{poly}(\log C)}{C^{5/2}}\right).
\)

\medskip

Based on the above expression of \(\varepsilon(C)\) as \(C \to 1^+\), 
we can derive an improvement for the Ramsey number \(r(\ell,k)\) over Erd\H{o}s’s probabilistic bound in the almost diagonal regime \(k = \ell + o(\ell)\); 
the explicit formula for this improvement is given in \eqref{equ:almost-diag} following Corollary~\ref{coro}. 
We believe that, with further development of this random model, nontrivial bounds may still be attainable for the diagonal Ramsey number. 
Finally, we note that our result remains fully compatible with the Lovász Local Lemma.

\section*{Acknowledgments}
The authors are grateful to B\'ela Bollob\'as, David Conlon, Jacob Fox, Simon Griffiths, Dhruv Mubayi, and Huy Tuan Pham for sharing their insights on random geometric graphs and providing relevant references, and to Boris Bukh, Marcelo Campos, Jacob Fox, Sam Mattheus, Sergey Norin, and Benny Sudakov for helpful comments and discussions that improved the presentation. They also thank Peter Frankl for drawing attention to explicit constructions of Ramsey numbers. The second author is particularly indebted to his advisor, Professor Shing-Tung Yau, for his unwavering support. 
This research was supported by the National Key Research and Development Program of China 2023YFA1010201, the National Natural Science Foundation of China 12125106, and 
Quantum Science and Technology-National Science and Technology Major Project 2021ZD0302902.

\section*{Note added in proof.}
We learned that Hunter, Milojević, and Sudakov \cite{HMS} obtained further developments that simplify and extend our main results.

\vspace{-0.2cm}

\appendix

\section*{Appendix}

\section{The Lower Bound from Erd\H{o}s’s Probabilistic Method}\label{app:Erd-pro}
Here, we present a proof that not only establishes the lower bound in \eqref{equ:original_bound} derived from Erd\H{o}s’s first moment method but also demonstrates its optimality.
Let $C\geq 1$ be a fixed constant.

For \( p \in (0, 1/2] \) and \( n \in \mathbb{N} \), consider the probability that a random edge-coloring of \( K_n \), where each edge is independently colored red with probability \( p \) and blue with probability \( 1-p \), contains either a red \( K_\ell \) or a blue \( K_{C\ell} \). This probability is clearly at most
\[
f(n, p) := A(n, p) + B(n, p), \quad \text{where } A(n, p) = \binom{n}{\ell} p^{\binom{\ell}{2}} \text{ and } B(n, p) = \binom{n}{C\ell} (1 - p)^{\binom{C\ell}{2}}.
\]
Hence, if \( f(n, p) \leq 0.99 \), there exists at least one such coloring with no monochromatic clique, implying \( r(\ell, C\ell) > n \). It thus suffices to find the maximum value of \( n = n(p) \) such that \( f(n, p) = 0.99 \). 
Assume this maximum is achieved at \( p = p_{C,\ell} \). Then, 
\[
\frac{\partial f(n,p_{C,\ell})}{\partial p}=\frac{\binom{\ell}{2}}{p_{C,\ell}}\binom{n}{\ell} p_{C,\ell}^{\binom{\ell}{2}}-\frac{\binom{C\ell}{2}}{1-p_{C,\ell}}\binom{n}{C\ell} (1-p_{C,\ell})^{\binom{C\ell}{2}}=0,
\]
implying that $\log A(n,p_{C,\ell})=\log B(n,p_{C,\ell})+O(\log \ell)$.
Solving this along with \( A(n,p_{C,\ell}) + B(n,p_{C,\ell}) = 0.99 \), 
we obtain that \( \log A(n,p_{C,\ell}) = O(\log \ell) \) and \( \log B(n,p_{C,\ell})=O(\log \ell) \), from which it follows
\[
-\log p_{C,\ell}=\frac{2\log (en/\ell)}{\ell-1}+O\left(\frac{\log \ell}{\ell^2}\right) \mbox{~~and~} -\log (1-p_{C,\ell})=\frac{2\log (en/C\ell)}{C\ell-1}+O\left(\frac{\log \ell}{\ell^2}\right). 
\]
We then derive that \( p_{C,\ell}=p_C+O\left(1/\ell\right)\), where the constant $p_C$ satisfies \( C = \frac{\log p_C}{\log(1 - p_C)} \).
It follows directly from the above that 
\( n = \frac{\ell}{e} \cdot p_{C,\ell}^{-(\ell-1)/2} 
\cdot e^{O\left(\frac{\log \ell}{\ell}\right)} 
= \Theta(\ell)\cdot \left(p_C + O\left(\frac{1}{\ell}\right)\right)^{-\ell/2} 
= \Theta\bigl(\ell \cdot M_C^{\ell}\bigr),\)
where \( M_C := p_C^{-1/2} \).
This establishes \( r(\ell, C\ell) = \Omega\bigl(\ell \cdot M_C^{\ell}\bigr) \). 
\qed

\bigskip

We remark that, with some additional work extending the proof above, the optimal lower bound (with the best leading constant) obtainable via Erdős’s probabilistic method can be shown to be 
\[
r(\ell, C\ell) \ge \bigl(\beta_C/e + o(1)\bigr)\cdot \ell \cdot M_C^{\ell - 1},
\]
where \(\beta_C := \exp\left(\bigl(\frac{C - 1}{2}\cdot\log(1-p_C) -\log C\bigr)\cdot(1 - p_C)\cdot\log(1 - p_C)/H(p_C)\right)\), 
and \(H(x) := -x \log x - (1 - x) \log(1 - x)\) denotes the entropy function. 

\section{Proof that \(u = |\pi_{X'}(\By)|^2\) follows a beta distribution}\label{app:beta}
In this appendix, we prove the following statement needed in the proof of Lemma~\ref{Volumenonperfect}: 
For any integer \(r\) with \(1 \leq r \leq k\),
let \(X'\) be a subspace of $\mathbb{R}^{k+1}$ of dimension \(r\). 
Let \(\By\) be uniformly distributed on \(S^k\).
Then $u:= |\pi_{X'}(\By)|^2$ follows the beta distribution \(\operatorname{Beta}\!\left(\frac{r}{2},\; \frac{k - r + 1}{2}\right)\).

\begin{proof}
Let \(\Bx = (x_1,\ldots,x_{k+1})\) follow the standard normal distribution on \(\mathbb{R}^{k+1}\). 
Define \(z = x_1^2 + \cdots + x_{r}^2\) and \(w = x_{r+1}^2 + \cdots + x_k^2\).
Then \(z \sim \chi^2(r)\) and \(w \sim \chi^2(k - r + 1)\) are independent, where \(\chi^2(d)\) denotes the chi-squared distribution with \(d\) degrees of freedom.
Since \(\By\) is uniformly distributed on \(S^k\), it has the same distribution as \(\frac{\Bx}{|\Bx|}\).
Therefore, for a subspace \(X'\) of dimension \(r\), we have
\[
u = |\pi_{X'}(\By)|^2 \sim \frac{z}{z + w}.
\]
Recall that the joint probability density function (PDF) of \(z\) and \(w\) is given by
\[
f(z,w) = f(z)f(w) = \left(\frac{\left(\frac{z}{2}\right)^{\frac{r}{2}-1}e^{-\frac{z}{2}}}{2\Gamma\bigl(\frac{r}{2}\bigr)}\right)\cdot \left(\frac{\left(\frac{w}{2}\right)^{\frac{k-r+1}{2}-1}e^{-\frac{w}{2}}}{2\Gamma\bigl(\frac{k-r+1}{2}\bigr)}\right).
\]
Set \(v = z + w\). Then \(z = uv\) and \(w = v(1-u)\). The joint PDF of \(u\) and \(v\) satisfies
\[
\begin{aligned}
f(u,v) &= \left|\frac{\partial(z,w)}{\partial(u,v)}\right| f(z,w)
=v\cdot  \left(\frac{\left(\frac{z}{2}\right)^{\frac{r}{2}-1}e^{-\frac{z}{2}}}{2\Gamma\bigl(\frac{r}{2}\bigr)}\right)\cdot \left(\frac{\left(\frac{w}{2}\right)^{\frac{k-r+1}{2}-1}e^{-\frac{w}{2}}}{2\Gamma\bigl(\frac{k-r+1}{2}\bigr)}\right)\\
&= \frac{v/2}{2\Gamma\bigl(\frac{r}{2}\bigr)\Gamma\bigl(\frac{k-r+1}{2}\bigr)}\cdot \left(\frac{uv}{2}\right)^{\frac{r}{2}-1}e^{-\frac{uv}{2}}\cdot \left(\frac{v(1-u)}{2}\right)^{\frac{k-r+1}{2}-1}e^{-\frac{v(1-u)}{2}}\\
&=\left(\frac{\Gamma\bigl(\frac{k+1}{2}\bigr)}{\Gamma\bigl(\frac{r}{2}\bigr)\Gamma\bigl(\frac{k-r+1}{2}\bigr)} u^{\frac{r}{2}-1}(1-u)^{\frac{k-r+1}{2}-1}\right)\cdot \left(\frac{\left(\frac{v}{2}\right)^{\frac{k+1}{2}-1}e^{-\frac{v}{2}}}{2\Gamma\bigl(\frac{k+1}{2}\bigr)}\right).
\end{aligned}
\]
This shows that \(u\) and \(v\) are independent, and that \(u\) follows the beta distribution
\(
\operatorname{Beta}\!\left(\frac{r}{2},\; \frac{k - r + 1}{2}\right).
\)
\end{proof}

In particular, for the case \(r = C\ell\) needed in the proof of Lemma~\ref{Volumenonperfect}, we have
\[
u = |\pi_{X'}(\By)|^2 \sim \operatorname{Beta}\!\left(\frac{C\ell}{2},\; \frac{k - C\ell + 1}{2}\right).
\]

\section{Proof of Proposition~\ref{prop:bound=p+1-p}}\label{app:p+1-p}

In this proof, besides the function \(f(x) = \frac{1}{x^2} - \frac{C}{(1 - x)^2}\), we define
\( g(x) = \frac{e^{-c_{k,x}^2/2}}{\sqrt{2\pi}}\), 
\( h(x) = \frac{1}{x^2},\)
and
\[
F(x) := x - \frac{1}{3D}\, g(x)^3\, h(x) 
= x - \left(\frac{e^{-c_{k,x}^2}}{2\pi}\right)^{3/2} \frac{1}{3D\, x^2}
\]
where \( x\in (0,1/2)\) and \(c_{k,x}\) is taken from \eqref{equ:c_pk}.
Since \(D\gg C\), and \(f(x),g(x)\) and \(h(x)\) are continuous on \((0,1/2)\) and are all positive at the point \(x=p_C\), we have
\(
0.99<\frac{f(x)}{f(p_C)},\, \frac{g(x)}{g(p_C)},\, \frac{h(x)}{h(p_C)}< 1.01
\)
for any \(x\in(p_C,p_C+1/(p_C^2\cdot D))\).
Let $c_0=c_{k,p_C}$, and define
\[
p_1=p_C+\left(\frac{e^{-c_{0}^2}}{2\pi}\right)^{\frac{3}{2}}\frac{1}{3D}\left(\frac{0.5}{p^2_C}-\frac{1}{3}f(p_C)\right)
\quad \mbox{ and } \quad 
p_2=p_C+\left(\frac{e^{-c_0^2}}{2\pi}\right)^{\frac{3}{2}}\frac{1}{3D}\left(\frac{1.5}{p^2_C}-\frac{1}{3}f(p_C)\right).
\]
Since \(0 < f(p_C) < h(p_C)= \frac{1}{p_C^2}\), it is easy to verify that \(p_C < p_1 < p_2 < p_C+1/(p_C^2\cdot D)\).
Then 
\begin{align*}
F(p_1)&= p_C+\left(\frac{e^{-c_0^2}}{2\pi}\right)^{\frac{3}{2}}\frac{1}{3D}\left(\frac{0.5}{p^2_C}-\frac{1}{3}f(p_C) \right)- \frac{1}{3D}(g(p_1))^3h(p_1)\\
&\leq p_C + \left(\frac{e^{-c_0^2}}{2\pi}\right)^{\frac{3}{2}}\frac{1}{3D}
\left(\frac{0.5}{p^2_C}-\frac{1}{3}f(p_C)-\frac{0.99^4}{p_C^2} \right)< p_C-\left(\frac{e^{-c_0^2}}{2\pi}\right)^{\frac{3}{2}}\frac{1}{9D}f(p_C);\\
F(p_2)&= p_C+\left(\frac{e^{-c_0^2}}{2\pi}\right)^{\frac{3}{2}}\frac{1}{3D}\left(\frac{1.5}{p^2_C}-\frac{1}{3}f(p_C) \right)- \frac{1}{3D}(g(p_2))^3h(p_2)\\
&\geq p_C+ \left(\frac{e^{-c_0^2}}{2\pi}\right)^{\frac{3}{2}}\frac{1}{3D}
\left(\frac{1.5}{p^2_C}-\frac{1}{3}f(p_C)-\frac{1.01^{4}}{p_C^2} \right)> p_C-\left(\frac{e^{-c_0^2}}{2\pi}\right)^{\frac{3}{2}}\frac{1}{9D}f(p_C).
\end{align*}
By the intermediate value theorem, there exists some \(p \in (p_1, p_2) \subseteq \bigl(p_C,\, p_C + \frac{1}{p_C^2 D}\bigr)\) such that
\begin{equation*}
F(p)=p-\left(\frac{e^{-c_{k,p}^2}}{2\pi}\right)^{\frac{3}{2}} \frac{1}{3Dp^2} = p_C-\left(\frac{e^{-c_0^2}}{2\pi}\right)^{\frac{3}{2}}\frac{1}{9D}f(p_C),
\end{equation*}
which implies \eqref{eq:bound_Fp}.
Using this equality, we obtain 
\begin{equation*}
\begin{aligned}
&\quad\ 1-p+\left(\frac{e^{-c_{k,p}^2}}{2\pi}\right)^{\frac{3}{2}}\frac{C}{3D(1-p)^2}
=1-\left(p-\left(\frac{e^{-c_{k,p}^2}}{2\pi}\right)^{\frac{3}{2}}\frac{1}{3Dp^2}\right)-\frac{1}{3D} (g(p))^3f(p)\\
&\leq 1-p_C + \left(\frac{e^{-c_0^2}}{2\pi}\right)^{\frac{3}{2}}\frac{1}{9D}f(p_C) - 0.99^4 (g(p_C))^3\frac{1}{3D}f(p_C)
\leq 1-p_C-\left(\frac{e^{-c_0^2}}{2\pi}\right)^{\frac{3}{2}}\frac{1}{9D}f(p_C),
\end{aligned}
\end{equation*}
which implies \eqref{equ:bound_1-p}. 
This completes the proof of Proposition~\ref{prop:bound=p+1-p}. \qed

\vspace{1.0cm}

\indent{\it Email address}: \texttt{jiema@ustc.edu.cn}
\vspace{0.3cm}

\indent{\it Email address}: \texttt{shenwj22@mails.tsinghua.edu.cn}
\vspace{0.3cm}

\indent{\it Email address}: \texttt{jeff\_776532@mail.ustc.edu.cn}

\end{document}